\documentclass[11pt, a4paper]{amsart}
\usepackage{amsmath,amssymb,amsthm,mathtools,wasysym,calc,verbatim,tikz,url,hyperref,mathrsfs,cite,fullpage,bbm}
\usepackage{thm-restate}
\usepackage[noabbrev,capitalize,nameinlink]{cleveref}
\usepackage[shortlabels]{enumitem}

\DeclareMathOperator{\cof}{\cof}

\theoremstyle{plain}

\usepackage{amsmath, amsthm, amssymb}
\usepackage{hyperref}
\usepackage{mathrsfs}
\usepackage{asymptote}
\usepackage{graphicx}
\usepackage{theoremref}

\theoremstyle{plain}
\newtheorem{theorem}{Theorem}
\newtheorem{lemma}{Lemma}
\newtheorem{proposition}{Proposition}
\newtheorem{corollary}{Corollary}

\theoremstyle{definition}
\newtheorem{definition}{Definition}
\theoremstyle{remark}
\newtheorem{remark}{Remark}

\numberwithin{equation}{section}
\numberwithin{lemma}{section}
\numberwithin{proposition}{section}
\numberwithin{corollary}{section}
\numberwithin{remark}{section}
\numberwithin{definition}{section}
\numberwithin{claim}{section}

\usepackage[margin=1.0in]{geometry}

\begin{document}
\title[Improved quantitative equidistribution of nilsequences]{Efficient equidistribution of periodic nilsequences and applications}
\author{James Leng}
\address{Department of Mathematics, UCLA, Los Angeles, CA 90095, USA}
\email{jamesleng@math.ucla.edu}

\maketitle

\begin{abstract}
This is a companion paper to \cite{Len2}. We deduce an equidistribution theorem for periodic nilsequences and use this theorem to give two applications in arithmetic combinatorics. The first application is quasi-polynomial bounds for a certain complexity one polynomial progression, improving the iterated logarithm bound previusly obtained. The second application is a proof of the quasi-polynomial $U^4[N]$ inverse theorem. In work with Sah and Sawhney, we obtain improved bounds for sets lacking nontrivial $5$-term arithmetic progressions.
\end{abstract}

\section{Introduction}
In \cite{Len2}, the author gave a proof of improved bounds for the equidistribution of nilsequences. The author found that the Ratner-type factorization theorem was inefficient for quantitative higher order Fourier analysis. In attempting to salvage that theorem, the author proved an equidistribution theorem for a $G_{(s)}$-vertical character. An informal version of that equidistribution theorem is as follows (the precise statement can be found in \cite[Theorem 3]{Len2}).
\begin{theorem}[Informal Equidistribution Theorem]
    Let $G/\Gamma$ is a nilmanifold and $F(g(n)\Gamma)$ is a nilsequence with $F$ a $G_{(s)}$-vertical character with nonzero frequency $\xi$. Suppose 
    $$\left|\mathbb{E}_{n \in [N]} F(g(n)\Gamma)\right| \ge \delta.$$
    Then $F$ is ``morally" a step $\le s - 1$ nilsequence with ``good bounds."
\end{theorem}
Here, ``good bounds" denotes bounds whose exponents are polynomial in the dimension of the nilmanifold. Since nilsequences may be Fourier expanded into nilcharacters via \cite[Lemma A.7]{Len2}, it follows that such a theorem can recover an equidistribution theorem for all nilsequences with good bounds. In this article, we shall consider the analogue of this equidistribution theorem for periodic nilsequences along with applications of that theorem.
\subsection{Main Results}
Our first main result is the following equidistribution theorem with better bounds (see Section 2 and \cite[Section 2]{Len2} for undefined notions).
\begin{restatable}{theorem}{mainresulta}\thlabel{mainresult1}
Let $\delta \in (0, 1/10)$, $N > 100$ prime, and $G/\Gamma$ be a nilmanifold of step $s$, degree $k$, dimension $d$, and complexity at most $M$. Furthermore, let $F(g(n)\Gamma)$ be a periodic nilsequence modulo $N$ with $F$ a $G_{(s)}$ vertical character with frequency $\xi$ with $|\xi| \le M/\delta$. Suppose
$$|\mathbb{E}_{n \in \mathbb{Z}/N\mathbb{Z}} F(g(n)\Gamma)| \ge \delta.$$
Then either $N \ll (M/\delta)^{O_k(d^{O_k(1)})}$ or there exists horizontal characters $\eta_1, \dots, \eta_r$ of size at most $(M/\delta)^{O_k(d^{O_k(1)})}$ with $1 \le r \le \mathrm{dim}(G/[G, G])$ such that
$$\|\eta_i \circ g\|_{C^\infty[N]} = 0$$
and such that any $s$ elements $w_1, \dots, w_s \in G' := \bigcap_{i = 1}^r \mathrm{ker}(\eta_i)$ satisfies
$$\xi([[[w_1, w_2], w_3],\dots, w_s]) = 0.$$
\end{restatable}
\begin{remark}
This theorem is proven by black-boxing \cite[Theorem 3]{Len2} (via an argument due to Candela and Sisask \cite{CS14}; see also, \cite{Len22, Kuc21, Kuc23}). However one main motivation of the article is to dicuss the proof in 3 simpler cases (with self contained proofs) which were crucial in developing the proof in \cite{Len2}.
\end{remark}
Our second main result is the following.
\begin{theorem}\thlabel{mainresult2}
Let $N$ be a large prime, $P, Q \in \mathbb{Z}[x]$ be linearly independent with $P(0) = Q(0) = 0$, and let $A \subseteq \mathbb{Z}_N$ be a subset such that $A$ contains no configuration of the form $(x, x + P(y), x + Q(y), x + P(y) + Q(y))$ with $y \neq 0$. Then
$$|A| \ll_{P, Q} N\exp(-\log^{c_{P, Q}}(N)).$$
\end{theorem}
This will be proven using a similar method as in \cite{Len22}, with input from \thref{mainresult1} rather than \cite{GT12}. Our final result is the quasi-polynomial $U^4(\mathbb{Z}/N\mathbb{Z})$ inverse theorem.
\begin{restatable}{theorem}{mainresultc}\thlabel{mainresult4}
Suppose $f\colon \mathbb{Z}/N\mathbb{Z} \to \mathbb{C}$ is a one-bounded function with
$$\|f\|_{U^4(\mathbb{Z}/N\mathbb{Z})} \ge \delta.$$
Then there exists a nilmanifold $G/\Gamma$ with degree at most $3$, dimension $\log(1/\delta)^{O(1)}$, complexity $O(1)$, and a nilsequence $F(g(n)\Gamma)$ on the nilmanifold with $F$ $1$-Lipschitz such that
$$|\mathbb{E}_{n \in [N]} f(n) F(g(n)\Gamma)|  \ge \exp(-\log(1/\delta)^{O(1)}).$$
\end{restatable}
Since the initial version of this paper, this result and work with Sah and Sawhney on improved bounds for $5$-term arithmetic progressions \cite{LSS24} have been generalized in work with Sah and Sawhney \cite{LSS24b} to the quasi-polynomial $U^{s + 1}[N]$ inverse theorem and sets lacking nontrivial $k$-term arithmetic progressions \cite{LSS24c}. The article can thus be treated as a ``stepping-stone" for the more general \cite{Len2} and \cite{LSS24b}.

\subsection{Discussion on the proof of the $U^4(\mathbb{Z}/N\mathbb{Z})$ inverse theorem}
\thref{mainresult4} will be proven using a very similar method as in \cite{GTZ11}, with input from \thref{twostepcase} instead of \cite{GT12} which simplifies the argument somewhat. The advantage of the approach of \cite{GTZ11} is that it avoids the use of ``$1\%$ quadratics" that Gowers and Milicevic \cite{GM17} and Manners \cite{Man18} consider. We instead consider a ``$1\%$ linear" equation, which is substantially simpler. \\\\
Our primary improvement over \cite{GTZ11} is the ``sunflower" and ``linearization" steps which correspond to Sections 7 and 8 of that article. In \cite{GTZ11}, the authors prove these steps by invoking the \emph{Ratner-type factorization theorem}, which given a nilmanifold of dimension $d$ involves $O(d)$ iterations of the Leibman-type theorem of Green and Tao \cite[Theorem 2.9]{GT12}. The sunflower and linearization steps are also proven iteratively, with the number of iterations being $O(d)$. Altogether, their proof involves $O(d^2)$ many iterations of the quantitative Leibman theorem of \cite{GT12}. We are able to do each of the ``sunflower" and ``linearization" steps in one single application of \thref{twostepcase}. The proofs given here are refinements of the proofs given in \cite{GTZ11}. For instance, the proof of the sunflower step (\thref{sunflower}) relies on the Furstenberg-Weiss commutator argument used in \cite{GTZ11} (but adapted to the setting of \cite[Theorem 8]{Len2}) and the proof of linearization step (\thref{linearizationstep}) is similar to that of \cite{GTZ11} in that it also relies on the Balog-Szemer\'edi-Gowers theorem and a Freiman type theorem, though we use the refinement of the theory due to Sanders \cite{San12b}.
\subsection{Organization of the paper}
In Section 2, we will define notation we need in addition to the notation in \cite{Len2}. Sections 3-5 will delve into examples and simple cases of the proof of \thref{mainresult1}, in order to motivate the proof of \cite[Theorem 3]{Len2}. In Sections 3, we will prove an equidistribution theorem for degree two periodic bracket polynomials. In Section 4, we will prove the degree two periodic case of \thref{mainresult1}. In Section 5, we will prove the full two-step case of \thref{mainresult1}. We will prove \thref{mainresult1} in Section 6 (by black-boxing \cite[Theorem 3]{Len2}). We will prove \thref{mainresult2} in Section 7, and \thref{mainresult4} in Section 8.\\\\
In Appendix A, we will deduce some auxiliary lemmas, used in Sections 3 and 4. In Appendix B, we will collect and prove some lemmas on bracket polynomials; these will be useful in Sections 3 and 8. Finally, in Appendix C, we will include a different proof than the one suggested in Section 3 of a key lemma of the equidistribution theory: the \emph{refined bracket polynomial lemma}. We believe this proof is somewhat more intuitive and offers more motivation for the statement of the lemma than the proof presented in \cite{Len2}, which is cleaner. 
\subsection{Acknowledgements}
We would like to thank Terry Tao for advisement and for suggesting a simpler proof of the refined bracket polynomial lemma than the proof the author initially came up with, which is present in Appendix C. We would also like to thank Ashwin Sah and Mehtaab Sawhney for their interest in the author's work and being helpful and extremely supportive of the author. In particular, the persistent questions of Ashwin Sah and Mehtaab Sawhney in the author's work led the author to realize a mistake in Section 8 in an earlier version of the document. When the author presented a fix to them, they pointed out that there is a more efficient way to ``Pigeonhole" in the proof of \thref{sunflower} which lead to the quasi-polynomial bounds in \thref{mainresult4} instead of bounds of the shape of $\exp(-\exp(O(\log\log(1/\delta)^2)))$ that the author had in a previous version. We are immensely grateful to them for communicating this point with the author and allowing the author to write up their argument here. We would in addition like to thank Ben Green and Sarah Peluse for helpful discussions. The author is supported by an NSF Graduate Research Fellowship Grant No. DGE-2034835.

\section{Notation}
We shall use the notation in \cite{Len2}, with a few differences, which we shall describe below. 
\begin{definition}[Periodic nilsequences]
Given a nilmanifold $G/\Gamma$ and an integer $N$, a polynomial sequence $g \in \text{poly}(\mathbb{Z}, G)$, and a Lipschitz function $F: G/\Gamma \to \mathbb{C}$, we say $F(g(n)\Gamma)$ is a \emph{periodic nilsequence modulo $N$} if $g(n + N) \Gamma = g(n)\Gamma$ for each $n \in \mathbb{Z}$.
\end{definition}
\begin{definition}[Smoothness norms]
Given a polynomial $p\colon \mathbb{Z} \to \mathbb{R}$ with $p(n) = \sum_{i = 0}^d \alpha_i n^i$, we write
$$\|p\|_{C^\infty[N]} = \sup_{1 \le i \le d} N^i\|\alpha_i\|_{\mathbb{R}/\mathbb{Z}}.$$
\end{definition}
The reason we work with the above definition of $C^\infty$ norm rather than the definition in \cite{Len2} involving binomial coefficients is that this definition is better adapted to polynomials $p\colon \mathbb{Z}/N\mathbb{Z} \to \mathbb{T}$, which appear in the analysis of periodic nilsequences. \\\\
We shall also use $c(\delta)$ as in \cite{Len2} as any quantity $\gg (\delta/M)^{O_{k, \ell}(d)^{O_{k, \ell}(1)}}$; here, since we will be working with single parameter nilsequences, we will always have $\ell = 1$.

\section{Bracket polynomial heuristics}
In this section, we shall deduce an equidistribution theorem for a periodic degree two bracket polynomial. Such a polynomial is of the form
$$e(\phi(n)) := e(-(\alpha_1 n[\beta_1 n] + \cdots + \alpha_d n[\beta_d n]) + P(n))$$
with $\phi(n + N) \equiv \phi(n) \pmod{1}$ and $\alpha_i, \beta_i$ are rational with denominator $N$. Such a function is \emph{not} a nilsequence but can be written as $F(g(n)\Gamma)$ where $F\colon G/\Gamma \to \mathbb{C}$ is only piecewise Lipschitz on a degree two nilmanifold $G/\Gamma$. To realize this, we require the following definitions.
\begin{definition}[Elementary two-step nilmanifold]
We shall define the \emph{elementary two-step nilmanifold} of dimension $2d + 1$. Define
$$G = \begin{pmatrix} 1 & \mathbb{R} & \cdots & \cdots  & \mathbb{R} \\ 0 &1 & \cdots & 0 & \mathbb{R} \\ 0 & 0 & 1 & \cdots & \mathbb{R} \\
0 & 0 & 0 & \ddots & \vdots \\
0 & 0& 0 & \cdots  & 1 \end{pmatrix}$$
and
$$\Gamma  = \begin{pmatrix} 1 & \mathbb{Z} & \cdots & \cdots  & \mathbb{Z} \\ 0 &1 & \cdots & 0 & \mathbb{Z} \\ 0 & 0 & 1 & \cdots & \mathbb{Z} \\
0 & 0 & 0 & \ddots & \vdots \\
0 & 0& 0 & \cdots  & 1 \end{pmatrix}$$
equipped with the lower central series filtration. We write this in coordinates as
$$(x_1, \dots, x_d, y_1, \dots, y_d, z).$$
We furthermore define the \emph{horizontal component} as
$$\psi_{\mathrm{horiz}}(x_1, \dots, x_d, y_1, \dots, y_d, z) := (\vec{x}, \vec{y})$$
and the \emph{vertical component} as $\psi_{\mathrm{vert}}(x_1, \dots, x_d, y_1, \dots, y_d, z) := z$.
\end{definition}
We also define an elementary bracket quadratic below.
\begin{definition}[Elementary bracket quadratic]
Consider $G/\Gamma$ an elementary two-step nilmanifold of dimension $2d + 1$. Given real numbers $\alpha_1, \dots, \alpha_d, \beta_1, \dots, \beta_d$, and $P$ a quadratic polynomial, we define the \emph{elementary polynomial sequence} associated to $(\vec{\alpha}, \vec{\beta}, P)$ via $g(n) = (\alpha_1n, \dots, \alpha_d n, \beta_1n \cdots, \beta_d n, P(n))$.
Note that $G/\Gamma$ has the fundamental domain of $(-1/2, 1/2]^{2d + 1}$ via the map 
$$(x_1, \dots, x_d, y_1, \dots, y_d, z) \mapsto (\{x_1\}, \dots, \{x_d\}, \{y_1\}, \dots, \{y_d\}, \{z - \vec{x} \cdot [\vec{y}]\}).$$
We can thus define the function $F: G/\Gamma \to \mathbb{C}$ as
$$F((x_1, \dots, x_d, y_1, \dots, y_d, z)\Gamma) = e(-\sum_{i = 1}^d x_i [y_i] + z).$$
We see that
$$F(g(n)\Gamma) = e(-\alpha_1 n[\beta_1 n] + \cdots + \alpha_d n[\beta_d n] + P(n));$$
we define an \emph{elementary bracket quadratic} associated to $(\vec{\alpha}, \vec{\beta}, P)$ as $F(g(n)\Gamma)$. We say that the two-step bracket quadratic is \emph{periodic} modulo $N$ if $g(n + N) \Gamma = g(n)$. 
\end{definition}
We next define the asymmetric bilinear form of an elementary two-step nilmanifold.
\begin{definition}[Bilinear form of an elementary two-step nilmanifold]
Given an elementary two-step nilmanifold $G/\Gamma$ of dimension $d$, we can define the associated \emph{asymmetric bilinear|} form $\omega\colon (\mathbb{R}^{2d})^2 \to \mathbb{R}$ as follows. If $u_1 = (\vec{x}, \vec{y}), u_2 ((\vec{z}, \vec{w}) \in (\mathbb{R}^{2d})^2$, we define
$$\omega(u_1, u_2) =  \vec{x} \cdot \vec{w} - \vec{y} \cdot \vec{z}.$$  
\end{definition}
Finally, we require the following lemma, whose proof we omit (see \thref{periodic}).
\begin{lemma}\thlabel{periodiclemma}
If $F(g(n)\Gamma)$ is an $N$-periodic elementary bracket quadratic, and if $g(n) = (\vec{\alpha}n, \vec{\beta}n, P)$ with $P = an^2 + bn + c$, then $\vec{\alpha}, \vec{\beta}$ are rational with denominator $N$, and $a$ is rational with denominator $2N$.
\end{lemma}
We are now ready to state the equidistribution theorem for degree two bracket polynomials.
\begin{theorem}\thlabel{twostepbracket}
Let $\delta \in (0, 1/10)$ and $N > 10$ a prime, and $G/\Gamma$ an elementary two-step nilmanifold of dimension $2d + 1$ with bilinear form $\omega$, and $F(g(n)\Gamma)$ an elementary two-step bracket quadratic. Suppose
$$|\mathbb{E}_{n \in [N]} F(g(n)\Gamma)| \ge \delta.$$
Then one of the following holds.
\begin{itemize}
    \item $N \ll \delta^{-O(d^{O(1)})}$;
    \item or there exists some $0 \le r \le 2d$ and $w_1, \dots, w_r$ and $\eta_1, \dots, \eta_{2d - r}$ all linearly independent vectors in $\mathbb{Z}^{2d}$ with $\langle w_i, \eta_j \rangle = 0$ and such that
    $$\|\eta_i \cdot \psi_{\mathrm{horiz}}(g)\|_{C^\infty[N]} = 0 \text{ and } \|\omega(w_i, \psi_{\mathrm{horiz}}(g))\|_{C^\infty[N]} = 0.$$
\end{itemize}
\end{theorem}
\begin{remark}
Observing that $\omega$ is nondegenerate on the first $2d$ coordinates, we see that $\omega$ is a \emph{symplectic form} on $\mathbb{R}^{2d}$. This theorem then states that $\psi_{\mathrm{horiz}}(g)$ lies on a \emph{isotropic subspace} $V$ of $\omega$; that is, a subspace $V$ with $\omega(V, V) = 0$. To see this, note that if $v, w \in V$, it suffices to show that $\omega(v, w) = 0$. Since $\eta_i$ annihilates $v$, it may be spanned by $w_j$'s, and because $\omega(w_j, w) = 0$, it follows that $\omega(v, w) = 0$. 
\end{remark}
\subsection{Proof of \thref{twostepbracket}}
Let
$$g(n) = (\alpha_1 n, \dots, \alpha_d n, \beta_1 n, \dots, \beta_d n, P(n)).$$
Since
$$\alpha_i n[\beta_i n] = (\{\alpha_i \} + [\alpha_i])n[(\{\beta_i\} + [\beta_i])n] = \{\alpha_i\}n[\{\beta_i n\}] + \{\alpha_i\}[\beta_i]n^2 \pmod{1},$$
we may reduce to the case where $|\alpha_i|, |\beta_i| \le \frac{1}{2}$ by modifying $P$. Applying the van der Corput inequality, there exists $\delta^{O(1)}N$ many $h \in [N]$ such that
\begin{equation}\label{vandercorputbracket}
|\mathbb{E}_{n \in [N]} e(\phi(n + h) - \phi(n))| \ge \delta^{O(1)}.    
\end{equation}

We next observe that if $\alpha$ and $\beta$ are arbitrary, then
$$\alpha (n + h)[\beta(n + h)] - \alpha n[\beta n] = \alpha n[\beta h] + \alpha h[\beta n] + [\text{Lower order terms}]$$
where ``lower order terms" denote a sum of $O(1)$ terms of the form $\{\alpha n\}\{\beta n\}, \{\alpha h\}\{\beta n\}$, $\{\alpha h\}\{\beta h\}$, $\{\alpha n\}\{\beta h\}$. The key observation is that each of the functions 
$$e(\{\alpha n\}\{\beta n\}), e(\{\alpha h\}\{\beta n\}), \text{ and } e(\{\alpha n\}\{\beta h\})$$
are functions on $\mathbb{T}^2$ in $(\{\alpha n\}, \{\beta n\})$, $(\{\alpha h\}, \{\beta n\})$, and $(\{\alpha n\}, \{\beta h\})$, respectively, and thus may be Fourier expanded into terms such as (say) $e(k_1\{\alpha n\} + k_2\{\beta n\}) = e(k_1 \alpha n + k_2 \beta n)$. Note that we have ``removed" a bracket. Such a heuristic is made precise in \thref{onevarfouriercomplexity} and \thref{bilinearfouriercomplexity}. \\\\
We may further analyze
$$\alpha n[\beta h] = \alpha n (\beta h - \{\beta h\}), \alpha h[\beta n] = \{\alpha h\}(\beta n - \{\beta n\})$$
and so
$$\alpha (n + h)[\beta(n + h)] - \alpha n[\beta n] = \alpha \beta nh + [(\alpha n, \beta n), (\{\alpha h\}, \{\beta h\})] + [\text{Lower order terms}]$$
where $[(x, y), (z, w)] = xw - yz$. Thus, letting $\beta$ denote the coefficient the quadratic term of the vertical component of $g$, we have
\begin{equation}\label{bracketcommutator}
 |\mathbb{E}_{n \in [N]} e(2\beta nh -n\omega(\psi_{\mathrm{horiz}}(g), \{\psi_{\mathrm{horiz}}(g)\}) + \text{[Lower order terms]})| \ge \delta^{O(1)}.
\end{equation}
By \thref{onevarfouriercomplexity}, \thref{bilinearfouriercomplexity}, we thus have for real $\gamma$ of denominator $N$ that
$$|\mathbb{E}_{n \in \mathbb{Z}_N} e(\gamma nh -n\omega(\psi_{\mathrm{horiz}}(g), \{\psi_{\mathrm{horiz}}(g)\}) + \gamma n)| \ge \delta^{O(d^{O(1)})}.$$
By \thref{periodiclemma} and letting $\alpha = \psi_{\mathrm{horiz}}(g)$ and $a$ equal to vector induced by $a \cdot y = \omega(\psi_{\mathrm{horiz}}(g), y)$, we have
\begin{equation}\label{refinedhypothesis}
\|a \cdot \{\alpha h\} + \beta + \gamma h\|_{\mathbb{R}/\mathbb{Z}} = 0.    
\end{equation}
for $\delta^{O(d^{O(1)})}N$ many $h \in [N]$.
In the next subsection, we prove a crucial Diophantine approximation result for this hypothesis.
\subsection{The refined bracket polynomial lemma}
We have the following lemma.
\begin{lemma}[Refined bracket polynomial lemma]\thlabel{periodicrefined}
    Let $\delta \in (0, 1/10)$ and $N > 100$ be prime. Suppose $a, \alpha \in \mathbb{R}^{d}$, $a$ and $\alpha$ are rational with denominator $N$, $|a| \le M$, and
    $$\|\beta + a \cdot \{\alpha h\}\|_{\mathbb{R}/\mathbb{Z}} = 0$$
    for $\delta N$ many $h \in [N]$. The either $N \ll (\delta/d^{d}M)^{-O(1)}$ or $K/N \ge 1/10$ or else there exists linearly independent $w_1, \dots, w_r$ and $\eta_1, \dots, \eta_{d - r}$ in $\mathbb{Z}^{d}$ with size at most $(\delta/d^{d} M)^{-O(1)}$ such that $\langle w_i, \eta_j \rangle = 0$,
    $$\|\eta_j \cdot \alpha\|_{\mathbb{R}/\mathbb{Z}} = 0, \text{ and } |w_i \cdot a| = \frac{(\delta/M)^{-O(1)}d^{O(d)}K}{N}.$$
\end{lemma}
\begin{remark}
The name ``refined bracket polynomial lemma" is derived from the analogous ``bracket polynomial lemma" of \cite[Proposition 5.3]{GT12}.
\end{remark}
This lemma can be proved via the following lemma by applying $K = 1/N^2$.
\begin{lemma}\thlabel{magicargument}
    Let $\delta \in (0, 1/10)$ and $N > 100$ be prime. Suppose $a, \alpha \in \mathbb{R}^{d}$ and $\alpha$ is rational with denominator $N$, $|a| \le M$, and
    $$\|\beta + a \cdot \{\alpha h\}\|_{\mathbb{R}/\mathbb{Z}} < K/N$$
    for $\delta N$ many $h \in [N]$. The either $N \ll (\delta/d^{d}M)^{-O(1)}$ or $K/N \ge 1/10$ or else there exists linearly independent $w_1, \dots, w_r$ and $\eta_1, \dots, \eta_{d - r}$ in $\mathbb{Z}^{d}$ with size at most $(\delta/d^{d} M)^{-O(1)}$ such that $\langle w_i, \eta_j \rangle = 0$,
    $$\|\eta_j \cdot \alpha\|_{\mathbb{R}/\mathbb{Z}} = 0, \text{ and } |w_i \cdot a| = \frac{(\delta/M)^{-O(1)}d^{O(d)}K}{N}.$$
\end{lemma}
The \thref{magicargument} is proved in \cite[Lemma 3.4]{Len2}. We will give another proof of this lemma with worse bounds in Appendix C. Although our bounds there are worse, they are sufficient for \thref{twostepbracket} and for any application of the refined bracket polynomial lemma in (the one-variable case of) \cite{Len2}. A corollary of this lemma is the following.
\begin{corollary}\thlabel{bracketpolynomialcorollary}
    Let $\delta \in (0, 1/10)$ and $N > 100$ be prime. Suppose $a, \alpha \in \mathbb{R}^d$ are rationals with $N$, $|a| \le M$, $\beta, \gamma \in \mathbb{R}$ with $\beta$ rational with denominator $N$, and
    $$\|\gamma + a \cdot \{\alpha h\} + \beta h\|_{\mathbb{R}/\mathbb{Z}} = 0$$
    for $\delta N$ many $h \in [N]$. Then either $N \ll (d\delta/M)^{-O(d)^{O(1)}}$ or there exists linearly independent $w_1, \dots, w_r$ and $\eta_1, \dots, \eta_{d - r}$ such that $\langle w_i, \eta_j \rangle = 0$ and
    $$\|\eta_j \cdot \alpha\|_{\mathbb{R}/\mathbb{Z}} = 0, \hspace{0.1in} \|w_i \cdot a\|_{\mathbb{R}/\mathbb{Z}} = 0.$$
\end{corollary}
\begin{proof}
    We define $\tilde{a} = (a, 1) \in \mathbb{R}^{d + 1}$ and $\tilde{\alpha} = (\alpha, \beta) \in \mathbb{R}^{d + 1}$. Invoking \thref{periodicrefined}, there exists $w_1, \dots, w_{r}$ and $\eta_1, \dots, \eta_{d + 1 - r}$ such that $|w_i(\tilde{a})| = 0$, $\|\eta_j(\tilde{\alpha})\|_{\mathbb{R}/\mathbb{Z}} = 0$. We denote $w_i = (u_i, v_i)$ and $\eta_j = (\mu_j, \nu_j)$ where $u_i \in \mathbb{R}^d$ and $\mu_j \in \mathbb{R}^d$ for each $i$ and $j$. Suppose $\nu_1 \neq 0$. Let $\tilde{\eta_j} = \nu_j \mu_1 - \mu_j \nu_1$. We see that $\|\tilde{\eta_j}(\mu)\|_{\mathbb{R}/\mathbb{Z}} = 0$. We claim that the $\tilde{\eta_j}$'s are independent of each other. Suppose there exists some $a_i$ such that
    $$\sum_{i \neq 1} a_i (\nu_i \mu_1 - \mu_i \nu_1) = 0.$$
    We can rewrite this sum as
    $$\mu_1\left(\sum_{i \neq 1} a_i\nu_i\right) + \sum_{i \neq 1} (-a_i\nu_1)\mu_i = 0.$$
    Letting these coefficients of $\mu_i$ be $c_i$, we see that
    $$\sum_i c_i\nu_i = \nu_1\left(\sum_{i \neq 1} a_i \nu_i \right) - \sum_{i \neq 1} a_i\nu_1\nu_i = 0.$$
    Thus, each of these coefficients are zero, and since $\nu_1$ is nonzero, $a_i = 0$. Thus, $\tilde{\eta_j}$'s are independent of each other. We next claim that $\tilde{\eta_j}$ are orthogonal to the $u_i$'s. This is because
    $$\tilde{\eta_j} \cdot u_i = \nu_j \mu_1 \cdot u_i - \nu_1 \mu_j \cdot u_i$$
    $$\eta_j \cdot w_i = \mu_j \cdot u_i + \nu_j \cdot v_i = 0$$
    $$\eta_1 \cdot w_i = \mu_1 \cdot u_i + \nu_1 \cdot v_i = 0$$
    so subtracting the second and third equations gives that the first equation is equal to zero. Finally, we claim that the $u_i$'s are linearly independent of each other. To see this, note that $(u_i, v_i)$ are orthogonal to $(\tilde{\eta_j}, 0)$ and $(\mu_1, \nu_1)$. Since $(0, 1)$ is not orthogonal to $(\mu_1, \nu_1)$, it follows that $(u_i, v_i)$ cannot span $(0, 1)$, so $(u_i, v_i), (0, 1)$ are linearly independent of each other, which implies that $u_i$ are linearly independent of each other. \\\\
    If $\nu_i \neq 0$ for some $i$, we let $\nu_i$ play the role of $\nu_1$ in the above argument. If $\nu_i = 0$ for all $i$, we observe that $\mu_j$ are all linearly independent and that $\|\mu_j \cdot \alpha\|_{\mathbb{R}/\mathbb{Z}} = 0$. In addition, since $v_i \cdot 1 \in \mathbb{Z}$, we have $\|w_i \cdot \tilde{a}\|_{\mathbb{R}/\mathbb{Z}} = \|u_i \cdot a\|_{\mathbb{R}/\mathbb{Z}} = 0$. Hence, choosing a linearly independent subset of the $u_i$'s, we finish.
\end{proof}
We are now ready to finish the proof of \thref{twostepbracket}.
\subsection{Finishing the proof}
We return to (\ref{refinedhypothesis}). Applying \thref{bracketpolynomialcorollary}, we obtain $w_1, \dots, w_r, \eta_1, \dots, \eta_{d - r}$ such that
$$\|w_i \cdot a\|_{\mathbb{R}/\mathbb{Z}} = 0 \text{ and }\|\eta_j \cdot \alpha\|_{\mathbb{R}/\mathbb{Z}} = 0.$$
Unwinding the definitions of $a$ and $\alpha$ gives \thref{twostepbracket}.

\section{The degree two nilsequence case}
Having discussed bracket polynomial heuristics, we now turn to the degree-two two-step nilsequence case. Our proof follows that of \cite{GT12} which applies the van der Corput inequality and performs an analysis of the resulting nilsequence on the group
$$G^\square := G \times_{G_2} G := \{(g', g) \in G^2: g'g^{-1} \in G_2\}.$$
Properties of $G^\square$ can be found in \cite[Lemma A.3, Lemma A.4]{Len2}. The reader is encouraged to observe parallels between this section and the previous section. We now state the main theorem of this section:
\begin{theorem}\thlabel{twostepcase}
Let $N$ be a prime, $\delta \in (0, 1/10)$, and $G/\Gamma$ a two-step nilmanifold of dimension $d$, complexity $M$, equipped with the lower central series filtration. Furthermore, let $F(g(n)\Gamma)$ be a periodic nilsequence modulo $N$ on $G/\Gamma$ with $F$ a $1$-Lipschitz vertical character of nonzero frequency $|\xi| \le M/\delta$. Suppose
$$|\mathbb{E}_{n \in [N]} F(g(n)\Gamma)| \ge \delta.$$
Then either $N \ll (\delta/M)^{-O(d)^{O(1)}}$ or there exists some integer $d_{\mathrm{horiz}} \ge r \ge 0$ and elements $w_1, \dots, w_r \in \Gamma/(\Gamma \cap [G, G])$ and horizontal characters $\eta_1, \dots, \eta_{d_{\mathrm{horiz}} - r}$ all bounded by $(\delta/M)^{-O(d)^{O(1)}}$ such that 
\begin{itemize}
\item $\psi_{\mathrm{horiz}}(w_i)$'s are linearly independent of each other and $\eta_j$'s are linearly independent of each other and $\langle \eta_j, w_i \rangle = 0$ for all $i$ and $j$.
\item We have
$$\|\xi([w_i, g])\|_{C^\infty[N]} = 0$$
$$\|\eta_j \circ g\|_{C^\infty[N]} = 0.$$
\end{itemize}
\end{theorem}
It's worth noting that the subgroup $\tilde{G} = \{g \in G: \eta_j(g) = 0, [w_i, g] = \mathrm{id}_G \forall i, j\}$ is an abelian subgroup of $G$ since given two elements $g, h \in \tilde{G}$, we see that since $\eta_j(h) = 0$ and $w_i$ and $\eta_j$ are orthogonal, it follows that the horizontal component of $h$ can be spanned by the $w_i$'s, so to verify that $[g, h] = 0$, it just suffices to verify that $[w_i, g] = \mathrm{id}_G$, which is true by definition. In fact, by \cite[Lemma A.9]{Len2}, each abelian rational subgroup of $G$ is a subgroup of some group of this form. Combining this lemma with \thref{factorization}, we obtain the following.
\begin{corollary}\thlabel{twostepcor}
Let $N$ be a prime, $0 < \delta < \frac{1}{10}$, $G/\Gamma$ be a two-step nilmanifold of dimension $d$, complexity $M$, and equipped with the standard filtration. Furthermore, let $F(g(n)\Gamma)$ be a periodic nilsequence modulo $N$ on $G/\Gamma$ with $F$ a $1$-Lipschitz vertical character of nonzero frequency $|\xi| \le M/\delta$. Suppose $G/\Gamma$ has a one-dimensional vertical torus, and
$$|\mathbb{E}_{n \in [N]} F(g(n)\Gamma)| \ge \delta.$$
Then either $N \ll (\delta/M)^{-O(d)^{O(1)}}$ or we can write $g(n) = \epsilon(n)g_1(n)\gamma(n)$ where $\epsilon$ is constant, $g_1(n)$ lies on an abelian subgroup of $G$ with rationality $(\delta/M)^{-O(d)^{O(1)}}$ and the image of $\gamma$ lies inside $\Gamma$.
\end{corollary}
\begin{proof}[Proof of \thref{twostepcase}.]
We first make a few preliminary reductions. By \cite[Lemma 2.1]{Len2}, we may reduce to the case that $g(0) = \mathrm{id}_G$ and $|\psi(g(1))| \le \frac{1}{2}$. \\\\
Suppose 
$$|\mathbb{E}_{n \in [N]} F(g(n)\Gamma)| \ge \delta.$$
Using the van der Corput inequality, we see that there are $\delta^{O(1)}N$ many $h$'s such that for each such $h$,
\begin{equation}\label{vandercorputnilsequence}
|\mathbb{E}_{n \in [N]} F(g(n + h)\Gamma) \overline{F(g(n)\Gamma)}| \ge \delta^{O(1)}.
\end{equation}
We recall (once again) from \cite[Definition 4.1, 4.2]{Len2} the definitions
$$G \times_{G_2} G = \{(g', g): g'g^{-1} \in G_2\} = G^\square$$
$$\Gamma \times_{\Gamma \cap G_2} \Gamma = \{(\gamma, \gamma'): \gamma'\gamma^{-1} \in G_2\} := \Gamma^\square$$
and $g_2(n) = g(n)g(1)^{-n}$. By defining $F_h(x, y) = F(\{g(1)^h\}x) \overline{F(y)}$, the nonlinear part $g_2(n) = g(n)g(1)^{-n}$, and $g_h(n) = (\{g(1)^h\}^{-1} g_2(n + h) g(1)^n\{g(1)^h\}, g_2(n)g(1)^{n})$, we see that
$$|\mathbb{E}_{n \in [N]} F_h(g_h(n)\Gamma)| \ge \delta^{O(1)}.$$
One can verify that equipping $G^\square$ with the lower central series filtration, $g_h \in \mathrm{poly}(\mathbb{Z}, G^\square)$, and that $F_h(g_h(n)\Gamma^\square)$ is a periodic nilsequence. Since $F_h$ is invariant under $G_2^\triangle$, the diagonal subgroup of $G_2^2$, and since $[G^\square, G^\square] = G_2^\triangle$ it follows that $F_h$ descends to a function $\overline{F_h}$ on $\overline{G^\square} := G^\square/G_2^\triangle$, which is a one-step nilpotent group. By \cite[Lemma A.6]{Len2} we may approximate $\overline{F_h}$ by
$$\overline{F_h} = \sum_{|\eta| \le c(\delta)^{-1}} F_{h, \eta} + O_{L^\infty}(c(\delta))$$
with $F_{h, \eta}$ $c(\delta)^{-1}$-Lipschitz vertical characters with frequency $\eta$. \\\\
We now analyze characters on $\overline{G^\square}$. Such a character lifts to a horizontal character on $G^\square$ which annihilates $G_2^\triangle$. By \cite[Lemma A.3]{Len2} that we may decompose $\eta(g_2g_1, g_1) = \eta(g_1, g_1) + \eta(g_2, id) = \eta^1(g_1) + \eta^2(g_2)$ with $g_1 \in G$ and $g_2 \in G_2$ where $\eta^1$ and $\eta^2$ are horizontal characters of size at most $c(\delta)^{-1}$. In order to emphasize that $\eta^2$ ``lies in the $x_2 - x_1$ direction", we shall write $\eta^2$ as $\eta^2 \otimes \overline{\eta^2}$. Since $F$ is a frequency $\xi$ on $G_2$, we expect $F_{h, \eta}$ to also be of frequency $\xi \otimes \overline{\xi}$ on $G_2^2/G_2^\triangle$. This may not necessarily be true as written above, but we may average over $G_2^2/G_2^\triangle$ as follows:
$$\overline{F_h}(x) = \int_{G_2^2/G_2^\triangle} \tilde{F}_h(g_2x\Gamma)e(-\xi \otimes \overline{\xi}(g_2))dg_2 = \sum_{|\eta'| \le c(\delta)^{-1}} F_{h, \eta'} + O_{L^\infty}(c(\delta)).$$
The point is that $(\eta' - \xi \otimes \overline{\xi})(G_2^2) = 0$. Here, note that we have abusively lifted $\eta'$ to a horizontal character on $G \times_{G_2} G_2/G_2^\triangle$ and $\xi \otimes \overline{\xi}$ to a character on $G_2^2$ that annihilates $G_2^\triangle$. By applying the Pigeonhole principle, there exists one frequency $\eta'$ independent of $h$ such that for $c(\delta)N$ many $h \in \mathbb{Z}/N\mathbb{Z}$,
\begin{equation}\label{twostepcharacter}
|\mathbb{E}_{n \in \mathbb{Z}/N\mathbb{Z}} F_{\eta'}(g(n)\Gamma)| \ge c(\delta).
\end{equation}
Since $G$ was two-step, $G \times_{G_2} G$ is one-step, so decomposing $\eta'(g', g) = \eta_1(g) + \xi(g'g^{-1})$, (\ref{twostepcharacter}) is equivalent to the fact that
\begin{equation}\label{bracketform}
|\mathbb{E}_{n \in [N]} e(n\eta_1(g(1)) + \xi(g_2(n + h)g_2(n)^{-1}) + \xi([\{g(1)^h\}, g(1)^n]))| \ge (\delta/M)^{O(d)^{O(1)}}
\end{equation}
for $c(\delta)N$ many elements $h \in [N]$. \\\\
We claim that $\xi([g(1), [g(1)^h])$ is rational with denominator $N$. To see this, note that 
$$\xi([g(1), \{g(1)^h\}]) = \xi([g(1), [g(1)^h]]),$$
and since $\xi([g(1), [g(1)^h]])^N = \xi([g(1)^N, [g(1)^h]]) = 0$, it follows that $\xi([g(1), [g(1)^h])$ is rational with denominator $N$. From (\ref{bracketform}) and \thref{periodic}, we have for some $\beta, \gamma \in \mathbb{R}$ with $\beta$ rational with denominator $AN$ for some $A = O_k(1)$,
$$\|\beta n + \gamma + \xi([g(1), \{g(1)^h\}])\|_{\mathbb{R}/\mathbb{Z}} = 0.$$
We can write 
$$\xi([g(1), \{g(1)^h\}]) = \langle (C - C^t)g(1), \{g(1)^h\} \rangle = \langle a, \{\alpha h\} \rangle$$
where $C - C^t$ is the antisymmetric matrix representing the commutator identity in the horizontal torus. Note that $C - C^t$ has height at most $M$. This suggests that we should've made a change of variables $n \mapsto M_1n$ in (\ref{vandercorputnilsequence}) (using the fact that $M_1$ has a modular inverse mod $N$ since $N$ is prime) for some $M_1 \le AM^d$ which is a multiple of the denominator of each entry of $C - C^t$ and $A$,
$$|\mathbb{E}_{n \in [N]} F(g(M_1n)\Gamma)| \ge \delta.$$
Thus, after applying the change of variables, we can assume that $a$, $\alpha$, and $\beta$ have denominator $N$. Applying \thref{bracketpolynomialcorollary} and noticing that $C - C^t$ is antisymmetric, we obtain $w_i$'s and $\eta_j$'s which are linearly independent, $\eta_j(w_i) = 0$, $\|M_1\xi([w_i, g])\|_{\mathbb{R}/\mathbb{Z}} = 0$, and $\|\eta_j(g)\|_{\mathbb{R}/\mathbb{Z}} = 0$. Using the fact that $N$ is prime, we obtain $\|\xi([w_i, g])\|_{\mathbb{R}/\mathbb{Z}} = 0$. This completes the proof of \thref{twostepcase}.    
\end{proof}

\section{The two-step polynomial sequence case}
In the previous section, we observed that in the two-step case, applying van der Corput once landed us in the group $\overline{G^\square}$, which was one-step. In general, this may not be true, even in the two-step polynomial sequence case. This is exhibited by the following example in bracket polynomial formalism. Consider
$$\alpha n^2 [\beta n].$$
Differentiating once in $h$ gives a top term of
$$2\alpha nh[\beta n] + \alpha n^2[\beta h].$$
Here, we see that the term $2\alpha nh[\beta n]$ is still a ``bracket term" in $n$. Hence, our proof divides into two cases.
\begin{itemize}
    \item The first case is what happens when $\overline{G^\square}$ is one-step. This would correspond to \thref{centralcase} and in bracket polynomial formalism, what happens when the bracket polynomial is of the form
    $$\sum_{i = 1}^d \alpha_i n[\beta_i n] + P(n)$$
    for $P$ a (not necessarily degree two) polynomial.
    \item The second case is what happens when $\overline{G^\square}$ is not one-step. We have not isolated a specific lemma for that case, for its proof occupies much of \thref{twosteppolynomial}. In bracket polynomial formalism, this corresponds to a bracket polynomial of the form
    $$\sum_{i = 1}^d P_i(n)[Q_i(n)] + R(n)$$
    where $P_i, Q_i$ are polynomials with at least one with degree larger than one and $R$ is a polynomial.
\end{itemize}
We now state the main theorem of this section.
\begin{theorem}\thlabel{twosteppolynomial}
Let $\delta \in (0, 1/10)$, $N > 100$ prime, and $G/\Gamma$ a two-step nilmanifold of dimension $d$, complexity $M$, and degree $k$. Furthermore, let $F(g(n)\Gamma)$ be a periodic nilsequence modulo $N$ on $G/\Gamma$ with $F$ a $1$-Lipschitz vertical character with nonzero frequency $|\xi| \le M/\delta$. Suppose
$$|\mathbb{E}_{n \in [N]} F(g(n)\Gamma)| \ge \delta.$$
Then either $N \ll (\delta/M)^{-O_k(d)^{O_k(1)}}$ or else there exists some integer $d_{\mathrm{horiz}} \ge r \ge 1$ elements $w_1, \dots, w_r \in \Gamma$ with $\psi_{horiz}(w_1), \dots, \psi_{horiz}(w_r)$ linearly independent, linearly independent horizontal characters $\eta_1, \dots, \eta_{d_{horiz} - r}$ with $|w_i|, |\eta_j| \le (\delta/M)^{-O_k(d)^{O_k(1)}}$, $\langle \eta_j, w_i \rangle = 0$, and
$$\|\xi([w_i, g])\|_{C^\infty[N]} , \|\eta_j \circ g\|_{C^\infty[N]} = 0.$$
\end{theorem}
It turns out that the proof of the two-step polynomial sequence case breaks down naturally into two cases: one where $\xi([G_2, G]) = 0$ and one where $\xi([G_2, G]) = \mathbb{R}$. This is analogous to the two cases considered in \cite[Section 4]{Len2}. We start with the case of when $\xi([G_2, G]) = 0$ first.
\begin{lemma}\thlabel{centralcase}
Let $\delta \in (0, 1/10)$, $N > 100$ prime, and $G/\Gamma$ be a two-step nilmanifold of dimension $d$, complxity $M$, and degree $k$ with $\xi([G_2, G]) = 0$. Furthermore, let $F(g(n)\Gamma)$ be a periodic nilsequence modulo $N$ on $G/\Gamma$ with $F$ a $1$-Lipschitz vertical character with nonzero frequency $|\xi| \le M/\delta$. Suppose
$$|\mathbb{E}_{n \in [N]} F(g(n)\Gamma)| \ge \delta.$$
Then either $N \ll (M/\delta)^{O_k(d^{O_k(1)})}$ or else there exists some integer $d_{\mathrm{horiz}} \ge r \ge 1$, elements $w_1, \dots, w_r \in \Gamma$ with $\psi_{\mathrm{horiz}}(w_1), \dots, \psi_{\mathrm{horiz}}(w_r)$ linearly independent, and linearly independent horizontal characters $\eta_1, \dots, \eta_{d_{horiz} - r}$ with $|w_i|, |\eta_j| \le (M/\delta)^{O_k(d^{O_k(1)})}$, $\langle \eta_j, w_i \rangle = 0$, and
$$\|\xi([w_i, g])\|_{C^\infty[N]} , \|\eta_j \circ g\|_{C^\infty[N]} = 0.$$
\end{lemma}
\begin{proof}
By \cite[Lemma 2.2]{Len2}, we may reduce to the case of when $g(0) = 1$ and $|\psi(g(1))| \le 1/2$. We proceed similarly as the proof of \thref{twostepcase} with one key difference. Instead of Fourier expanding along the vertical torus, we Fourier expand along the $G_2$-torus. By \cite[Lemma A.6]{Len2}, we may decompose
$$F = \sum_{|\alpha| \le c(\delta)^{-1}} F_\alpha + O(c(\delta))$$
where $\alpha$ is a $G_2$-vertical frequency. By averaging over $G_{(2)}$, we may write
$$F(x) = \int F(g_2x)e(-\xi(g_2)) dg_2 = \sum_{|\beta| \le c(\delta)^{-1}} F_\beta + O(c(\delta))$$
where crucially, $(\beta - \xi)(G_{(2)}) = 0$. \\\\
Once again, we apply the van der Corput inequality
$$|\mathbb{E}_{n \in [N]} F(g(n)\Gamma)\overline{F(g(n + h)\Gamma)}| \ge \delta^2$$
for $\delta^2N$ many $h \in [N]$. Then letting $g_2(n) := g(n)g(1)^{-n}$ be the nonlinear part of $g$, we denote
$$g_h(n) = (\{g(1)^h\}^{-1} g_2(n + h)g(1)^n\{g(1)^h\}, g(n))$$
and $F_h(x, y) = \overline{F}(\{g(1)^h\}x)F(y)$. By \cite[Lemma A.3]{Len2}, we see that $g_h$ lies inside $G^\square$ and $F_h(g_h(n))$ descends to $\tilde{F}_h(\tilde{g}_h(n)\overline{\Gamma^\square})$ a nilsequence on $\overline{G^\square}$. Since $g(n)\Gamma$ is periodic modulo $N$, it follows that $g_h(n)\Gamma \times \Gamma$ is periodic modulo $N$ and thus $g_h(n)\Gamma \times_{G_2 \cap \Gamma} \Gamma$ is also periodic modulo $N$, and so $\tilde{g}_h(n)\overline{\Gamma^\square}$ is periodic modulo $N$. The hypothesis then rearranges to
$$|\mathbb{E}_{n \in [N]} \tilde{F_h}(\tilde{g_h}(n)\overline{\Gamma^\square})| \ge \delta^2$$
for $\delta^2N$ many $h \in [N]$. Making a change of variables for some integer $1 \le M_1 \le (10^kk)!M^d$
$$|\mathbb{E}_{n \in [N]} \tilde{F_h}(\tilde{g_h}(M_1n)\overline{\Gamma^\square})| \ge \delta^2$$
By \cite[Lemma A.3]{Len2}, it follows that $F_h$ has Lipschitz norm at most $M^{O(1)}$ on $G^\square$, $G^\square$ is abelian; we now repeat a similar procedure as in the proof of \thref{twostepcase},approximating $\tilde{F}_h$ by $F_{\eta, h}$ where $(\eta - \xi \otimes \overline{\xi})(G_{(2)}) = 0$, Pigeonholing in $\eta$ so it is $h$-independent.\\\\
Again, we decompose via \cite[Lemma A.3]{Len2} $\eta(g_2g_1, g_1) = \eta^1(g_1) + \eta^2(g_2)$ with $g_1 \in G$ and $g_2 \in G_2$. It thus follows that
$$|\mathbb{E}_{n \in \mathbb{Z}/N\mathbb{Z}} e(\eta^1(g(M_1n)) + \eta^2(g_2(M_1n + h)) - \eta^2(g_2(M_1n)) + \xi([g(1)^{M_1n}, \{g(1)^h\}]))| \ge c(\delta)$$
for $c(\delta)N$ many $h \in \mathbb{Z}/N\mathbb{Z}$. Thus, by classical results in Diophantine approximation (e.g., \cite[Lemma A.11]{Len2}),
$$\|\eta^1(g(M_1n)) + \eta^2(g_2(M_1n + h)) - \eta^2(g_2(M_1n)) + \xi([g(1)^{M_1n}, \{g(1)^h\}])\|_{C^\infty[N]} = 0.$$
Expanding $\eta^2(g_2(M_1n + h))$ out, applying the hypothesis, and applying Vinogradov's lemma to eliminate the coefficients of $nh^i$ for $i > 1$ (see also \cite[Lemma 4.2]{Len2}), we see that there exists $\beta$ and $\gamma$ such that
$$\|\beta h + \xi([g(1)^{M_1n}, \{g(1)^h\}]) + \gamma\|_{\mathbb{R}/\mathbb{Z}} = 0$$
The point of making the change of variables is so that by \thref{periodic}, the coefficients of $g_2(M_1 \cdot)$ have denominator $N$, and $\xi([g(1)^{M_1n}, \{g(1)^h\}])$ consists of $\langle an, \{\alpha h\} \rangle$ where $a$ and $\alpha$ have denominator $N$. Applying \thref{bracketpolynomialcorollary} and using the fact that $N$ is prime yields $w_i$'s and $\eta_j$'s which satisfy the conclusions of the lemma.
\end{proof}
We now address the case of when $\xi([G_2, G]) = \mathbb{R}$. The key fact to keep in mind in the below proof is that if $G_{(2)}$ is one-dimensional, then $[G_2, G] = [G, G]$.
\begin{proof}[Proof of \thref{twosteppolynomial}]
By \cite[Lemma 2.3]{Len2}, we may assume that $G_{(2)}$ is one-dimensional. If $G_2$ lies in the center of $G$, we may apply \thref{centralcase} to finish. \\\\
At this point, the proof of \cite{GT12} Fourier expands $F$ into $G_k$-vertical characters and reduces to a nilsequence on $\overline{G^\square}$ of degree $k - 1$. In our proof, we wish to preserve information of $F$ being a $G_{(2)}$-vertical character in some way. This can be done by modifying the filtration. For $k \ge \ell > 2$, we replace $G_\ell$ with $G_\ell G_{(2)}$. Then since $G_{(2)}$ is in the center of $G$, this gives a filtered nilmanifold; by \cite[Lemma B.12]{Len2}, there exists a Mal'cev basis which makes this filtered nilmanifold of complexity at most $c(\delta)^{-1}$. \\\\
As in previous arguments, at the cost of replace $\delta$ with $c(\delta)$ and increasing the Lipschitz constant of $F$ to $c(\delta)^{-1}$, we may replace $F$ with a $G_k$-vertical character of frequency $\eta$ with $(\eta - \xi)(G_{(2)}) = 0$. \\\\
By the van der Corput inequality, we have for $c(\delta)N$ many $h$'s
$$|\mathbb{E}_{n \in [N]} F(g(n + h)\Gamma) \overline{F(g(n)\Gamma)}| \ge c(\delta).$$
Defining
$$g_h(n) = (\{g(1)^h\}^{-1} g_2(n + h)g(1)^n\{g(1)^h\}, g_2(n)g(1)^n)$$
and $F_h(x, y) = F(\{g(1)^h\}x)\overline{F(y)}$, it follows that
$$|\mathbb{E}_{n \in [N]} F_h(g_h(n))| \ge \delta^{O(1)}.$$
As before, by \cite[Lemma A.3, Lemma A.4]{Len2}, $g_h$ is a periodic polynomial sequence on $G^\square$. By \cite[Lemma A.3]{Len2}, this group has the filtration $(G \times_{G_2} G)_i = G_i \times_{G_{i + 1}} G_i$ with $(G \times_{G_2} G)_k = G_k^\triangle$. However, $F_h$ is $G_k^\triangle$-invariant, so descends via a quotient by $G_k^\triangle$ to a degree $d - 1$ nilsequence $\tilde{F}_h(\overline{g_h}(n)\overline{\Gamma^\square})$. The observation to make is that $\tilde{F}_h$ is a nilcharacter of frequency $\xi \otimes \overline{\xi}$ on $\overline{G^\square}$. This is because $[G_2, G] = [G, G]$ so $(\overline{G^\square})_{(2)} = G_{(2)} \times G_{(2)}/G_{k}^\triangle$. Applying the induction hypothesis, we have for linearly independent horizontal characters $\eta_1, \dots, \eta_{d'}$ and elements $w_1, \dots, w_{d_{horiz, \overline{G^\square}} - d'} \in \overline{\Gamma^\square}$ such that $\psi_{horiz, \overline{G^\square}}(w_1), \dots, \psi_{horiz, \overline{G^\square}}(w_{d_{horiz, \overline{G^\square}} - d'})$ are linearly independent and for each $i, j$, and $c(\delta)N$ many $h \in [N]$ that
$$\|\xi \otimes \overline{\xi}([w_i, \tilde{g_h}(n)])\|_{C^\infty[N]} = 0$$
$$\|\eta_j \circ \tilde{g_h}(n)\|_{C^\infty[N]} = 0.$$
where $\tilde{g_h}$ is the projection of $g_h$ to $G^\square$. Abusively lifting $\eta_j$ to be horizontal characters on $G^\square$ and lifting $w_j$ to be elements in $\Gamma^\square$ and appending elements $z_1, \dots, z_\ell$ in $G_k^\triangle \cap \Gamma^\square$ to the set of $w_i$'s such that $\overline{\exp(z_1)}, \dots, \overline{\exp(z_\ell)}$ span $G_k^\triangle/[G^\square, G^\square]$, we see that
\begin{itemize}
    \item since $G_k^\triangle/[G^\square, G^\square]$ can be naturally identified with its Lie algebra which can be identified via the Mal'cev coordinates as the orthogonal complement within $G_k^\triangle$ of $[G^\square, G^\square]$, that we can take $z_1, \dots, z_\ell$ to be size at most $c(\delta)^{-1}$, and thus $w_1, \dots, w_{d_{horiz, G^\square} - d'}$ also has size at most $c(\delta)^{-1}$;
    \item Since $\eta_1, \dots, \eta_{d'}$ annihilates $G_k^\square$, they span the annihilators of $w_1, \dots, w_{d_{horiz, G^\square} - d'}$.
\end{itemize}
Hence, we have for each $i, j$
$$\|\xi \otimes \overline{\xi}([w_i, g_h(n)])\|_{C^\infty[N]} = 0$$
$$\|\eta_j \circ g_h(n)\|_{C^\infty[N]} = 0.$$
By \cite[Lemma A.3]{Len2}, we may write $w_i = (u_iv_i, u_i)$ and decompose $\eta_j(g', g) = \eta_j^1(g) + \eta_j^2(g'g^{-1})$ so
\begin{align*}
\eta_j(g_h(n)) &= \eta_j^1 (g(n)) + \eta_j^2(g_2(n + h)g_2(n)^{-1}) \\
\eta_j(w_i) &= \eta_j^1(u_i) + \eta_j^2(v_i) = 0 \\
\xi \otimes \overline{\xi}([w_i, g_h(n)]) &= \xi \otimes \xi^{-1}(([u_i, g(n)], [u_iv_i, g(n)][u_iv_i, g_2(n + h)g_2(n)^{-1}]) \\
&= \xi([v_i, g(n)] + [u_iv_i, g_2(n + h)g_2(n)^{-1}])
\end{align*}
Not here we are crucially using that $[G_2, G] = [G, G]$, so $\beta_j$ annihilates $[G, G]$. By expanding out the polynomial $[u_iv_i, g_2(n + h)g_2(n)^{-1}]$, and grouping coefficients, and applying a polynomial Vinogradov-type lemma (e.g., \cite[Lemma A.11]{Len2}; see also \cite[Lemma 4.2]{Len2}), it follows that 
$$\|\eta_j^1 \circ g\|_{C^\infty[N]}, \|\xi([v_i, g])\|_{C^\infty[N]}, \|\eta_j^2 \circ g_2\|_{C^\infty[N]}, \|\xi([u_iv_i, g_2])\|_{C^\infty[N]} = 0.$$
Note that here we must use the fact that $N$ is prime and the fact that $g_2$ has horizontal component with denominator $N$ to eliminate the binomial coefficients that come from expanding out $g_2(n + h)g_2(n)^{-1}$. Let $\tilde{G} = \{g \in G: \eta_j^1 \circ g = 0, [v_i, g] = 0\}$ and $\tilde{G}_2 = \{g \in \tilde{G} \cap G_2, \eta_j^2(g) = 0, [u_iv_i, g] = 0\}$. We note that $[G, G] \subseteq \tilde{G}$ and $G_2$ and also that $\eta_j^2([G, G]) = 0$. Hence, (abusing 
notation) the sequence of subgroups $\tilde{G}_i = \tilde{G}_2 \cap G_i$ for $i > 2$ and $\tilde{G}_0 = \tilde{G}_1 = \tilde{G}$ form a filtration. \\\\
We claim that $[\tilde{G}, \tilde{G}_2] = 0$. To show this, we let $\tilde{H} = \tilde{G}^\square$ and we claim that $\tilde{H}$ is Abelian. To see this, note that each element of $\tilde{H}$ of the form $(gg_2, g)$ satisfies $\eta_j^1(g) + \eta_j^2(g_2) = 0$ and $[v_i, g] + [u_iv_i, g_2] = 0$. Since $\eta_j^1(u_i) + \eta_j^2(v_i) = 0$ for each $i, j$, it follows that $(gg_2, g)$ can be generated by $(u_iv_i, u_i)$ modulo $[G, G]^2$. However, for any other $(hh_2, h)$ in $\tilde{H}$, we have $\xi \otimes \overline{\xi}[(u_iv_i, u_i), (hh_2, h)] = 0$. Hence $\tilde{H}$ is Abelian. Finally, we have $\xi \otimes \overline{\xi}([(g, g), (hg_2, h)]) = \xi([g, g_2]) = 0$ whenever $g \in \tilde{G}$ and $g_2 \in \tilde{G}_2$. This shows that $\xi([\tilde{G}, \tilde{G}_2]) = 0$ and since $[G, G]$ is one-dimensional, $[\tilde{G}, \tilde{G}_2] = 0$. Thus, by \thref{factorization} and \thref{removerational}, we may write $g(n) = g'(n)\gamma'(n)$  where $\gamma'(n)$ has image in $\Gamma$ and $g'(n)$ has image in $G'$. We see that
$$g'(1)^n \equiv g(1)^n\gamma'(1)^n \pmod{[G, G]}$$
and by abusing notation,
$$g'_2(n) \equiv g_2(n)\gamma'_2(n) \pmod{[G, G]}.$$
Since $\eta_j$ and $\xi([u_iv_i, \cdot])$ both annihilate $[G, G]$, we see that
$$\|\eta_j \circ g'_2\|_{C^\infty[N]}, \|\xi([u_iv_i, g'_2])\|_{C^\infty[N]} = 0.$$
Thus, by \thref{factorization2}, and \thref{removerational}, we can write $g(n)\Gamma = g_1(n)\Gamma$ with $g_1 \in \text{poly}(\mathbb{Z}, \tilde{G})$. Finally, we apply \thref{centralcase} and \cite[Lemma A.9]{Len2} to $g_1(n)$ to find linearly independent $\alpha_1, \dots, \alpha_{d'}$ of size at most $c(\delta)^{-1}$ with the property that
$$\|\alpha_i\circ g\|_{C^\infty[N]} = \|\alpha_i \circ g_1\|_{C^\infty[N]} = 0$$
and such that if $w, w'$ are elements in $\Gamma/([G, G] \cap \Gamma)$ that are orthogonal to each of the $\alpha_i$'s, then $\xi([w, w']) = 0$. To satisfy the second conclusion, we invoke \thref{factorization} and \thref{removerational} to find a factorization $g_1(n) = g_1'(n)\gamma(n)$ where $g_1'$ lies in the kernel of $\alpha_i$ for all $i$ and $\gamma$ lies in $\Gamma$. It follows that $g_1'$ lies in the subspace generated by orthogonal elements to $\alpha_i$'s. Thus, for any $w \in \Gamma/(\Gamma \cap [G, G])$ orthogonal to all of the $\alpha_i$'s, it follows that
$$\|\xi([w, g_1'(n)\Gamma])\|_{C^\infty[N]} = \|\xi([w, g_1(n)])\|_{C^\infty[N]} = \|\xi([w, g(n)])\|_{C^\infty[N]} = 0.$$
By \cite[Lemma A.8]{Len2}, we may choose elements $w_1, \dots, w_{d_{horiz} - d'}$ inside $\Gamma$ whose size is at most $c(\delta)^{-1}$ and such that their projections to $\Gamma/([G, G] \cap \Gamma)$ are linearly independent and are orthogonal to the $\alpha_i$. To verify size bounds, we must verify that the projection of the Mal'cev coordinates of $w_i$ to the orthogonal complexity of $[\mathfrak{g}, \mathfrak{g}]$ is bounded by $c(\delta)^{-1}$. This follows from invoking \cite[Lemma A.7]{Len2} to construct a $c(\delta)^{-1}$-rational basis for the orthogonal to $[\mathfrak{g}, \mathfrak{g}]$ and constructing $c(\delta)^{-1}$-rational basis for $[\mathfrak{g}, \mathfrak{g}]$ and rewriting the Mal'cev coordinates of $w_i$ in terms of linear combinations of these bases and projecting to the dimensions where $[\mathfrak{g}, \mathfrak{g}] = 0$. By Cramer's rule we can ensure that if we write $w_i$ terms of these linear combinations, all components are rational with height at most $c(\delta)^{-1}$.
\end{proof}

\section{The general periodic case}
Recall the following theorem.
\mainresulta*

\begin{proof}
As the direct proof of this theorem is no shorter than the proof in \cite{Len2}, we shall only give a proof of this assuming \cite[Theorem 3]{Len2}. We see from hypothesis that
$$|\mathbb{E}_{n \in [MN]} F(g(n)\Gamma)| \ge \delta$$
for each $M$. Invoking \cite[Theorem 3]{Len2}, we find $\eta_1, \dots, \eta_r$ such that
$$\|\eta_i \circ g\|_{C^\infty[MN]} \le c(\delta)^{-1}$$
and such that $s$ elements $w_1, \dots, w_s \in G' := \bigcap_{i = 1}^r \text{ker}(\eta_i)$ satisfies
$$\xi([w_1, \dots, w_s]) = 0.$$
Sending $M$ to infinity, we see that
$$\|\eta_i \circ g\|_{C^\infty[MN]} = 0$$
as desired.
\end{proof}

\section{The complexity one polynomial Szemer\'edi theorem}
We now deduce \thref{mainresult2}. Fix $P(x)$ and $Q(x)$ to be linearly independent polynomials. For functions $f, g, h, k\colon \mathbb{Z}/N\mathbb{Z} \to \mathbb{C}$, define 
$$\Lambda(f, g, k, p) := \mathbb{E}_{x, y} f(x)g(x + P(y))k(x + Q(y))p(x + P(y) + Q(y)).$$
and 
$$\Lambda^1(f, g, j, p) := \mathbb{E}_{x, y, z} f(x)g(x + y)k(x + z)p(x + y + z).$$
We will show the following:
\begin{theorem}\thlabel{asymptotic}
There exists some $c_{P, Q}$ such that for any one-bounded $f, g, k, p$, we have
$$\Lambda(f, g, k, p) = \Lambda^1(f, g, k, p) + O\left(\frac{1}{\exp(\log^{c_{P, Q}}(N))}\right).$$
\end{theorem}
To see how this implies \thref{mainresult2}, suppose the set $A$ has no nontrivial configuration of the form $(x, x + P(y), x + Q(y), x + P(y) + Q(y))$. Then
$$|\Lambda(1_A, 1_A, 1_A, 1_A)| \le \frac{1}{N}.$$
We have $|\Lambda^1(1_A, 1_A, 1_A, 1_A)| = \|1_A\|_{U^2(\mathbb{Z}/N\mathbb{Z})}^4 \ge \|1_A\|_{U^1(\mathbb{Z}/N\mathbb{Z})}^4 = \alpha^4$
where $\alpha$ is the density of $A$ in $\mathbb{Z}/N\mathbb{Z}$. \thref{asymptotic} then implies that
$$\alpha \ll O\left(\frac{1}{\exp(\log^{c_{P, Q}}(N))}\right)$$
as desired. \\\\
The proof of \thref{asymptotic} will closely follow the proof of \cite[Theorem 3]{Len22}. To prove this theorem, we prove the following inverse-type theorem:
\begin{theorem}\thlabel{polynomialinverse}
Suppose $|\Lambda(f, g, k, p)| \ge \delta$. Then either $N \ll \exp(\log^{O_{P, Q}(1)}(1/\delta))$ or $\|p\|_{U^2} \gg \exp(-\log^{O_{P, Q}}(1/\delta))$.   
\end{theorem}
Let us assume for a moment that we can prove \thref{polynomialinverse}. We shall deduce \thref{asymptotic}. For arbitrary one-bounded $p$, we invoke \cite[Lemma 4.2]{Len22} to decompose $p = p_a + p_b + p_c$ with 
$$\|\hat{p_a}\|_{\ell^1} \le \epsilon_1^{-1}, \|p_b\|_{L^1} \le \epsilon_2, \|p_c\|_{L^\infty} \le \epsilon_3^{-1}, \|\hat{p_c}\|_{L^\infty} \le \epsilon_4$$
where $\epsilon_1, \dots, \epsilon_4$ will be chosen later and satisfy $\epsilon_1\epsilon_4^{-1} + \epsilon_2^{-1}\epsilon_3 \le \frac{1}{2}$. We thus have
$$\Lambda(f, g, k, p) = \Lambda(f, g, k, p_a)  + \Lambda(f, g, k, p_b)+ \Lambda(f, g, k, p_c),$$
$$\Lambda(f, g, k, p_a) = \Lambda^1(f, g, k, p_a) + O(N^{-\delta}\epsilon_1^{-1}),$$
and
$$|\Lambda(f, g, k, p_b)| \le \epsilon_2.$$
To control $|\Lambda(f, g, k, p_c)$, we invoke \thref{polynomialinverse} which states that either
$$|\Lambda(f, g, k, p_c)| \ll \epsilon_3^{-1} \exp(-\log^{c_{P, Q}}(N))$$
or
\begin{align*}
|\Lambda(f, g, k, p_c)| &\le \epsilon_3^{-1} |\lambda(f, g, k, \epsilon_3 p_c)| \\
&\le \epsilon_3^{-1}\exp(-\log^{C_{P, Q}}(\|\epsilon_3 \hat{p_c}\|_{L^\infty}^{-1/2})) \\
&= \epsilon_3^{-1} \exp(-\log^{C_{P, Q}}(\epsilon_3^{-1/2} \epsilon_4^{-1/2})).
\end{align*}
Choosing $\epsilon_1 = N^{\alpha}$, $\epsilon_3 = \exp(-\log^{c'_{P, Q}}(N))$, $\epsilon_4 = N^{-\alpha'}$ for $\alpha' > \alpha$, we see that 
$$\exp(-\log^{C_{P, Q}}(\epsilon_3^{-1/2} \epsilon_4^{-1})) \ll_\epsilon \exp(-(\alpha'/2 - \epsilon)^{C_{P, Q}}\log^{C_{P, Q}}(N)).$$
Thus, choosing $\epsilon_3$ to be larger than both $\exp(-\log^{c_{P, Q}}(N))$ and $\exp(-(\alpha' - \epsilon)^{C_{P, Q}}\log^{C_{P, Q}}(N))$, and $\epsilon_2$ to be less than $\epsilon_3$ but also of the form $\exp(-\log^{1/O_{P, Q}(1)}(N))$, we have the desired estimate of
$$\Lambda(f, g, k, p) = \Lambda^1(f, g, k, p_a) + O(\exp(-\log^{c_{P, Q}}(N))).$$
The decomposition we chose also allows us to prove the same estimates of $\Lambda^1$, with $\Lambda$ replaced with $\Lambda^1$, and one of the estimates being
$$|\Lambda^1(f, g, k, p_c)| \le \epsilon_3^{-1/2}\|\hat{p_c}\|_{L^\infty}^{1/2} \le \epsilon_3^{-1/2}\epsilon_4^{1/2} \ll_\epsilon N^{-\alpha'/2 + \epsilon}.$$
Thus, we can also show that
$$\Lambda^1(f, g, k, p) = \Lambda^1(f, g, k, p_a) + O(\exp(-\log^{c_{P, Q}}(N)))$$
which gives
$$\Lambda(f, g, k, p) = \Lambda^1(f, g, k, p) + O(\exp(-\log^{c_{P, Q}}(N)))$$
as desired. It thus remains to prove \thref{polynomialinverse}.
\subsection{Proof of \thref{polynomialinverse}}
We will show the following:
\begin{proposition}\thlabel{degreelowering}
Given functions $f, g, k: \mathbb{Z}/N\mathbb{Z} \to \mathbb{C}$, we define
$$\mathcal{D}(f, g, k)(x) = \mathbb{E}_y f(x - P(y) - Q(y)) g(x - Q(y)) k(x - P(y)).$$
Let $s \ge 2$. If $f, g, k: \mathbb{Z}/N\mathbb{Z} \to \mathbb{C}$ are one-bounded and
$$\|\mathcal{D}(f, g, k)\|_{U^{s + 1}(\mathbb{Z}/N\mathbb{Z})} \ge \delta,$$
then either $\delta \ll \exp(-\log^{1/O_{P, Q}(1)}(N))$ or $\|f\|_{U^s(\mathbb{Z}/N\mathbb{Z})} \gg \exp(-\log^{O_{P, Q}}(1/\delta))$.
\end{proposition}
First, assuming this is true, by \cite[Lemma 5.1, Lemma 5.2]{Len22}, for some $s = s_{P, Q}$, we have
$$|\Lambda(f, g, k, p)| \le \|p\|_{U^s(\mathbb{Z}/N\mathbb{Z})} + O(N^{-1/O_{P, Q}(1)}).$$
Thus, 
$$|\Lambda(f, g, k, p)| \le \left(\mathbb{E}_x |\mathcal{D}(f, g, k)|^2\right)^{1/2} = |\Lambda(\overline{f}, \overline{g}, \overline{k}, \mathcal{D}(f, g, k))|^{1/2} \le \|\mathcal{D}(f, g, k)\|_{U^s(\mathbb{Z}/N\mathbb{Z})}^{O(1)} + O(N^{-\zeta})$$
for some $\zeta$ that depends only on $P$ and $Q$. Thus, if we can show \thref{degreelowering}, this will in turn imply by an iterative argument \thref{polynomialinverse}. For the remainder of the argument, we will now indicate how to make improvements in \cite{Len22}. The first improvement we apply is the Sanders $U^3$ inverse theorem (see \cite[Appendix A]{Len22b} for how to deduce the improved $U^3$ inverse theorem from \cite{San12b}), where we end up with
\begin{equation}\label{u3application}
\mathbb{E}_{h_1, \dots, h_{s - 2}} |\langle \mathcal{D}_h(f, g, k), F_{\vec{h}}(m_h(x)\Gamma) \rangle| \gg \exp(-\log^{O(1)}(1/\delta)) 
\end{equation}
instead of the inferior
$$\mathbb{E}_{h_1, \dots, h_{s - 2}} |\langle \mathcal{D}_h(f, g, k), F_{\vec{h}}(m_h(x)\Gamma) \rangle|^8 \gg \exp(-\delta^{-O(1)}).$$
We will then set $\epsilon = \exp(-\log^{O(1)}(1/\delta))$ and continue as usual in the argument until \cite[Lemma 6.1]{Len22}. We now highlight the improvement to \cite{Len22}. Instead of invoking \cite[Lemma 6.1]{Len22}, we invoke the following quantitative improvement to \cite[Lemma 6.1]{Len22}:
\begin{lemma}
Let $\delta \in (0, 1/10)$, $N > 100$ be prime, and $G/\Gamma$ be a two-step nilmanifold of dimension $d$, complexity $M$, and degree $k$ equipped with the standard filtration and let $g(n) \in \mathrm{poly}(\mathbb{Z}, G)$ with $g(n)\Gamma$ periodic modulo $N$. Furthermore, let $F_1, F_2, F_3$ be $1$-Lipschitz functions on $G/\Gamma$ with the same nonzero frequency $\xi$. If
$$|\mathbb{E}_{n \in [N]} F_1(g(P(n))\Gamma)F_2(g(Q(n))\Gamma)\overline{F_3(g(P(n) + Q(n))\Gamma)}e(\alpha P(n) + \beta Q(n))| \ge \delta$$
for some frequencies $\alpha, \beta \in \widehat{\mathbb{Z}/N\mathbb{Z}}$, then either $N \ll (\delta/M)^{-O_{P, Q}(d^{O_{P, Q}(1)})}$ or there exists $w_1, \dots, w_r \in \Gamma$ with $\psi_{\mathrm{horiz}}(w_1), \dots, \psi_{\mathrm{horiz}}(w_r)$ linearly independent and linearly independent horizontal characters $\eta_1, \dots, \eta_{d - 1 - r}$ such that $|w_i|, |\eta_j| \le (\delta/M)^{-O_{P, Q}(d^{O_{P, Q}(1)})}$, $\langle \eta_j, w_i \rangle = 0$ for all $i, j$, and
$$\|\xi([w_i, g])\|_{C^\infty[N]} = \|\eta_j \circ g\|_{C^\infty[N]} = 0.$$
\end{lemma}
\begin{proof}
The first part of the argument is similar to the first part of the argument in \cite[Lemma 6.1]{Len22}. We first use the fact that $F$ is a nilcharacter of nonzero frequency to absorb $\alpha P(n)$ and $\beta Q(n)$ to the vertical component of $g(P(n))$ and $g(Q(n))$, respectively to obtain $g_1(P(n))$ and $g_2(Q(n))$. Since the conclusion only depends on the horizontal component of $g$, and since the horizontal component of $g_1$ and $g_2$ agree with $g$, it follows that we can assume that both $g_1$ and $g_2$ are $g$ and $\alpha, \beta = 0$. Let $H$ denote the subgroup of $G^3$ consisting of elements $\{(g_1, g_2, g_3): g_1g_2g_3^{-1} \in [G, G]\}$. We claim that $[H, H] = [G, G]^3$. By definition, we see that for $h \in [G, G]$ that $(1, h, h)$ and $(h, 1, h)$, and $(h, h, h^4)$ lies inside $[H, H]$ (the last fact is true because $[(g_1, g_1, g_1^2), (h_1, h_1, h_1^2)] = ([g_1, h_1], [g_1, h_1], [g_1, h_1]^4)$). This yields that $(1, 1, h^2)$ lies inside $[H, H]$, and because of connectedness and simple connectedness, it follows that $(1, 1, h) \in [H, H]$. We can verify from there that $[H, H] = [G, G]^3$. \\\\ 
We were given the polynomial sequence 
$$(g(P(n)), g(Q(n)), g(P(n) + Q(n)))$$
on $H$. Since $F_i$ are nilcharacters of frequency $\xi$ on $H$, $F_1 \otimes F_2 \otimes F_3$ is a nilcharacter on $H$ of frequency $(\xi, \xi, -\xi)$. Taking a quotient of $H$ by the kernel of $(\xi, \xi, -\xi)$, which is $(x, y, x + y)$, we obtain that the center is of the form $(x, x, -x)$ with $(x, y, z)$ being projected to $(x + y - z)(1, 1, 1)$. Let $H_1$ denote the subgroup with the one dimensional vertical directions. Applying \thref{twosteppolynomial}, we obtain $w_1, \dots, w_r$ and $\eta_1, \dots, \eta_{d - r}$ such that $\langle w_i, \eta_j \rangle = 0$ and $\eta_j \circ (g(P(n)), g(Q(n)), g(P(n) + Q(n))) \equiv 0 \pmod{1}$, and $\xi([w_i, g(P(n)), g(Q(n)), g(P(n) + Q(n))]) \equiv 0 \pmod{1}$. Denoting $\eta_j = (\alpha_j, \beta_j)$ and $w_i = (u_i, v_i, u_iv_i)$ and the action $\eta_j(w_i) := \alpha_j(u_i) + \beta_j(v_i)$, we see that 
$$\|\xi([v_i, g(P(n))]) + \xi([u_i, g(Q(n))])\|_{C^\infty[N]} \equiv 0 \pmod{1}$$
$$\|\alpha_j(g(P(n))) + \beta_j(g(Q(n)))\|_{C^\infty[N]} \equiv 0 \pmod{1}.$$
Since $P$ and $Q$ are linearly independent, it follows that there exists some coefficients $c_k x^k$, $c_\ell x^\ell$ of $P$, and $d_k x^k$ and $d_\ell x^\ell$ of $Q$ such that $c_k d_\ell - d_k c_\ell \neq 0$. Thus, the conditions become
$$c_k\xi([u_i, g(1)]) + d_k\xi([v_i, g(1)]) \equiv 0 \pmod{1}$$
$$c_\ell\xi([u_i, g(1)]) + d_\ell\xi([v_i, g(1)]) \equiv 0 \pmod{1}$$
which implies since $\alpha$ has denominator $N$ which is prime that $\xi([u_i, g(1)]) \equiv 0\pmod{1}$ and $\xi([v_i, g(1)]) \equiv 0 \pmod{1}$. Similarly, we have $\alpha_j(g(1)) \equiv 0 \pmod{1}$ and $\beta_j(g(1)) \equiv 0 \pmod{1}$. Let $\tilde{G} := \{g \in G: \xi([v_i, g]) = 0, \xi([u_i, g]) = 0, \alpha_j(g) = 0, \beta_j(g) = 0 \forall i, j\}$. We claim that $\tilde{G}$ is abelian, from whence the Lemma would follow from an application of \cite[Lemma A.9]{Len2}. This amounts to showing that for any $g, h \in \tilde{G}$ that $[g, h] = \mathrm{id}_G$. For such $g$, $(g, g)$ is annihilated by $(\alpha_j, \beta_j)$, and since $\alpha_j(u_i) + \beta_j(v_i) = 0$, it follows that $(g, g)$ can be written as a combination of $(u_i, v_i)$ modulo $[G, G]^2$. It follows that $[(g, g), (h, h)] = \mathrm{id}_G$, and thus $[g, h] = \mathrm{id}_G$. 
\end{proof}
to obtain that the image of $m_h$ under the kernel of $\xi$ lies in an abelian subnilmanifold of rationality at most $\epsilon^{-O(r)^{O(1)}}$ where this time $r = \log^{O(1)}(1/\epsilon)$. We can then Fourier expand $F_{\vec{h}}(m_h(x)\Gamma)$ in \ref{u3application} as in the argument after \cite[Lemma 6.1]{Len22} to eventually obtain
$$\|p\|_{U^s(\mathbb{Z}/N\mathbb{Z})} \gg \epsilon^{O(r)^{O(1)}}$$
which gives the desired estimate for \thref{degreelowering}. This completes the proof of \thref{asymptotic}.


\section{The $U^4(\mathbb{Z}/N\mathbb{Z})$ inverse theorem}
We now prove \thref{mainresult4}. We restate it for the reader's convenience.
\mainresultc*
The hypothesis implies that
$$\mathbb{E}_h \|\Delta_h f\|_{U^3}^8 \ge \delta^{16}.$$
An application of the inverse theorem of Sanders \cite{San12b} combined with the argument of \cite[Theorem 10.9]{GT08b} (see also \cite[Appendix A]{Len22b}) gives the following:
\begin{theorem}\thlabel{bracketpolynomialu3}
There exists some real number $c > 0$ with the following property: if $c > \eta > 0$ and $f\colon \mathbb{Z}/N\mathbb{Z} \to \mathbb{C}$ be one-bounded with
$$\|f\|_{U^3(\mathbb{Z}/N\mathbb{Z})} \ge \eta.$$
Then there exists a constant $C > 0$, a subset $S \subseteq \widehat{\mathbb{Z}/N\mathbb{Z}}$ with $|S| \le \log(1/\eta)^{C}$ and a phase $\phi\colon \mathbb{Z}/N\mathbb{Z} \to \mathbb{R}$ such that
$$\phi(n) = \sum_{\alpha, \beta \in S} a_{\alpha, \beta} \{\alpha \cdot n\}\{\beta \cdot n\} + \sum_{\alpha \in S} a_\alpha \{\alpha \cdot n\}$$
with $\alpha_i, \beta_i \in S$ and $a_i \in \mathbb{R}$ and
$$|\mathbb{E}_{n \in [N]} f(n)e(\phi(n))| \ge \exp(-\log(1/\eta)^C).$$
\end{theorem}
This implies that for a family of degree two periodic bracket polynomials $\chi_h(n)$ with at most $\log(1/\delta)^{O(1)}$ many bracketed phases that for a subset $H \subseteq \mathbb{Z}_N$ with $|H| \ge \delta^{O(1)}N$ that for any $h \in H$,
$$|\mathbb{E}_n \Delta_h f(n)\chi_h(n)| \ge \delta^{\log(1/\delta)^{O(1)}}.$$
We will now use \cite[Proposition 6.1]{GTZ11}:
\begin{proposition}\thlabel{additivequadruples}
Let $f_1, f_2 \colon \mathbb{Z}/N\mathbb{Z} \to \mathbb{C}$ be one-bounded and $H \subseteq \mathbb{Z}/N\mathbb{Z}$ a set of cardinality $\eta N$ such that for each $h \in H$,
$$|\mathbb{E}_{n \in \mathbb{Z}/N\mathbb{Z}} f_1(n)f_2(n + h)\chi_h(n)| \ge \delta.$$
Then for at least $\eta^8 \delta^4 N^3/2$ many quadruples $(h_1, h_2, h_3, h_4)$ in $H$ satisfying $h_1 + h_2 = h_3 + h_4$,
$$|\mathbb{E}_{n \in \mathbb{Z}/N\mathbb{Z}} \chi_{h_1}(n)\chi_{h_2}(n + h_1 - h_4)\overline{\chi_{h_3}(n)\chi_{h_4}(n + h_1 - h_4)}| \gg \eta^4 \delta^2.$$
\end{proposition}
\begin{proof}
We follow \cite[Proposition 6.1]{GTZ11}. We extend $\chi_h$ to be zero for $h \not\in H$. The condition implies that
$$\mathbb{E}_{h} 1_H(h)|\mathbb{E}_n f_1(n)f_2(n + h)\chi_h(n)|^2 \gg \delta^2\eta.$$
Expanding out, we obtain
$$\mathbb{E}_{h} 1_H(h) \mathbb{E}_{n, n'} f_1(n)\overline{f_1(n')}f_2(n + h)\overline{f_2(n' + h)} \chi_h(n)\overline{\chi_h(n')} \gg \delta^2 \eta.$$
Making a change of variables $h = m - n$, $n' = n + k$, we obtain
$$\mathbb{E}_{m, n, k} 1_H(m - n)f_1(n)\overline{f_1(n + k)} f_2(m)\overline{f_2}(m + k) \Delta_{k}\chi_{m - n}(n) \gg \delta^2 \eta.$$
Applying Cauchy-Schwarz twice, we obtain
$$\mathbb{E}_{k, m, n, m', n'} 1_H(m - n)\Delta_k \chi_{m - n}(n)$$
$$\overline{1_H(m' - n)\Delta_k \chi_{m' - n}(n)1_H(m - n') \Delta_k \chi_{m - n'}(n')} 1_H(m' - n')\Delta_k \chi_{m' - n'}(n') \gg \eta^4\delta^8.$$
This can be rewritten as
$$\mathbb{E}_{h_1 + h_2 = h_3 + h_4} |1_H(h_1)\chi_{h_1}(n)1_H(h_2)\chi_{h_2}(n + h_1 - h_4)\overline{1_H(h_3)\chi_{h_3}(n)1_H(h_4)\chi_{h_4}(n + h_1 - h_4)}|^2 \gg \eta^4\delta^8$$
as desired.
\end{proof}
We thus have that for at least $\delta^{O(\log(1/\delta))^{O(1)}} N^3$ many additive quadruples, that is quadruples $(h_1, h_2, h_3, h_4)$ with $h_1 + h_2 = h_3 + h_4$, we have
$$|\mathbb{E}_n \chi_{h_1}(n)\chi_{h_2}(n + h_1 - h_4)\overline{\chi_{h_3}(n)\chi_{h_4}(n + h_1 - h_4)}| \ge \delta^{\log(1/\delta)^{O(1)}}.$$
The remainder of this section is devoted to the following.
\begin{itemize}
\item[1.] Defining nilcharacters and giving constructions of three-step nilpotent Lie groups. We will show that we can take $\chi_h(n) = F(g_h(n)\Gamma)$ for some \emph{Fourier expanded nilcharacter} $F(g_h(n)\Gamma)$ (to be defined in Section 8.1). This occupies Sections 8.1 and 8.2
\item[2.] Analyzing the above inequality to glean structure out of $\chi_h$ on some dense subset of $\mathbb{Z}_N$. This will involve a ``sunflower-type decomposition", and a ``linearization argument." These are analogous to \cite[Section 7, Section 8]{GTZ11}. This argument occupies Sections 8.3 and 8.4 and is almost all of the improvement over \cite{GTZ11}.
\item[3.] Insert this structure to deduce the inverse theorem. This will comprise of the ``symmetry and integration steps." This is analogous to \cite[Section 9]{GTZ11} and we essentially follow their argument. This argument will take place in Section 8.5.
\end{itemize}
Before we proceed, we shall specify some notation. 
\begin{definition}[Lower order terms]\thlabel{lowerorderterms}
The quantity $[\text{Lower order terms}]$ denotes a sum of $d^{O(1)}$ quantities of the following form:
\begin{itemize}
    \item $a\{\alpha_h n\}\{\beta_h n\}$ or $a\{\alpha h\}\{\beta n\}$ or $a\{\alpha h_i\}\{\beta n\}$
    \item $a_h\{\alpha_h n\}$ or $a_{h_i}\{\alpha_{h_i} n\}$
    \item $\alpha h$
\end{itemize}
where all instances of $\alpha, \beta \in \mathbb{R}$ are rational with denominator $N$ and $|a| \le \exp(\log(1/\delta)^{O(1)})$.    
\end{definition}
The point of lower order terms is that they can be eliminated via \thref{onevarfouriercomplexity} and \thref{bilinearfouriercomplexity}. We will also use the following shorthand.
\begin{definition}[Equal up to lower order terms]\thlabel{shorthandequal}
We say that $a \equiv b$ if $a = be([\text{Lower order terms}])$.
\end{definition}

\subsection{Bracket polynomials and Fourier expanded nilcharacters}
In this section, we give precise definitions for notation for bracket polynomials we work with. We refer the reader to \cite[Section 2]{Len2} for various notions of Mal'cev bases.
\begin{definition}[Periodic fourier expanded nilcharacter]
Given a degree two two-step nilmanifold $G/\Gamma$ with one-dimensional vertical torus, we define a \emph{Fourier expanded nilcharacter}, which we will shorten as \emph{nilcharacter} on $G/\Gamma$ as follows. Write $\mathcal{X}= \{X_1, \dots, X_{d - 1}, Y\}$ as the Mal'cev basis for $G/\Gamma$. Then letting $g(n)$ a polynomial sequence on $G/\Gamma$ with $g(n)\Gamma$ periodic modulo $N$ and $\psi(g(n)) = (\alpha_1 n, \dots, \alpha_{d - 1}n,P(n))$ we let\footnote{We note by calculations done in \cite[Appendix B]{GT08} that this is indeed a function on $G/\Gamma$.}
$$F(g(n)\Gamma) = e(-k \sum_{i < j} C_{[i, j]}\alpha_i n[\alpha_j n] + k P(n))$$
where $C_{[i, j]}Y = [X_i, X_j]$ with $C_{[i, j]}$ an integer bounded by $Q$. We denote $k$ to be the \emph{frequency} of $F$. We define $\omega = k[\cdot, \cdot]$ as the \emph{associated asymmetric bilinear form} on $\text{Span}(X_1, \dots, X_{d - 1})$ to $F$.    
\end{definition}
As the proof will perform many ``change of bases", we require the following lemma.
\begin{lemma}\thlabel{changeofvar}
Let $Q \ge 2$ and $G/\Gamma$ be a two-step nilmanifold with Malcev basis $\mathcal{X} = \{X_1, \dots, X_{d - 1}, Y\}$ of complexity $Q$. Let $\mathcal{X}' = \{X_1', \dots, X_{d - 1}', Y\}$ be a basis of $\mathfrak{g}$ with $X_i'$ a $Q$-integer combination of elements in $\mathcal{X}$. Then the following hold.
\begin{itemize}
    \item Letting $\tilde{\Gamma} = \{\exp(t_1X_1')\exp(t_2X_2') \cdots \exp(t_dX_d') \exp(s Y): t_1, \dots, t_d, s \in \mathbb{Z}\}$, $G/\tilde{\Gamma}$ is a nilmanifold equipped with a Mal'cev basis $\mathcal{X}'$ of complexity $Q^{O(d^{O(1)})}$.
    \item If $F(g(n)\Gamma)$ is a Fourier-expanded nilcharacter on $G/\Gamma$, then there exists a periodic Fourier-expanded nilcharacter $\tilde{F}(\tilde{g}(n)\tilde{\Gamma})$ $F(g(2n)\Gamma) = F(\tilde{g}(n)\tilde{\Gamma})e(\text{[Lower order terms]})$ with frequency $k = 4$.
\end{itemize}
\end{lemma}
\begin{proof}
We first show the first item. We first note that since $\mathcal{X}$ is a Mal'cev basis, $[\exp(X_i), \exp(X_j)] = \exp([X_i, X_j]) \in \exp(\mathbb{Z} Y)$, so the structure constants of the Mal'cev basis $\mathcal{X}$ are integers. This implies that $\exp([X_i', X_j']) \in \exp(\mathbb{Z} Y)$ so $\tilde{\Gamma}$ is a group. Since we may find a bounded fundamental domain for $G/\tilde{\Gamma}$, it follows by Bolzano-Weierstrass and the Whitney embedding theorem that it is compact. Also, by Cramer's rule, $\mathcal{X}'$ has complexity $Q^{O(d^{O(1)})}$. \\\\
We now turn to the second item. Let 
$$\psi_{\mathcal{X}}(g(n)) = (\alpha_1 n, \dots, \alpha_d n, P(n))$$
and
$$\psi_{\mathcal{X}'}(g(n)) = (\tilde{\alpha}_1 n, \dots, \tilde{\alpha}_{d - 1}n, P(n))$$
Note that we have
$$\sum_i \tilde{\alpha_i} X_i' = \sum_i \tilde{\alpha_i} \sum_j a_{ij} X_j = \sum_j \left(\sum_i a_{ij}\tilde{\alpha_i}\right)X_j$$
with $a_{ij}$ integers at most $Q$,
so
$$\alpha_j = \sum_i a_{ij} \tilde{\alpha_i}.$$
In addition, we have (for a vector-valued $\text{[Lower order terms]}$)
$$\sum_i \tilde{\alpha_i} X_i' = \sum_j \alpha_j X_j$$
and
$$\sum_i [\tilde{\alpha_i} n] X_i' = \sum_i [\tilde{\alpha_i}n] \sum_j a_{ij} X_j = \sum_j [\alpha_j n] X_j + [\text{Lower order terms}] \cdot \mathcal{X}.$$
By the identity
$$\alpha n[\beta n] - \beta n[\alpha n] \equiv \alpha \beta n^2 + 2\alpha n[\beta n] \pmod{1}$$
it follows that there exists quadratic polynomials $P_1$ and $P_2$ such that
$$F(g(n)\Gamma) = e(-k/2\sum_{i, j} [\sum_i \alpha_i n X_i, \sum_i [\alpha_i n]X_i] + P_1(n) + [\text{Lower order terms}])$$
and
$$F(g(n)\Gamma) = e(-k/2\sum_{i, j} [\sum_i \tilde{\alpha_i} nX_i', \sum_i [\tilde{\alpha_i} n] X_i'] + P_2(n) + [\text{Lower order terms}]).$$
It follows that letting $Q = P_2 - P_1$ and
$$\psi_{\mathcal{X}'}(\tilde{g}(n)) = (\tilde{\alpha}_1 n, \dots, \tilde{\alpha}_d n, Q(n)),$$
we obtain (the purpose of working with $2n$ being that it cancels out with the factor of $\frac{1}{2}$ present in $k/2$)
$$F(g(2n)) = F(\tilde{g}(n)\tilde{\Gamma})e(\text{[Lower order terms]}).$$
By considering a modular inverse $m$ of $2$, and modifying $\tilde{g}(n)$ to be
$$((\tilde{\alpha}_1 2m)n, \dots, (\tilde{\alpha}_d 2m)n, Q(2mn))$$
it follows that $Q(2m \cdot)$ is divisble by $2$, and since the leading coefficient of the bracket part of $\tilde{F}(\tilde{g})$ is divisible by $4$, it follows that we can take our frequency to be $k = 4$.
\end{proof}
Finally, we will need a lemma that converts a bracket polynomial to a Fourier expanded nilcharacter.
\begin{lemma}\thlabel{u3fourierexpandednilcharacter}
Let $\alpha_1, \dots, \alpha_d, \beta_1, \dots, \beta_d \in \mathbb{R}$ be rationals with denominator $N$ and define a function $\phi\colon \mathbb{Z}/N\mathbb{Z} \to \mathbb{R}$ via
$$\phi(n) = \sum_{i = 1}^d a_{\alpha, \beta} \{\alpha \cdot n\}\{\beta \cdot n\} + \sum_{\alpha \in S} a_\alpha \{\alpha \cdot n\}.$$
Then there exists a Fourier expanded nilcharacter $F(g(n)\Gamma)$ of complexity $2$ and frequency $1$ such that
$$e(\phi(n)) = F(g(n)\Gamma) e([\text{Lower order terms}]).$$
\end{lemma}
\begin{proof}
An application of the $U^3$ inverse theorem gives us a bracket polynomial of the form
$$\sum_i a_i\{\alpha_i n\}\{\beta_i n\}.$$
We write
$$a_i\{\alpha_i n\}\{\beta_i n\} = [a_i](\alpha_i n - [\alpha_i n])(\beta_i n - [\beta_i n]) + \{a_i\}\{\alpha_i n\}\{\beta_i n\}.$$
Defining 
$$G = \begin{pmatrix} 1 & \mathbb{R} & \cdots & \cdots  & \mathbb{R} \\ 0 &1 & \cdots & 0 & \mathbb{R} \\ 0 & 0 & 1 & \cdots & \mathbb{R} \\
0 & 0 & 0 & \ddots & \vdots \\
0 & 0& 0 & \cdots  & 1 \end{pmatrix},$$
$$\Gamma  = \begin{pmatrix} 1 & \mathbb{Z} & \cdots & \cdots  & \mathbb{Z} \\ 0 &1 & \cdots & 0 & \mathbb{Z} \\ 0 & 0 & 1 & \cdots & \mathbb{Z} \\
0 & 0 & 0 & \ddots & \vdots \\
0 & 0& 0 & \cdots  & 1 \end{pmatrix},$$
with $G$ $2d + 1$ dimensional and $X_1, \dots, X_d$ represent the coordinates (from left to right) of the first row and $X_{d + 1}, \dots, X_{2d}$ the coordinates (from top to bottom), and $Y$ the last coordinate, we see that a bracket polynomial of the form $[a_i](\alpha_i n - [\alpha_i n])(\beta_i n - [\beta_i n]) + \{a_i\}\{\alpha_i n\}\{\beta_i n\}$ can be realized (up to lower order terms periodic in $N$) as a Fourier expanded nilcharacter on $G/\Gamma$ with $d$ being proportional to the number of brackets in the sum.
\end{proof}
Thus, we may assume that $\chi_h(n) = F(g_h(n)\Gamma)$ for some Fourier expanded nilcharacter of frequency $4$.

\subsection{Three-step nilmanifold constructions}

This section is meant to give explicit constructions of (approximate) nilsequences certain degree $3$ bracket polynomials. It is meant to be skimmed on first reading. 
\begin{lemma}
Let $\alpha_1, \dots, \alpha_k, \beta_1, \dots, \beta_k, \gamma_1, \dots, \gamma_k, \alpha_1', \dots, \alpha_k', \beta_1', \dots, \beta_k' \in \mathbb{R}$. Consider
$$e(-\sum_{j = 1}^k\alpha_j n\{\beta_j n\}\{\gamma_j n\} - \sum_{j = 1}^\ell \alpha_j' n^2\{\beta_j' n\}).$$
There exists a nilmanifold $G/\Gamma$ of degree $3$, complexity $O(1)$, and dimension $O(k)$, and an approximate nilsequence $F(g(n)\Gamma)$ with $F$ $O(d)$-Lipschitz outside of the boundary of the standard fundamental domain $\psi^{-1}((-1/2, 1/2]^d)$ of $G/\Gamma$ such that
$$F(g(n)\Gamma) = e(-\sum_{j = 1}^k\alpha_j n\{\beta_j n\}\{\gamma_j n\} - \sum_{j = 1}^\ell \alpha_j' n^2\{\beta_j' n\}).$$
\end{lemma}
\begin{proof}
We give two proofs. The first will follow\cite{GTZ11}. We let the free $3$-step Lie group:
$$G = \{e_1^{t_1}e_2^{t_2} e_3^{t_3} e_{21}^{t_{21}}e_{211}^{t_{211}}e_{31}^{t_{31}}e_{311}^{t_{311}}e_{32}^{t_{32}}e_{322}^{t_{322}}e_{212}^{t_{212}}e_{312}^{t_{312}}e_{213}^{t_{213}}e_{313}^{t_{313}} e_{323}^{t_{323}} \}$$
where the $t_i, t_{ij}, t_{ijk}$ range over $\mathbb{R}$ and with the relations $[e_i, e_j] = e_{ij}$ and $[[e_i, e_j], e_k] = e_{ijk}$ and the Jacobi identity
$$[[e_i, e_j], e_k][[e_j, e_k], e_i][[e_k, e_i], e_j] = 1.$$
We also take the lattice $\Gamma$ to be the subgroup of the above group where the $t_i, t_{ij}, t_{ijk}$ range over $\mathbb{Z}$. Letting $s_i, s_{ij}, s_{ijk}$ be the coordinates in the fundamental domain of $G/\Gamma$, we see that
$$s_i = \{t_i\}, s_{ij} = \{t_{ij} - t_i[t_j]\}, s_{ijk} = \{t_{ijk} - t_{ik}[t_j] - t_{ij}[t_i] + t_i[t_j][t_k]\}.$$
Thus, this expresses $e_1^{\alpha n} e_2^{\beta n}e_3^{\gamma n}$ as $e([\alpha n][\beta n]\gamma n)$ with a bunch of lower order terms that we can take care of using the construction from earlier parts. \\\\
The second proof follows \cite{GTZ12}. Let $G_1$ be the subgroup of the above matrix group we constructed which is generated by the elements where only the first (top) row is nonzero. This can be generated by elements $x_i$, which correspond to elements where only the $i$th element (from the left) of the first row is nonzero, and $z$ which represents the rightmost coordinate of the first row. Let $y_i$ be the $i + 1$th element (starting from the top) of the last column (from the left). The approximate nilsequence we desire to replicate is thus
$$g(n) = x_1^{\alpha_1 n\{\beta_1 n\}} \cdots x_k^{\alpha_1 n\{\beta_1 n\}} x_{k + 1}^{\alpha_{1}' n^2} \cdots x_{k + \ell}^{\alpha_\ell' n^2}$$
$$ y_1^{\gamma_1 n} \cdots y_k^{\gamma_k n} y_{k + 1}^{\beta_1' n}\cdots y_{k + \ell}^{\beta_\ell' n}$$
and $F(x\Gamma) = e(\{z - x[y]\})$. To define this to be a three-step nilsequence, we define a group $H = \langle x_1, x_2, \dots, x_k, z \rangle$ and we consider the semidirect product $G \ltimes H$ with the action being conjugation. We define the action $\rho^k$ of $\mathbb{R}^k$ on $G \ltimes H$ by $\rho(t)(g, g_1) := (gg_1^t, g_1)$, and finally consider $G' = \mathbb{R}^k \rtimes_{\rho^k} (G \ltimes H)$. One can check that this is a three-step nilpotent Lie group with respect to the lower central series. We now define $\Gamma' = \mathbb{Z} \rtimes_{\rho^k} (\Gamma \ltimes (H \cap \Gamma))$ and consider the polynomial sequence
$$g'(n) = (0, (x_{k + 1}^{\alpha_1'n^2} \cdots x_{k + \ell}^{\alpha_\ell' n^2}y_1^{\gamma_1 n} \cdots y_k^{\gamma_k n} \cdots y_{k + \ell}^{\beta_{\ell}'n},  $$
$$x_1^{\alpha_1 n} \cdots x_k^{\alpha_k n}), (\beta_1 n, \beta_2 n, \dots, \beta_k n), (id, id)).$$
and the function
$$F'((t, (g, g'))\Gamma') = F(g\Gamma).$$
It follows then that
$$g'(n)\Gamma' = (\{\beta_1 n\}, \dots, \{\beta_k n\}, (g(n), x_1^{\alpha_1 n} \cdots x_k^{\alpha_k n}))\Gamma'$$
so that
$$F'(g'(n)\Gamma') = e(-\sum_{j}\alpha_j n\{\beta_j n\}[\gamma_j n] - \sum_{j} \alpha_j' n^2[\beta_j' n]).$$
Finally, to complete the nilmanifold construction, we use $[\gamma_j n] = \gamma_j n - \{\gamma_j n\}$ and $[\beta_j' n] = \beta_j' n - \{\beta_j' n\}$ and apply a very similar nilmanifold construction to the remaining terms we get under this expansion.
\end{proof}

\subsection{Sunflower type decomposition}
We will now deduce one of the primary improvements to \cite{GTZ11}. This corresponds to \cite[Section 7]{GTZ11}. \\\\
\subsubsection{A few reductions}
We now make a set of reductions. Our first claim is that
$$\chi_{h_1}(n)\chi_{h_2}(n + h_1 - h_4)\overline{\chi_{h_3}(n)\chi_{h_4}(n + h_1 - h_4)}$$
is a Fourier expanded nilcharacter. To see this, by standard bracket polynomial manipulations (see e.g., \cite{GTZ11} and Section 3), we have
$$\chi_{h_2}(n + h_1 - h_4)\overline{\chi_{h_4}(n + h_1 - h_4)} \equiv \chi_{h_2}(n)\overline{\chi_{h_4}(n)}e(\alpha_{h_{1},h_2, h_3, h_4} n^2 + \beta_{h_1, h_2, h_3, h_4}n + f(h)).$$
Next, we observe that by tensoring two free nilcharacters together, as in adjoining them $n \mapsto \chi_1(n)\chi_2(n)$ forms a free nilcharacter as well by taking a product of the corresponding nilsequences $F_1(g_1(n)\Gamma_1)F_2(g_2(n)\Gamma_2) = \tilde{F}(\tilde{g}(n)\tilde{\Gamma})$ with $\tilde{F}(x, y) = F_1(x)F(y)$, $\tilde{G} = G_1 \times G_2/\text{ker}(Y_1 + Y_2)$ where $Y_i$ is the vertical component of $G_i$. Let $\pi: G_1 \times G_2 \to \tilde{G}$ be the quotient map. We note that induced asymmetric bilinear form on $\tilde{G}$ is $k_1 \omega_1 + k_2 \omega_2$ where $k_i$ is the frequency of $F_i$. We thus have $\tilde{g} = \pi(g_1(n), g_2(n))$ and $\tilde{\Gamma} = \pi(\Gamma_1 \times \Gamma_2)$. Finally, we observe that modifying the quadratic term does not change whether or not a function is a free nilcharacter. Hence, we've realized
$$\chi_{h_1}(n)\chi_{h_2}(n + h_1 - h_4)\overline{\chi_{h_3}(n)\chi_{h_4}(n + h_1 - h_4)}$$
as a free nilcharacter on $G^{(4)} := G^4/\mathrm{ker}(Y_1 + Y_2 - Y_3 - Y_4)$.
\subsubsection{Sunflower iteration}
If $\chi_h(n) = F(g_h(n)\Gamma)$, we define $\tilde{\chi_h}(n) = F(g_h(n)\Gamma)\psi(g_h(n)\Gamma)$ where $\psi$ is a smooth cutoff function on the horizontal torus $\mathbb{R}^{d}/\mathbb{Z}^{d} \cong [-1/2, 1/2)^{d}$ supported on $[-1/2 + \delta^{O(d)^{O(1)}}, 1/2 - \delta^{O(d)^{O(1)}}]^{d}$, it follows that $$\|\chi_h - \tilde{\chi}_h\|_{L^1[N]} \le \delta^{O(d^{O(1)})},$$
so we may replace $\chi_h$ with $\tilde{\chi}_h$ at the cost of shrinking the right hand side of the inequality by a bit (but will still appear as $\delta^{O(d)^{O(1)}}$). \\\\
For simplicity, we shall relabel the $\delta^{\log(1/\delta)^{O(1)}}$ as just $\delta$ and $\tilde{\chi}_h$ as simply $\chi_h$ for the remainder of this section. Similarly to \cite{GTZ11}, for some subset $H \subseteq \mathbb{Z}_N$, $h \in H$ we let
\begin{itemize}
    \item $\Xi$ will denote the phases of the Fourier expanded nilcharacter;
    \item $\Xi_h$ be the $h$-dependent phases of the Fourier expanded nilcharacter;
    \item $\Xi_{*}$ be the $h$-independent phases as $h$;
    \item $\Xi_h^G$ be the subgroup in $G$ generated by the $h$-dependent dimensions;
    \item and $\Xi_*^G$ the subgroup of $G$ generated by the $h$-independent dimensions.
\end{itemize}
Given a nilmanifold $G/\Gamma$ with Mal'cev basis $(X_1, \dots, X_{d - 1}, Y)$ for the lower central series filtration with one-dimensional vertical component, and a polynomial sequence of the form $\psi(g(n)) = (\alpha_1n, \dots, \alpha_{d - 1} n, P(n))$, we denote
$$\psi_{horiz}(g(n)) = \sum_i \alpha_i X_i.$$
We shall also write
$$\psi_{horiz}(g) \equiv \sum_i \beta_i X_i \pmod{\Gamma}$$
to signify that $\psi_{horiz}(g)$ is equal to $\sum_i\beta_i X_i$ up to an element in the $\mathbb{Z}$-span of $(X_1, \dots, X_{d - 1})$. We are finally ready to state the sunflower step.
\begin{lemma}\thlabel{sunflower}
Suppose for $\delta |H|^3$ many additive quadruples $(h_1, h_2, h_3, h_4)$ inside $H^4$ with $|H| \ge \delta N$ that
$$|\mathbb{E}_{n \in [N]} \chi_{h_1}(n)\chi_{h_2}(n + h_1 - h_4)\overline{\chi_{h_3}(n)\chi_{h_4}(n + h_1 - h_4)}| \ge \delta.$$
Then there exists a subset $H' \subseteq H$ such that $|H'| \ge \delta^{O(d)^{O(1)}}|H|$ and for each $h \in H'$, there exists some $r > 0$ and fixed $\alpha_1, \dots, \alpha_r$ such that denoting $\psi_{horiz}(G) = \mathrm{span}(X_1, \dots, X_k, Z_1, \dots, Z_\ell)$ the Mal'cev basis with complexity at most $\delta^{-O(d)^{O(1)}}$, there exists some integer $|q| \le \delta^{-O(d)^{O(1)}}$ and some ordering of the Mal'cev basis $(X_1, \dots, X_k, Z_1, \dots, Z_\ell)$ such that denoting
$$\psi_{horiz}(g_h) = \sum_i \xi_h^i X_i + \sum_j \xi_*^j Z_j$$,
we have
$$g_h(q \cdot) \equiv \alpha_1 X_1 +  \cdots + \alpha_r X_r + \xi_h^{r + 1}Y_{r + 1} + \cdots + \xi_h^{k}Y_{k} + \xi_*^1Z_1 + \cdots + \xi_*^\ell Z_\ell \pmod{\Gamma}.$$
where $X_1, \dots, X_r, Y_{r + 1}, \dots, Y_{k}, Z_1, \dots, Z_\ell$ are linearly independent and each $Y_j$ is a $\delta^{-O(d)^{O(1)}}$-rational combination of $X_i$'s and we have that the asymmetric bilinear form $\omega$ restricted to $\mathrm{span}(Y_1, \dots, Y_k)$ vanishes.
\end{lemma}
\begin{remark}
This lemma will be used (following the procedure in Section 7.2 to change bases $(X_1, \dots, X_k, Z_1, \dots, Z_\ell)$ to $(X_1. \dots, X_r, Y_{r + 1}, \dots, Y_k, Z_1, \dots, Z_\ell)$ and reordering so that $Y_i$'s appear before the $X_i$'s and the $Z_i$'s) with the ``new" $X_i$'s consisting of $Y_{r + 1}, \dots, Y_k$ and the ``new" $Z_j$'s consisting of $X_1, \dots, X_r, Z_1, \dots, Z_\ell$. Additionally, since we can accomplish this lemma in one single application of \cite[Theorem 8]{Len2}, it follows that we do not need to consider $h$-independent phases in the hypothesis for our argument of the main theorem.
\end{remark}
\begin{proof}
Let $\omega$ denote the asymmetric bilinear form associated to $G/\Gamma$ representing the commutator bracket. Let $d = k + \ell + 1$ be the dimension of $G$. Let $G_{(4)} = \{x \in G^{(4)}: \text{the } Z_i \text{ components of } x \text{ are equal}\}$. The commutator bracket on $G^{(4)}$ is of the form $\omega_1 + \omega_2 - \omega_3 - \omega_4$ where $\omega_i$ represents the commutator bracket on the $i$th coordinate of $G^4$. We let $\tilde{\omega}$ be the restriction of that commutator bracket to $G_{(4)}$. Then by \thref{twostepcor}, and Pigeonholing in $(h_1, h_2, h_3, h_4)$, one can find a $\delta^{-O(d)^{O(1)}}$-rational subspace $V$ with $\tilde{\omega}(V, V) = 0$ such that for a set $R$ of at least $\delta^{O(d)^{O(1)}} N^3$ many additive quadruples $(h_1, h_2, h_3, h_4)$, (up to an integer shift possibly depending on $h_1, h_2, h_3, h_4$) $\psi_{horiz}(g_{h_1}, g_{h_2}, g_{h_3}, g_{h_4}, g_*)$ lies in $V$. Let $V_{123} = \pi_{123}(V)$ be the projection into the first three coordinates and the $h$-independent coordinate and $V_{124} = \pi_{124}(V)$ be the projection of $V$ into the first, second, and fourth coordinates and the $h$-independent coordinate. We write
$$V_{123} = \{(g_1, g_2, g_3, g_*): \eta_{h_1}(g_1) + \eta_{h_2}(g_{2}) + \eta_{h_3}(g_3) + \eta_*(g_*) = 0 \forall \eta = (\eta_{h_1}, \eta_{h_2}, \eta_{h_3}, \eta_*) \in S\}$$
where $S$ consists of linearly independent integer vectors. Similarly, we write
$$V_{124} = \{(g_1, g_2, g_4, g_*): \eta_{h_1}'(g_1) + \eta_{h_2}'(g_{2}) + \eta_{h_4}'(g_4) + \eta_*'(g_*)= 0 \forall \eta' = (\eta_{h_1}', \eta_{h_2}', \eta_{h_4}', \eta_*') \in S'\}.$$
Since $V$ is $\delta^{-O(d)^{O(1)}}$-rational, it follows that $V_{123}$ is also $\delta^{-O(d)^{O(1)}}$-rational, so by Cramer's rule (e.g., \cite[Lemma A.8]{Len2}), it follows that we may take all $\eta_{h_i}$, $\eta_{h_i}'$, $\eta_*$, $\eta_*'$ to be of size $\delta^{-O(d)^{O(1)}}$. Let
$$W_1 = \{v \in \text{span}(X_1, \dots, X_k): (v, 0, 0, 0) \in V_{123}\}$$
$$W_2 = \{v \in \text{span}(X_1, \dots, X_k): (v, 0, 0, 0) \in V_{124}\}.$$
Note that $v$ lies in $W_1$ if and only if $\eta_{h_1}(v) = 0$ for all $\eta \in S$ and similarly $w \in W_2$ if and only if $\eta_{h_1}'(w) = 0$ for all $\eta' \in S'$. Furthermore, since $\tilde{\omega}(V, V) = 0$, and $v \in W_1$ and $w \in W_2$, it follows that $(v, 0, 0, 0) \in V_{123}$ lifts to $(v, 0, 0, z, 0)$ and $(w, 0, 0, 0) \in V_{124}$ lifts to $(w, 0, z', 0, 0)$ for elements $z, z' \in \text{span}(Z_1, \dots, Z_\ell)$. Thus, denoting $\tilde{v}$ and $\tilde{w}$ these lifts, we have $0 = \tilde{\omega}(\tilde{v}, \tilde{w}) = \omega(v, w)$. Hence $\omega(W_1, W_2) = 0$. Let $W = W_1 \cap W_2$. Note that an element $w \in W$ if and only if for each $\eta \in S$ and $\eta' \in S'$ that $\eta_{h_1}(w) = 0$ and $\eta_{h_1}'(w) = 0$. \\\\
By the Pigeonhole principle, there are thus $\delta^{O(d)^{O(1)}}N$ many $h_1$ which is part of at least $\delta^{O(d)^{O(1)}} N^2$ additive quadruples $(h_1, h_2, h_3, h_4)$ in $R$. Thus, there are at least $\delta^{O(d)^{O(1)}}N^5$ many elements $(h_1, h_2, h_3, h_2', h_4')$ such that $(h_1, h_2, h_3, h_1 + h_2 - h_3) \in R$ and $(h_1, h_2', h_1 + h_2' - h_4', h_4') \in R$. By the Pigeonhole principle, there exists choices $(h_2, h_3)$ and $(h_2', h_4')$ such that there are at least $\delta^{O(d)^{O(1)}}N$ many $h_1$ with $(h_1, h_2, h_3, h_1 + h_2 - h_3) \in R$ and $(h_1, h_2', h_1+ h_2' - h_4', h_4') \in R$. Let $H'$ denote this set of $h_1$'s. \\\\
Let the projection to $h_1$ component of elements in $S \cup S'$ be $(\tilde{\eta}^1, \dots, \tilde{\eta}^r)$ and suppose without a loss of generality that $\tilde{\eta}^i$ are linearly independent. Let us assume without a loss of generality that $\tilde{\eta}^1, \dots, \tilde{\eta}^r$ are Gaussian eliminated and (after possibly reordering the variables) are in row-reduced echelon form. This can be done while keeping similar bounds for these elements. Noticing that $h_2, h_3, h_2', h_4'$ are fixed, there exists fixed $h$-independent phases $\alpha_1, \dots, \alpha_r$ such that for each $h \in H'$, we have
$$\tilde{\eta}^i(g_h) + \alpha_i \equiv 0 \pmod{1}.$$
By scaling $g_h(\cdot)$ by $q$, we may assume that the first nonzero component of $\tilde{\eta}^i$ is an integer. Thus (since $\tilde{\eta}^1, \dots, \tilde{\eta}^r$ are in row-reduced echelon form), we may write
$$\psi_{horiz}(g_h(q \cdot)) \equiv \alpha_1 X_1 +  \cdots + \alpha_r X_r + \xi_h^{r + 1}Y_{r + 1} + \cdots + \xi_h^{k}Y_{k} + \xi_*^1Z_1 + \cdots + \xi_*^\ell Z_\ell \pmod{\Gamma}$$
as prescribed in the lemma. We claim that $Y_j \in W$ for each $j$. This is because $\tilde{\eta}^1, \dots, \tilde{\eta}^r$ are in row-reduced echelon form so if
$$g = \sum_j g_j X_j$$
and $\tilde{\eta}^i(g) = 0$, then we may write
$$\tilde{\eta}^i(qg) = \sum_j g_j Y_j.$$
Since $Y_j$'s are linearly independent, the set of all such $g$'s have the same span as the set of the $Y_j$'s and we note that $W = \bigcap_{i = 1}^r \text{ker}(\tilde{\eta}^i)$. Hence the result.
\end{proof}
Thus, we have:
\begin{proposition}[Sunflower decomposition]
There exists a set $H$ of size at most 
$$\delta^{\log(1/\delta)^{O(1)}} N$$ 
and associated Fourier expanded nilcharacters $\chi_h(n)$ associated to $h$ on a ($h$-independent) nilmanifold $G/\Gamma$ such that
$$|\mathbb{E}_{n \in \mathbb{Z}_N} \Delta_h \tilde{f}(n) \chi_h(n)| \ge \delta^{\log(1/\delta)^{O(1)}}$$
for some function $\tilde{f} = f(y \cdot )$ where $y$ is an integer relatively prime to $N$, and
the frequency set of $\chi_h$ decomposes as $\Xi_h$, an $h$-dependent set, and $\Xi_*$, an $h$-independent set both of size at most $\log(1/\delta)^{O(1)}$ such that $\omega(\Xi_h^G, \Xi_h^G) = 0$.
\end{proposition}

\subsection{Linearization}
We begin with the following lemma:
\begin{lemma}\thlabel{bracketlinear}
Suppose $H \subseteq \mathbb{Z}_N$ and $f: H \to (\widehat{\mathbb{Z}_N})^d$ has $\delta N^3$ additive quadruples it preserves. Then on a subset of $H$ of size at least $\exp(-O(\log^4(1/\delta)))N$, $f$ can be written as $f(n) = (f_1(n), \dots, f_d(n))$
$$f_j(n) = \sum_{i = 1}^\ell a_{ij} \{\alpha_i (n - k)\}$$
where $a_{ij} \in \mathbb{R}$, $\alpha_i \in \widehat{\mathbb{Z}_N}$, $k$ an integer, and $\ell \le O(\log^{O(1)}(1/\delta))$.
\end{lemma}
\begin{proof}[Proof of \thref{bracketlinear}]
Balog-Szemer\'edi-Gowers \cite[Proposition 7.3]{Gow01} (or rather as it is stated \cite[Theorem 5.2]{GT08b}) shows that we may pass to a subset $H'$ of size $\delta^{O(1)}N$ where denoting $\Gamma$ to be the graph of $f$ restricted to $H'$, we have $|\Gamma + \Gamma| \le \delta^{-O(1)}|\Gamma|$. To continue the proof of \thref{bracketlinear}, we need the following lemma: 
\begin{lemma}\thlabel{expandingfreiman}
There exists a subset $H'' \subseteq H'$ of cardinality at least $\delta^{O(1)}N$ such that $f$ is a $8$-Freiman homomorphism.
\end{lemma}
\begin{proof}[Proof of \thref{expandingfreiman}]
We follow the proof of \cite[Lemma 9.2]{GT08b}. Let $\Gamma$ be the graph of $f$. Let $A \subseteq (\widehat{\mathbb{Z}_N})^d$ be the set of all $\xi$ such that $(0, \xi) \in 8\Gamma - 8\Gamma$. Since $\Gamma$ is a graph, it follows that $|\Gamma + A| = |\Gamma||A|$. On the other hand, $|\Gamma + A| \le |9\Gamma - 8\Gamma| \le \delta^{-O(d)}N$. This tells us that $|A| \le \delta^{-O(1)}$. By \cite[Lemma 8.3]{GT08b}, we may find a set $T \subseteq \mathbb{Z}_N^d$ with $|T| \le O(\log_2(1/\delta))$ such that $A \cap B(T, 1/4) = \{0\}$ where $B(T, 1/4) = \{\xi \in \widehat{\mathbb{Z}_N}^d: \|t(\xi)\|_{\mathbb{R}/\mathbb{Z} \forall t \in T} < 1/4\}$. \\\\
Let $\Psi: \widehat{\mathbb{Z}_N}^d \to (\mathbb{R}/\mathbb{Z})^T$ be the homomorphism $\Psi(\xi) = (t(\xi))_{t \in T}$. We may cover $(\mathbb{R}/\mathbb{Z})^T$ by $2^{6|T|} \le 2^6 \delta^{-O(1)}$ many cubes of side length $\frac{1}{64}$. Since $|\Gamma| \ge \delta N$, it follows that there exists some cubes $Q$ for which
$$\Gamma' := \{(h, f(h)) \in \Gamma: \Psi(f(h)) \in Q\}$$
has cardinality $\delta^{O(1)}N$. By linearity of $t$, it follows that if $(0, \xi) \in 8\Gamma' - 8\Gamma'$ then $\|t(\xi)\|_{\mathbb{R}/\mathbb{Z}} \le \frac{16}{64}$, so $\xi \in B(T, 1/4)$. On the other hand, $\xi$ also lies in $A$, so $\xi = 0$ by construction of $T$. This allows us to conclude that $4\Gamma' - 4\Gamma'$ is a graph, or in other words, $f$ is an $8$-Freiman homomorphism on $H''$, the projection of $\Gamma'$ to the first coordinate.
\end{proof}
Since $f$ is an $8$-Freiman hoomorphism on $H''$, $f$ is a Freiman homomorphism on $2H'' - 2H''$, which by \cite{San12b} and \cite[Proposition 27]{Mil21} contains a Bohr set $B = B(S, \rho)$ where $|S| \le \log(1/\delta)^{O(1)}$ and $\rho \ge \delta^C$ for some constant $C > 0$. Thus, $f$ restricts to a Freiman homomorphism on a Bohr set $f:B \to (\widehat{\mathbb{Z}/N\mathbb{Z}})^d$. We now require the following lemma which is essentially \cite[Proposition 10.8]{GT08b} to finish the proof of \thref{bracketlinear}:
\begin{lemma}\thlabel{bohrbracket}
Suppose $f: B \to \widehat{\mathbb{Z}/N\mathbb{Z}}$ is a Freiman homomorphism. Then letting $B_\epsilon = B(S, \rho\epsilon)$, it follows that on $B_{|S|^{-O(|S|)}}$, one can write
$$f(n) = \sum_i a_i \{\alpha_i n\} + \gamma$$
where each $\alpha_i \in S$.
\end{lemma}
\begin{proof}[Proof of \thref{bohrbracket}]
By \cite[Proposition 10.5]{GT08b}, $B$ contains a symmetric proper generalized arithmetic progression $P = \{\ell_1 n_1 + \cdots + \ell_{d'} n_{d'}: n_i \in [-N_i, N_i]\}$ with $1 \le d' \le |S|$ such that $P$ contains $B_{|S|^{-O(|S|)}}$ and $(\{\alpha \ell_i\})_{\alpha \in \mathbb{R}^S}$ are linearly independent. It follows that we may write
$$f(\ell_1 n_1 + \cdots  + \ell_d n_d) = n_1f(\ell_1) + \cdots + n_df(\ell_d) + f(0) = n_1k_1 + \cdots + n_dk_d + \gamma.$$
Let $\Phi: B \to \mathbb{R}^S$ be $\phi(x) = (\{\alpha x\})_{\alpha \in S}.$  We see that
$$\Phi(\ell_1n_1 + \cdots + \ell_{d'} n_{d'}) = n_1\Phi(\ell_1) + \cdots + n_d\Phi(\ell_{d'}).$$
Since $\Phi(\ell_i)$ are linearly independent, it follows that (by scaling up a vector in the orthogonal complement of $\Phi(\ell_j)$ for $j \neq i$ but not in the orthogonal complement of $\Phi(\ell_i)$ for all $i$) there exists a vector $u_i$ such that $u_i \cdot \Phi(\ell_1n_1 + \cdots + \ell_{d'} n_{d'}) = \ell_i$. Thus, writing $u_i = (u_{i\alpha})_{\alpha \in S}$, we see that
$$\ell_i = \sum_{\alpha \in S} u_{i\alpha} \{\alpha x\}.$$
The lemma follows from this.
\end{proof}
Thus, each component of $f$ agrees with a bracket linear function with at most $d$ many phases on $B_{d^{-O(d)}}$. By the Pigeonhole principle, there exists some $x_0$ such that $|(x_0 + B_{1/2}) \cap H| \ge \delta^{O(1)}N$ where $B_{1/2}$ is the $\frac{1}{2}$-dilation of $B$ (say regularized as well). Let $A = (x_0 +B) \cap H$. Then for $h, h' \in A$, we have $(h - h', f(h) - f(h'))$ lies inside $\Gamma' - \Gamma'$, so $f(x_0 + h) - f(x_0 + h') = f(h - h')$, so there exists $\xi_0$ such that $f(x_0 + h) = \xi_0 + f(h)$. Thus, there exists some set of size at least $\delta^{O_\epsilon(\log^{3 + \epsilon}(1/\delta))}N$ such that $f$ is a bracket polynomial with $\log^{3 + \epsilon}(1/\delta)$ many bracket phases.
\end{proof}
We will also need the following lemma:
\begin{lemma}[Cauchy-Schwarz inequality for energy]\thlabel{CSenergy}
We have the following:
$$E(A_1, A_2, A_3, A_4) \le E(A_1)^{1/4}E(A_2)^{1/4}E(A_3)^{1/4}E(A_4)^{1/4}.$$
\end{lemma}
\begin{proof}
We see that
$$E(A_1, A_2, A_3, A_4) = \frac{1}{N^3}\sum_{z \in (A - C) \cap (B - D)} |A \cap (C + z)||B \cap (D + z)|.$$
By Cauchy-Schwarz, we may bound this by
$$\sqrt{\left(\frac{1}{N^3}\sum_{z \in A - C} |A \cap (C + z)|^2\right)\left(\frac{1}{N^3}\sum_{z \in B - D} |B \cap (D + z)|^2\right)}.$$
We see by \cite[Lemma 2.9]{TV10} that
$$\frac{1}{N^3}\sum_{z \in A - C} |A \cap (C + z)|^2 = E(A, C) \le E(A)^{1/2}E(C)^{1/2}$$
and similarly
$$\frac{1}{N^3}\sum_{z \in B - D} |B \cap (D + z)|^2 = E(B, D) \le E(B)^{1/2}E(D)^{1/2}.$$
\end{proof}
We now prove the linearization step:
\begin{lemma}\thlabel{linearizationstep}
Suppose
$$|\mathbb{E}_{n \in [N]} \chi_{h_1}(n)\chi_{h_2}(n + h_1 - h_4)\overline{\chi_{h_3}(n)\chi_{h_4}(n + h_1 - h_4)}| \ge \delta.$$
and that $\omega(\Xi_h^G, \Xi_h^G) = 0$. Then there exists a set $S \subseteq \widehat{\mathbb{Z}_N}$ such that $|S| \le \log(1/\delta)^{O(1)}$ and a subset $H' \subseteq H$ with $|H'| \ge \log(1/\delta)^{O(1)}$ such that for each $h \in H'$ there exists some $k \in \mathbb{Z}_N$ such that denoting $\xi_h^{struc}$ as a generic element of the form
$$\xi_h^{struc} = \sum_{\alpha \in S} a_\alpha \{\alpha(h - k)\}$$
where $a_\alpha \in \mathbb{R}$ so that each $h \mapsto \{\alpha(h - k)\}$ is a Freiman homomorphism on $H'$, we may write for some appropriate integer scaler $q$ with $|q| \le \delta^{-O(d)^{O(1)}}$ and some $r > 0$, such that if
$$\psi_{horiz}(g(n)) = \sum_i \xi_h^i n X_i + \sum_j \xi_*^j n Z_j,$$
we have (for some ordering of the basis $(X_1, \dots, X_k, Z_1, \dots, Z_\ell)$),
$$\psi_{horiz}(g_h(qn)) \equiv \sum_{j \le r} \xi_h^{struc, j} nX_j  + \sum_{j > r} \xi_h^j n Y_j + \sum_{j \ge r_1} \xi_*^j n \tilde{Z}_j \pmod{\Gamma}$$
where $\tilde{Z}_j$ is a $\delta^{-d^{O(1)}\log(1/\delta)^{O(1)}}$-integer span of $Z_j$ and $Y_j$ is the $\delta^{-d^{O(1)}\log(1/\delta)^{O(1)}}$-integer span of the $X_i$'s and the $X_i, Y_i, \tilde{Z}_i$ are linearly independent and the restriction of $\omega(Y_i, \tilde{Z}_j) = 0$ for each $i$ and $j$.
\end{lemma}
\begin{proof}
By using various bracket polynomial identities, we may rewrite 
$$\chi_{h_1}(n)\chi_{h_2}(n)\overline{\chi_{h_3}(n)\chi_{h_4}(n)}$$
as a free nilcharacter of the form $\tilde{\chi}_{h_1, h_2, h_3, h_4}(n)$ which has the same free nilcharacter of $\chi_h$ except that at each phase $\xi_h$, we replace it with $\xi_{h_1} + \xi_{h_2} - \xi_{h_3} - \xi_{h_4}$ and there are no phases of the form $\xi_*^i n[\xi_*^j n]$. Thus, there exists some $\alpha_{h_1, h_2, h_3, h_4}'$ and $\beta_{h_1, h_2, h_3, h_4}'$ such that
$$|\mathbb{E}_{n \in [N]} \tilde{\chi}_{h_1, h_2, h_3, h_4}(n)e(\alpha'_{h_1, h_2, h_3, h_4}n^2 + \beta'_{h_1, h_2, h_3, h_4}n)| \ge \delta^{O(d)^{O(1)}}.$$
Let $g_{h_1, h_2, h_3, h_4}$ denote the nilsequence corresponding to $\chi_{h_1, h_2, h_3, h_4}$ and \cite[Theorem 8]{Len2} shows that (up to an integer shift), $g_{h_1, h_2, h_3, h_4}$ lies in a subspace $V$ such that $\omega'(V, V) = 0$ where $\omega$ is the asymmetric bilinear form corresponding to $G/\Gamma$ and $\omega'(X_i, X_j) = \omega(X_i, X_j)$ and $\omega'(X_i, Z_j) = \omega(X_i, Z_j)$, and $\omega'(Z_i, Z_j) = 0$ (this last fact reflects there being no terms of the form $\xi_*^i [\xi_*^j]$). Letting $\eta^i = (\eta^i_h, \eta^i_*)$ be linearly independent and span the annihilators of $V$. We have
$$\eta^i(g_{h_1, h_2, h_3, h_4}) = \eta^i_h(\xi_{h_1} + \xi_{h_2} - \xi_{h_3} - \xi_{h_4}) + \eta^i_*(\xi_*) \equiv 0 \pmod{1}.$$
Thus, by a combination of \thref{CSenergy} and \thref{bracketlinear}, there exists a subset $H'$ of $H$ of size at least $\delta^{(d\log(1/\delta))^{O(1)}}|H|$ such that for each $h \in H'$, $\eta^i_h(\xi_h)$ is of the prescribed bracket polynomial form (and is a Freiman homomorphism) and such that for many additive quadruples, we have
$$\eta^i_h(\xi_{h_1} + \xi_{h_2} - \xi_{h_3} - \xi_{h_4}) + \eta^i_*(\xi_*) \equiv 0 \pmod{1}.$$
Let $w \in \text{span}_{\mathbb{Z}}(X_1, \dots, X_k)$ be orthogonal to all of the $\eta^i_h$. It follows that $(w, 0)$ is orthogonal to $\eta^i$ and thus $w \in V$. Hence
$$\omega'(w, g_{h_1, h_2, h_3, h_4}) = \omega(w, g_*) \equiv 0 \pmod{1}$$
where $g_*$ is the $h$-independent part of $g_{h_1, h_2, h_3, h_4}$. Choosing such $w_1, \dots, w_s$ of size at most $\delta^{-d^{O(1)}\log(1/\delta)^{O(1)}}$ to be in the span of the orthogonal subspace of the $\eta_i^h$'s, and reducing $r$ when necessary, we may assume that $\omega(w_i, \cdot)$ are linearly independent relations and that $\omega(w_i, g_*) \equiv 0 \pmod{1}$. Let $W$ be the subspace of $\psi_{horiz}(G)$ defined as $\{x: \eta^i_h(x) = 0, \omega(w_j, x) = 0 \forall i, j\}$. It follows that if $x, y \in W$ and $x$ lies in the span of $X_i$'s then $\omega(x, y) = 0$. This is because $x \in W$ implies that $\eta^i_h(x) = 0$, so $x$ can be generated by a combination of the $w_j$'s. However, $\omega(w_j, y) = 0$. Hence, $\omega(x, y) = 0$. Writing $\nu_1, \dots, \nu_s$ to be the relations defined by $\omega(w_j, \cdot)$, we have that
$$\eta_i^h(g_h) + \xi_h^{struc, i} \equiv 0 \pmod{1}$$
$$\nu_j(g_*) \equiv 0\pmod{1}.$$
We assume that $\eta_i$'s and $\nu_j$'s are written (after possibly reordering variables) in row-reduced echelon form. Thus, writing
$$\psi_{horiz}(g_h) = \sum_i \xi_h^i X_i + \sum_j \xi_*^j Z_j$$
and scaling $g_h$ by an appropriate integer, it follows that there exists some integer $q \le \delta^{-(\log(1/\delta)d)^{O(1)}}$ such that taking the linear relations above into account, we may write
$$\psi_{horiz}(g_hq) \equiv \sum_{i \le r} \xi_{h}^{struc, i} X_i + \sum_{i > r} \xi_h^i Y_i + \sum_j \tilde{\xi}_*^j \tilde{Z}_j \pmod{\Gamma}$$
where $\tilde{Z}_j$ is a $\delta^{-(d\log(1/\delta))^{O(1)}}$-integer combination of $Z_j$'s and $Y_i$ is a $\delta^{-(d\log(1/\delta))^{O(1)}}$-integer combination of the $X_i$'s and that the combination of $X_i, Y_i, Z_i$ are linearly independent and $\xi^{struc, i}_h$ are of the bracket polynomial form as specified above, and that $Y_i$ and $\tilde{Z}_j$ lie inside $W$. The last property we can guarantee because the subspace $\{x \in \text{span}(X_1, \dots, X_k): \eta_h^i(x) = 0\}$ is spanned by $Y_i$'s and the subspace $\{x \in \text{span}(Z_1, \dots, Z_\ell): \nu_j(x) = 0\}$ is spanned by $\tilde{Z}_j$'s.
The result follows.
\end{proof}
\begin{remark}
This proof can be made entirely free of use of ``bracket polynomial identities" and ``Fourier complexity lemmas." In particular, the step where we take our modified bracket polynomial $g_{h_1, h_2, h_3, h_4}$ and replace it with $g_h$ but with $\xi_h$ replaced with $\xi_{h_1} + \xi_{h_2} - \xi_{h_3} - \xi_{h_4}$ is unnecessary. One can quotient by the subspace of $\{\xi_{h_1} + \xi_{h_2} - \xi_{h_3} - \xi_{h_4} = 0\}$ and argue that any horizontal character on $G_{(4)}$ that annihilates that subspace must be of the form 
$$\eta^i(\xi_{h_1}, \xi_{h_2}, \xi_{h_3}, \xi_{h_4}, \xi_*) = \eta^i_h(\xi_{h_1} + \xi_{h_2} - \xi_{h_3} - \xi_{h_4}) + \eta^i_*(\xi_*).$$
\end{remark}
By another application of the procedure in Section 8.2, we may thus write
$$\psi(g_h(n)) = (\xi_h^1 n, \dots, \xi_h^\ell n, \xi_*^1 n, \dots, \xi_*^k n, 0, 0, \dots, 0, P_h(n))$$
where $\xi_h^i$ is of the form specified by the Lemma above and $\xi_*^j$ is $h$-independent and the $0$'s represent the coordinates $Z_1, \dots, Z_{r_1 - 1}$ which we eliminated in the above lemma. \\\\
This linearizes the entire bracket polynomial except for the vertical phase $P_h(n)$. We write $P_h(n) = \alpha_h n^2 + \beta_h n - P_h(1)n - {n \choose 2}\sum_{i < j} C_{[i, j]} g^i g^j $
where $(g^1, \dots, g^{k + \ell})$ is the ordering 
$$(\xi_h^1, \dots, \xi_h^\ell, \xi_*^1, \dots, \xi_*^k).$$
We note that the Mal'cev coordinates of $g_h(1)^n$ is 
$$(\xi_h^1 n, \dots, \xi_h^\ell n, \xi_*^1 n, \dots, \xi_*^k n, P_h(1)n + \frac{1}{2}{n \choose 2}\sum_{i < j} C_{[i, j]} g^i g^j).$$
To linearize the quadratic term in $P_h(n)$, it thus suffices to linearize $\alpha_h$ since if $i$ and $j$ correspond to some bases $X_{i'}$ and $X_{j'}$, then $[X_{i'}, X_{j'}] = C_{[i, j]} = 0$. We aim to linearize $\alpha_h$. It turns out that we do not need to linearize the linear term in $P_h(n)$ since we will be able to eliminate it via the symmetry argument. First we show that $\alpha_h$ is rational with denominator $2N$. To show this, we essentially follow the procedure in \thref{periodic}. Consider the sequence
$$\tau_k(g_h(n))\Gamma \times \Gamma  = (g_h(n + k), g_h(n)) \Gamma \times \Gamma = (g_h(1)^n g_h^{nlin}(n + k)[\{g_h(1)^k\}^{-1}, g(1)^{-n}], g(1)^n) \Gamma \times \Gamma$$
where $g_h^{nlin}(n) = g_h(n)g_h(1)^{-n}$. This restricts to a sequence in $G \times_{G_2} G/(\Gamma \times_{G_2 \cap \Gamma} \Gamma)$. Since $ (g_h(n + k), g_h(n)) \Gamma \times \Gamma$ is periodic, so is $(g_h(1)^n g_h^{nlin}(n + k)[\{g_h(1)^k\}^{-1}, g(1)^{-n}], g_h(1)^n) \Gamma \times \Gamma$ and its restriction. By \cite[Lemma A.3]{Len2}, each horizontal character $\eta$ on $G \times_{G_2} G$ decomposes as $\eta(g, g') = \eta_1(g) + \eta_1(g'g^{-1})$ where $\eta_1$ is a horizontal character on $G$ and $\eta_2$ is a character on $G_2/[G_2, G] = G_2$. Since $\tau_k$ is periodic modulo $N$, it follows that for any $\eta_2$, that
$$\eta_2(g_h^{nlin}(n + k)) - \eta_2(g_h^{nlin}(n)) + \eta_2([\{g_h(1)^k\}, g_h(1)^{n}]) \equiv 0 \pmod{1/N}.$$
Since $\{g_h(1)^k\}[g_h(1)^k] = g_h(1)^k$ and $[g_h(1)^k, g_h(1)^n] = 0$, it follows that
$$\eta_2([\{g_h(1)^k\}, g_h(1)^{n}]) = \eta_2([[g_h(1)^k]^{-1}, g_h(1)^{n}]).$$
Since the image of $g_h(1)^n$ under the quotient map $G \mapsto G/[G, G]$ lies inside $\Gamma /[G, G]$, it follows that 
$$\eta_2([[g_h(1)^k]^{-1}, g_h(1)^{n}]) \equiv 0 \pmod{1/N}.$$
Thus,
$$\eta_2(g_h^{nlin}(n + k)) - \eta_2(g_h^{nlin}(n)) \equiv \eta_2(2\alpha_hnk) + f(k) \equiv 0 \pmod{1/N}$$
for any $\eta_2$ and some function $f$. Hence, $\alpha_h$ is rational with denominator $2N$. Thus, using periodicity to scale $n$ by $2n$, we may work with $g_h(2n)$ instead of $g_h(n)$ so we may assume that $\alpha_h$ is rational with denominator $N$. \\\\
To linearize $\alpha_h$, we apply the linearization argument one more time, obtaining that there exists $\delta^{\log(1/\delta)^{O(\log\log(1/\delta))}} N^3$ many additive quadruples $(h_1, h_2, h_3, h_4)$ such that
$$\mathbb{E}_{n \in \mathbb{Z}/N\mathbb{Z}} e(\alpha_{h_1} n^2 + \alpha_{h_2}(n + h_1 - h_4)^2 + \beta_{h_1, h_2, h_3, h_4} n - \alpha_{h_3}n^2 - \alpha_{h_4}(n + h_1 - h_4)^2) \ge \delta^{\log(1/\delta)^{O(1)}}$$
for some phase $\beta_{h_1, h_2, h_3, h_4} \in \widehat{\mathbb{Z}/N\mathbb{Z}}$. This implies that
$$\alpha_{h_1} + \alpha_{h_2} \equiv \alpha_{h_3} + \alpha_{h_4} \pmod{1}.$$
Arguing as above, this implies that there exists some subset $H' \subseteq H$ and $k \in \mathbb{Z}/N\mathbb{Z}$ such that for each $h \in H'$,
$$\alpha_{h} = \sum_{i = 1}^{d'} a_i\{\alpha_i (h - k)\}$$
with $(\alpha_i)_{i = 1}^{d'} \in \widehat{\mathbb{Z}/N\mathbb{Z}}$ and $d' \le \log(1/\delta)^{O(1)}$. This linearizes $\alpha_h$. Thus, we have the following.
\begin{lemma}[Linearization]\thlabel{linearizationargument}
There exists a subset $H' \subseteq \mathbb{Z}_N$ of size at least $\delta^{O(\log(1/\delta))^{O(1)}}N$ such that the following holds:
\begin{itemize}
    \item For each $h \in H'$, there exists an associated periodic Fourier expanded nilcharacter $\chi_h$ for each $h \in H'$, $\chi_h$ is the exponential of a bracket polynomial with at most $\log(1/\delta)^{O(1)}$ terms which include
$$\xi_*^i n[ \xi_h^j n], \xi_h^j n[\xi_*^i n]$$
and a single term of the form
$$\xi_{h}^{struc} n^2$$
where $\xi_h^j$, and $\xi_h^{struc}$ are of the form
$$\sum_{\alpha \in S} a_\alpha \{\alpha(h - k)\}$$
where $a_\alpha \in \mathbb{R}$, $S \subseteq \widehat{\mathbb{Z}_N}$ is a subset of size at most $\log(1/\delta)^{O(1)}$, and $\xi_h^j, \xi_h^{struc}, \xi_*$ are all rational with denominator $N$ and $H' \subseteq k + B(S, \delta^{O(1)})$.
    \item There exists some nonzero $y \in \mathbb{Z}_N$ such that
    $$|\mathbb{E}_{n \in [N]} \Delta_h f(y n) \chi_h(n)| \ge \delta^{O(\log(1/\delta))^{O(1)}}.$$
\end{itemize}
\end{lemma}
Note that we may unwind the scaling of $f$ by making a change of variables $n \mapsto pm$ where $p$ is the modular inverse of $y$. We can then absorb the factor of $p$ in the frequencies of $\chi_h$.
\subsection{Symmetry and integration argument}
Thus, we have obtained that (essentially) $\Delta_h f(x)$\footnote{We will assume here that we have unwinded the scaling of $\tilde{f}$ so we work with $f$.} correlates with sums of terms of the form
$$\left(\sum_{i = 1}^\ell a_i \{\alpha_i(h - k)\} n\right) [\beta n]$$
and
$$\beta n\left[\sum_{i = 1}^\ell a_i \{\alpha_i(h - k)\}n\right].$$
Writing $[\beta n] = \beta n -\{\beta n\}$, it follows up to lower order terms modulo $1$, we may write the first one as
$$\sum_{i = 1}^\ell a_i \beta n^2\{\alpha_i (h - k)\} - \{\beta n\} \sum_{i = 1}^\ell a_i n\{\alpha_i (h - k)\}$$
and the second one as (up to lower order terms modulo $1$)
$$\{\beta n\}\sum_{i = 1}^\ell a_i n \{\alpha_i (h - k)\}.$$
The point of these computations is to illustrate that for our situation, each of the $\alpha_i$'s and the $\beta$ will lie in $\widehat{\mathbb{Z}_N}$, so the phases in brackets will be periodic modulo $N$. We may also translate our set $H$ by $k$ to eliminate the $k$ term. \\\\
It follows that $\Delta_{h + k} f$ correlations with a bracket polynomial of the form $e(3T(h, n, n) + B(n) + \theta_h n)$ where
$$T(x, y, z) := \sum_{j = 1}^d \{\alpha_j x\}\frac{\beta_j}{6} y \{\gamma_j z\} +\{\alpha_j x\}\frac{\beta_j}{6} z \{\gamma_j y\} + \sum_{j = 1}^d \frac{\alpha_j'}{3} \{\beta_j' x\}yz$$
with $\alpha_j, \gamma_j, \beta_j' \in \widehat{\mathbb{Z}_N}$ and $B(n)$ is a degree two bracket polynomial of at most $2d$ phases and is of the form $B(n) = \sum_{j = 1}^d \lambda_j n[\mu_j n]$ (and $\theta_h n$ should be thought of as a ``lower order term"). We wish to show that $f$ correlates with a degree three bracket polynomial. One important observation is that this form is trilinear on the Bohr set with frequencies $S := \{1/N, \alpha_1, \dots, \alpha_d, \gamma_1, \dots, \gamma_d, \beta_1', \dots, \beta_d'\}$ and radius, say $1/10$ ($1/N$ is added to $S$ to ensure that $T$ is genuinely trilinear on the Bohr set with frequencies of $S$ and radius $1/10$). \\\\
We will now give a sketch of this step. By hypothesis, we only have local linearity in the first variable. Let us assume for simplicity that $T$ is trilinear.  We wish to write $3T(h, n, n)$ as a combination of $T(n + h, n + h, n + h) - T(n, n, n)$ plus lower order terms in $n$. If this can be done, then we may absorb the $T(n, n, n)$ and the $T(n + h, n + h, n + h)$ into $f(n)$ and $f(n + h)$, respectively, obtaining that
$$\mathbb{E}_h \|f_1(n)f_2(n + h)\|_{U^2} \gg_\epsilon 1$$
for $f_1$ a product of $f$ and bracket cubics. We can deduce that $f_1$ has large $U^3$ norm, so applying the $U^3$ inverse theorem, we can conclude the $U^4$ inverse theorem. \\\\
The only thing preventing us from showing this is that we have three terms $T(h, n, n) + T(n, h, n) + T(n, n, h)$, which may not necessarily be equal. Since we symmetrized the last two variables of $T$, we may assume that the latter two terms are equal. It is not actually possible to show that all three terms are equal, but rather, it is possible to show that the difference of two of them are of ``lower order," in the sense of being a two-step in $n$ and $h$ (or one step in $n$).\\\\
To show that $T$ is symmetric in its first two variables, we go back to the hypothesis
$$\mathbb{E}_h |\mathbb{E}_{n} \Delta_{h + k} f(n)\chi_h(n)|^2 \ge \delta^{O(\log(1/\delta))^{O(1)}}.$$
For simplicity, we will relabel the right hand side as simply $\delta$. By \cite[Proposition 6.1]{GTZ11}, it follows that
$$\mathbb{E}_{h_1 + h_2 = h_3 + h_4} |\mathbb{E}_{n} \chi_{h_1}(n) \chi_{h_2}(n + h_1 - h_4) \overline{\chi_{h_3}(n) \chi_{h_4}(n + h_1 - h_4)}|^2 \ge \delta^{O(1)}$$
or for some one-bounded function $b(h_1, h_2, h_3, h_4)$:
$$\mathbb{E}_{h_1 + h_2 = h_3 + h_4} \mathbb{E}_{n} \chi_{h_1}(n) \chi_{h_2}(n + h_1 - h_4) \overline{\chi_{h_3}(n) \chi_{h_4}(n + h_1 - h_4)}b(h_1, h_2, h_3, h_4) \ge \delta^{O(1)}$$
Writing $\chi_{h}(n) = e(3T(h, n, n) + \theta_h n)$ and using trilinearity and making a change of variables $(h_1, h_2, h_3, h_4) = (h, h + x + y, h + y, h + x)$, we obtain eventually (after Pigeonholing in $h$) that
$$|\mathbb{E}_{n, x, y} e(-6(T(x, y, n))b(n, x)b(n, y)b(n, x + y) b(x, y)| \gg \delta^{O(1)}$$
for one-bounded functions $b$ as in Gowers' notation \cite{Gow01} (where we identify $\mathbb{Z}/N\mathbb{Z}$ with the interval $[-(N - 1)/2, (N - 1)/2]$). Applying Cauchy-Schwarz in $x$ and $n$ gives
$$|\mathbb{E}_{n, x, y, y'} e(-6(T(x, y - y', n))b(n, y')b(n, y)b(n, x + y)b(n, x + y') b(x, y, y')| \gg \delta^{O(1)}.$$
Making a change of variables $z = x + y + y'$, we obtain
$$|\mathbb{E}_{n, z, y, y'} e(-6(T(z - y - y', y - y', n))b(n, y')b(n, y)b(n, z - y)b(n, z - y') b(z, y, y')| \gg \delta^{O(1)}.$$
Pigeonholing in $z$ and using trilinearity again, we obtain
$$|\mathbb{E}_{n, y, y'} e(-6(T(y, y', n) - T(y', y, n))) b(n, y)b(n, y')b(y, y')| \ge \delta^{O(1)}.$$
From here, we see that these one-bounded functions are of ``lower order" compared to the top term $e(T(y, y', n) - T(y', y, n))$; this allows us to conclude either with an application of the equidistribution theorem, or via some Bohr set argument similar to \cite{GTZ11} that $T(y, y', n) - T(y', y, n)$ is ``lower order" as well, implying that $T$ is symmetric in the first two variables up to lower order. This argument can either be made rigorous using a combination of bracket polynomial arithmetic argument and the equidistribution theorem found here, or by a Bohr set argument similar to the flavor of \cite{GTZ11}. We opt for the latter approach, since it is overall less annoying in the $U^4$ case though we imagine that the other approach generalizes better than the approach we opt for.
\begin{remark}
In \cite{GTZ11}, the authors work with the density assumption that there exists many additive quadruples $(h_1, h_2, h_3, h_4)$  such that
$$\mathbb{E}_n \chi_{h_1}(n) \chi_{h_2}(n + h_1 - h_4) \overline{\chi_{h_3}(n) \chi_{h_4}(n + h_1 - h_4)}$$
is large. Because of this, they need a further tool involving the bilinear Bogolyubov argument\footnote{The author was unable to understand their argument at the point where they use the bilinear Bogolyubov-type theorem; however, one can also follow their argument and use \cite{Mil21} to finish and obtain desirable quantitative bounds.}. We remove the use of bilinear Bogolyubov by taking advantage of additional averaging from the hypothesis
$$\mathbb{E}_{h_1 + h_2 = h_3 + h_4} |\mathbb{E}_{n} \chi_{h_1}(n)  \chi_{h_2}(n + h_1 - h_4) \overline{\chi_{h_3}(n)  \chi_{h_4}(n + h_1 - h_4)}|^2 \ge \delta^{O(1)}$$
rather than the density hypothesis of additive quadruples.
\end{remark}
We begin with the following lemma:


Thus, by applying \thref{FourierComplexity2} and \thref{onevarfouriercomplexity} (or rather the procedure of \thref{onevarfouriercomplexity}) and making the above manipulations, we may simplify
$$\mathbb{E}_{h_1 + h_2 = h_3 + h_4} |\mathbb{E}_{k, n} \chi_{h_1}(n) \chi_{h_2}(n + h_1 - h_4) \overline{\chi_{h_3}(n) \chi_{h_4}(n + h_1 - h_4)}|^2 \ge \delta^{O(1)}$$
to
$$|\mathbb{E}_{n, y, y'} e(-6(T(y, y', n) - T(y', y, n))) b(n, y)b(n, y')b(y, y')| \ge \delta^{O(d)^{O(1)}}.$$
Let us relabel $y'$ to $x$, so we instead end up with
$$|\mathbb{E}_{n, y, z} e(6(T(x, y, n) - T(y, x, n))) b(n, x)b(n, y)b(x, y)| \ge \delta^{O(d)^{O(1)}}.$$

\subsubsection{Bohr sets}
In the remainder of the proof, we will need to do various arithmetic operations over Bohr sets. In this section, we formulate notation we will use for Bohr sets. 
\begin{definition}
Given a subset $S \subseteq \widehat{\mathbb{Z}_N}$, and a real number $\rho \in (0, 1/2)$, we denote the \emph{Bohr set} with frequencies $S$ and radius $\rho$ to be $B(S, \rho) := \{x \in \mathbb{Z}_N: \|\alpha x \| < \rho \forall \alpha \in S\}$. 
\end{definition}
A Bohr set is \emph{regular} if whenever $\epsilon \le \frac{1}{100|S|}$, 
$$|B(S, \rho)| (1 - 100|S|\epsilon) \le |B(S, \rho (1 + \epsilon))| \le |B(S, \rho)|(1 + 100|S|\epsilon).$$
It was shown by Bourgain in \cite{Bou99} that regular Bohr sets are ubiquitous in the sense that there exists $\rho' \in [\rho/2, \rho]$ such that $B(S, \rho')$ is regular. If $S$ and $\rho$ are specified, we will shorten $B(S, \rho)$ to be $B$. Given a real number $\epsilon$ so that $\epsilon \rho \in (0, 1)$, we will denote by $B_\epsilon$ as $B(S, \epsilon \rho)$. We will also define the ``norms"
$$\|n\|_S = \sup_{\alpha \in S} \|\alpha n \|_{\mathbb{R}/\mathbb{Z}}.$$
\subsubsection{Finishing the symmetry argument}
Recall that we defined 
$$S = \{1/N, \alpha_1, \dots, \alpha_d, \gamma_1, \dots, \gamma_d, \beta_1', \dots, \beta_d'\}.$$ 
Let $B = B(S, \rho)$ be the Bohr set where all of the frequencies in $B$ are the bracket frequencies of $T$ where $\rho$ is selected in $[1/10, 1/5]$ such that $B(S, \rho)$ is regular. It follows that Pigeonholing in one of those translations, we have for some $n_0, x_0, y_0$ that
$$\sum_{n \in n_0 + B, x \in x_0 + B, y \in y_0 + B} e(6(T(x, y, n) - $$
$$T(y, x, n)))b(n, x)b(n, y)b(x, y)| \gg N^3 2^{-O(d)}\delta^{O(d)^{O(1)}}.$$
Making a change of variables, we have
$$\mathbb{E}_{n, x, y \in B} e(6(T(x - x^0, y - y^0, n - n_0)$$
$$- T(y - y^0, x - x^0, n - n_0)))b(n, x)b(n, y)b(x, y) \gg \delta^{O(d)^{O(1)}}.$$
Expanding out the bracket polynomials as $T(x - x^0, y - y^0, n - n_0) = T(x, y, n) + [\text{lower order terms}]$ (with the point of the lower order terms depending on fewer variables after applying a Fourier complexity lemma), and similarly for $T(y - y^0, x - x^0, n - n_0)$ and applying \thref{FourierComplexity2}, we therefore obtain
$$\mathbb{E}_{n, x, y \in B} e(6(T(x, y, n) - T(y, x, n)))b(n, x)b(n, y)b(x, y)| \gg \delta^{O(d)^{O(1)}}.$$
Fix elements $a, b, c \in B_{1/|S|100} = B(S, \rho/|S|100)$. It then follows that if $f$ is one-bounded, then
$$\mathbb{E}_{n \in B} f(n) = \mathbb{E}_{n \in B} \mathbb{E}_{k \in [-K, K]} f(n + ka) + O\left(K \frac{|S|\|a\|_S}{\rho}\right)$$
whenever $K \le \frac{\rho}{|S|\|a\|_S}$. We may thus write for $a, b, c \in B_{\delta^{O(d)^{O(1)}}/100|S|}$ that
$$|\mathbb{E}_{n, x, y \in B} \mathbb{E}_{k_1 \in [-K_1, K_1], k_2 \in [-K_2, K_2], k_3 \in [-K_3, K_3]}, e(6(T(x + k_1a, y + k_2b, n + k_3c) $$
$$- T(y + k_2b, x + k_1a, n + k_3c)))b(n + k_3c, x + k_1a)b(n + k_3c, y + k_2b)b(x + k_1a, y + k_2b)| \gg \delta^{O(d)^{O(1)}}$$
for $K_1 \le \frac{\delta^{O(d)^{O(1)}}\rho}{\|a\|_S}$, $K_3 \le \frac{\delta^{O(d)^{O(1)}}\rho}{\|b\|_S}$, $K_3 \le \frac{\delta^{O(d)^{O(1)}}\rho}{\|c\|_S}$. We note that if $f$ is locally linear on a Bohr set, and if $n + ka$ lies inside the Bohr set for each $k \in [-K, K]$ (say for $K$ significantly larger than $2$), then we may write $f(n + ka) = f(n) + kf(a)$. Thus, we may write $T(x + k_1a, y + k_2b, n + k_3c)$ as a polynomial in the variables $k_1, k_2, k_3$ with top term $k_1k_2k_3 T(a, b, c)$ and similarly, we may write $T(y + k_2b, x + k_1a, n + k_3c)$ as a polynomial in $k_1, k_2, k_3$ with top term $k_1k_2k_3 T(b, a, c)$. We will now Pigeonhole in values $n, x, y$ so that
$$\mathbb{E}_{k_1 \in [-K_1, K_1], k_2 \in [-K_2, K_2], k_3 \in [-K_3, K_3]}, e(6(T(x + k_1a, y + k_2b, n + k_3c) - T(y + k_2b, x + k_1a, n + k_3c)))$$
$$b(k_1, k_3)b(k_1, k_2)b(k_2, k_3)| \gg \delta^{O(d)^{O(1)}}.$$
Thus, applying trilinearity of $T$, we obtain
$$\mathbb{E}_{k_i \in [-K_i, K_i]} e(6k_1k_2k_3(T(a, b, c) - T(b, a, c)))b(k_1, k_3)b(k_1, k_2)b(k_2, k_3)| \gg \delta^{O(d)^{O(1)}}.$$
Let $\alpha = 6(T(a, b, c) - T(b, a, c))$. Applying Cauchy-Schwarz three times to eliminate all of the $b$'s gives
$$|\mathbb{E}_{k_i, k_i' \in [-K_i, K_i]} e(\alpha P(k_1, k_2, k_3, k_1', k_2', k_3'))| \gg \delta^{O(d)^{O(1)}}$$
where 
$$P(k_1, k_2, k_3, k_1', k_2', k_3') = k_1k_2k_3 - k_1'k_2k_3 - k_1k_2'k_3 - k_1k_2k_3' + k_1'k_2'k_3 + k_1'k_2k_3' + k_1k_2'k_3' - k_1'k_2'k_3'.$$
Pigeonholing in $k_1', k_2', k_3'$, we obtain
$$|\mathbb{E}_{k_i \in [-K_i, K_i]} e(\alpha k_1k_2k_3 + [\text{Lower Order Terms}])| \gg \delta^{O(d)^{O(1)}}.$$
This gives us by multidimensional polynomial equidistribution theorem \cite[Proposition 7]{Tao15} that there exists some $q = q_{a, b, c} \le \delta^{-O(d)^{O(1)}}$ such that
$$\|q\alpha\|_{\mathbb{R}/\mathbb{Z}} \ll \frac{\delta^{-O(d)^{O(1)}}}{K_1K_2K_3}.$$
This implies via a geometry of numbers argument as in \cite[Lemma 11.4]{GT08} or in \cite[Lemma 7.5]{Len22b} that denoting $\epsilon = \delta^{O(d)^{O(1)}}$ that if $a, b, c \in B_{\epsilon^{O(d)^{O(1)}}}$, then for some $q \le \epsilon^{-1}$ independent of $a, b, c$, that
$$\|6q(T(a, b, c) - T(b, c, a))\|_{\mathbb{R}/\mathbb{Z}} \le \|a\|_S\|b\|_S\|c\|_S\delta^{-O(d)^{O(1)}}.$$
This gives the symmetry result. The rest, now, is putting everything together. 
\subsubsection{Integrating the result}
We go back to the hypothesis
$$\mathbb{E}_h |\mathbb{E}_{n \in \mathbb{Z}_N} \Delta_{h + k} f(n)\chi_h(n)|^2 \ge \delta^{O(d)^{O(1)}}.$$
We first claim that there exists some $h_0 \in H$ such that $H \cap (h_0 + B_{\delta^{O(1)}})$ has size at least $\delta^{O(d)^{O(1)}}  N$. To show this, we observe (following \cite[p. 33]{GTZ11}) that
$$\sum_{n \in [N]} 1_H * 1_{B'} * 1_{B'}(n)1_H(n) = \sum_{n \in [N]} (1_H * 1_{B'}(n))^2 \ge \left(\sum_{n \in [N]} 1_H * 1_{B'} (n)\right)^2/N \ge \delta \rho_1^{2d} N^3$$
with $B' = B_{\delta^{O(1)}/2}$. On the other hand, we have that $1_{B'} * 1_{B'}(n) \le |B_{\delta^{O(1)}}| 1_{B_{\delta^{O(1)}}}(n) \le N1_{B_{\delta^{O(1)}}}(n)$. Hence, by the Pigeonhole principle, such $h_0$ exists. Thus, taking $H' = (H - h_0) \cap B$, it follows that for each $h' \in H'$,
$$\mathbb{E}_{h' \in B_{\delta^{O(1)}}} |\mathbb{E}_{n} \overline{f(n)}f(n + k + h_0 + h')e(3T(h_0 + h', n, n) + \theta_{h_0 + h'}n)|^2 \ge \delta.$$
We may use local trilinearity of $T$ and relabeling $\overline{f} = f_1$ and $f_2 = f(h_0 + k + \cdot)$, we have (after adjusting $\theta$ a bit)
$$\mathbb{E}_{h \in B_{\delta^{O(1)}}} |\mathbb{E}_{n} f_1(n)f_2(n + h) e(3T(h, n, n) + \theta_h n)|^2 \ge \delta^{O(d)^{O(1)}}.$$
Pigeonoling $n$ to be in a Bohr set $B_{\delta^{O(1)}}$ and adjusting the radius so that it is regular, we have by the various bracket polynomial manipulations from previous sections that for some phases $\tilde{\theta}_h \in \widehat{\mathbb{Z}_N}$,
$$\mathbb{E}_{h} |\mathbb{E}_{n \in B_{\delta^{O(1)}}} \tilde{f}_1(n)\tilde{f}_2(n + h) e(3T(h, n, n) + \tilde{\theta}_hn)|^2 \ge \delta^{O(d)^{O(1)}}$$
where $\tilde{f}_i$ are products of $f$ and at most $O(d)$ many degree $\le 2$ bracket polynomials which are periodic in $n$ and $h$ modulo $N$. Pigeonholing in $n$ and $h$ in a smaller Bohr set $B_{\delta^{O(d)^{O(1)}}}$ gives
$$\mathbb{E}_{h \in B_{\delta^{O(d)^{O(1)}}}} |\mathbb{E}_{n \in B_{\delta^{O(d)^{O(1)}}}} \tilde{f}_1(n + n_0 + h + h_0)\tilde{f}_2(n + n_0)e(3T(h + h_0, n + n_0, n + n_0) + \tilde{\theta_h}^1 n)|^2 \ge \delta^{O(d)^{O(1)}}.$$
The point is that on $B_{\delta^{O(1)}}$, $T$ is genuinely trilinear, so by \thref{FourierComplexity2}, we may write $3T(h + h_0, n + n_0, n + n_0) = 3T(h, n + n_0, n + n_0)$ plus a lower order bracket quadratic term in $n$, which we can absorb in $\tilde{\theta_h}^1$ and in $\tilde{f}_1$. We may also write $3T(h, n + n_0, n + n_0) = 3T(h, n, n) + 3T(h, n_0, n) + 3T(h, n, n_0) + 3T(h, n_0, n_0)$, but noting that $T$ is symmetric in the last two variables, we have
$$3T(h, n_0, n) + 3T(h, n, n_0) = 3T(n + h, n+ h, n_0) - 3T(n, n, n_0) - 3T(h, h, n_0).$$
In addition, using that $6q(T(h, n, n) - T(n, h, n)) \equiv O(\delta^{O(d)^{O(1)}}) \pmod{1}$, it follows that $$6qT(h, n, n) \equiv 2q(T(n + h, n + h, n+ h) - T(n, n, n) - T(h, h, h) - T(h, n, h)$$
$$ - T(n, h, h) - T(h, h, n)) + O(\delta^{O(d)^{O(1)}}) \pmod{1}.$$
This suggests that we should rewrite the argument as follows: let $\ell \equiv (6q)^{-1} \pmod{N}$. We start with
$$\mathbb{E}_{h} |\mathbb{E}_n \Delta_h f(n) \chi_h(n)|^2 \ge \delta^{O(d)^{O(1)}}.$$
Instead of Pigeonholing in $B_{\delta^{O(1)}}$ By Pigeonholing on a translate of $\ell^{-1}B_{\delta^{O(1)}}$, we obtain:
$$\mathbb{E}_{h} |\mathbb{E}_{n \in \ell^{-1} B_{\delta^{O(1)}}} \tilde{f}_1(n)\tilde{f}_2(n + h) e(3T(6q\ell h, 6q\ell n, 6q\ell n) + \tilde{\theta}_hn)|^2 \gg \delta^{O(d)^{O(1)}}.$$
We may then argue as before, finally obtaining $e((6q)^3 3T(\ell h, \ell n, \ell n)) = e(\tilde{T}(n + h, n + h, n + h) - \tilde{T}(n, n, n) + \tilde{\theta}_h^3 n + \alpha) + O(\delta^{O(d)^{O(1)}})$.  Absorbing the $n + h$ terms into $f(n + h + n_0 + h_0)$ and the $n$ terms into $f(n + n_0)$ and applying \thref{trivialfouriercomplexity} on the terms $\tilde{T}(n, h, h)$, $\tilde{T}(h, n, h)$, and $\tilde{T}(h, h, n)$, we obtain
$$\mathbb{E}_h \|F_1(n + h)F_2(n)\|_{U^2([N])}^2 \ge \delta^{O(d)^{O(1)}}$$
for functions $F_1$ and $F_2$ where $F_2$ is a product of $f$ and $O(d)$ many degree $\le 3$ bracket polynomials. Using the Cauchy-Schwarz-Gowers inequality, combined with the results above on three-step nilmanifold constructions gives us that there exists some approximate degree $\le 3$ nilsequence $F(g(n)\Gamma)$ of dimension at most $O(d)^{O(1)}$ and complexity at most $\delta^{-O(d)^{O(1)}}$ such that
$$|\mathbb{E}_{n \in [N]} \Delta_h f(n) F(g(n)\Gamma)| \ge \delta^{O(d)^{O(1)}}.$$
We may also smooth out $F 1_B$ as follows: writing $1_{B_{\delta^{O(d)^{O(1)}}}}(n) = 1_{U}((\alpha n)_{\alpha \in S})$ where $U$ is the open set $\{\|x_i\|_{\mathbb{R}/\mathbb{Z}} < \rho \delta^{O(d)^{O(1)}}\}$ where $B = B(S, \rho)$, we let $\phi$ be a smooth cutoff function on $\mathbb{R}^{|S|}/\mathbb{Z}^{|S|}$ supported on an $\epsilon$-neighborhood of $U$. By regularity, (assuming $\epsilon$ is sufficiently small) $\phi(n) F(g(n)\Gamma) = 1_B(n)F(g(n)\Gamma) + O_{L^1[N]}(\epsilon)$. One can check that $\phi(n)F(g(n)\Gamma)$ is smooth on a three-step nilmanifold of complexity $O(1)$ and $\phi$ can be chosen in a way so that $\phi(n)F(g(n)\Gamma)$ has Lipschitz parameter at most $O(d\epsilon^{-1})$. Letting $\epsilon = \delta^{O(d)^{O(1)}}$, we obtain the desired correlation with a Lipschitz-smooth nilsequence of parameter at most $\delta^{-O(d)^{O(1)}}$. One can then divide by the Lipschitz constant to ensure a Lipschitz parameter of $1$. \\\\
Substituting $\delta$ with $\delta^{O(\log(1/\delta))^{O(1)}}$ and $d$ with $O(\log(1/\delta))^{O(1)}$ gives us \thref{mainresult4}.
\appendix

\section{Further auxiliary lemmas}
In this section, we shall state auxiliary lemmas we use in the proof of our main theorem. Most of these results come from \cite{GT12}. The first two lemmas are similar to \cite[Proposition 9.2]{GT12} and \cite[Lemma 7.9]{GT12}, respectively.
\begin{lemma}[Factorization lemma I]\thlabel{factorization}
Let $G/\Gamma$ be a nilmanifold of dimension $d$, complexity $M$, and degree $k$ and let $g \in \mathrm{poly}(\mathbb{Z}, G)$ with $g(n)\Gamma$ be $N$-periodic. Suppose $\eta_1, \dots, \eta_r$ are a set of linearly independent nonzero horizontal characters of size at most $L$ with $\|\eta_i \circ g\|_{C^\infty[N]} = 0$ for each character $i$. Then there exists a factorization $g(n) = \epsilon(n)g_1(n)\gamma(n)$ where $\epsilon$ is constant, $g_1$ is a polynomial sequence in $$\tilde{G} = \bigcap_{i = 1}^r \operatorname{ker}(\eta_i),$$
and $\gamma \in \mathrm{poly}(\mathbb{Z}, G)$ is a $(ML)^{O_k(d^{O_k(1)})}$-rational. If $g(0) = \mathrm{id}_G$, then we may take and $\epsilon(0) = \gamma(0) = \mathrm{id}_G$. 
\end{lemma}
\begin{proof}
    By setting $\epsilon(n) = g(0)$, we may work with the assumption that $g(0) = \mathrm{id}_G$. We may thus write in coordinates that
    $$\psi(g(n)) = \sum_i {n \choose i}t_i$$
    where $t_i$ are vectors representing the coordinates of the degree $i$ component of $g$ in Mal'cev coordinates. By Cramer's rule, we may pick a rational vector $v$ with height at most $(ML)^{O(d^2)}$ such that $\eta_i \cdot v = \eta_i \cdot \psi(g(n))$. We define a polynomial sequence $\gamma$ with
    $$\psi(\gamma(n)) = \sum_{i \ge 1} {n \choose i} v_i.$$
    Thus, the polynomial sequence $g_1(n) := g(n)\gamma(n)^{-1}$ lies inside $\tilde{G}$ and by construction $\gamma(n)\Gamma$ is $(ML)^{O_k(d)^{O_k(1)}}$-rational. 
\end{proof}
Before we state the next lemma, we recall the definition of $g_2$ from \cite[Definition 4.1]{Len2} via given a polynomial sequence $g \in \mathrm{poly}(\mathbb{Z}, G)$, we define $g_2(n) := g(n)g(1)^{-n}$.
\begin{lemma}[Factorization lemma II]\thlabel{factorization2}
    Let $G/\Gamma$ be a nilmanifold of dimension $d$, complexity $M$, and degree $k$ and let $g \in \mathrm{poly}(\mathbb{Z}, G)$ with $g(n)\Gamma$ be $N$-periodic. Suppose $\eta_1, \dots, \eta_r$ are a set of linearly independent nonzero horizontal characters on $G_2$ which annihilate $[G, G]$ of size at most $L$ with $\|\eta_i \circ g_2\|_{C^\infty[N]} = 0$ for each character $i$. Then we may write $g(n) = \epsilon(n)g_1(n)\gamma(n)$ where $\epsilon$ is constant, $g_1(n) \in \mathrm{poly}(\mathbb{Z}, \tilde{G})$ with $\tilde{G}$ given the filtration
    $$\tilde{G}_0 = \tilde{G}_1 = G, \tilde{G}_2 = \bigcap_{i = 1}^r \operatorname{ker}(\eta_i), \text{ and } \tilde{G}_i = \tilde{G}_2 \cap G_i$$
    for all $i \ge 2$, and $\gamma$ is $(ML)^{O_k(d^{O(1)})}$-rational, and $g_1(0) = \gamma(0) = 1$. 
\end{lemma}
\begin{proof}
    By setting $\epsilon(n) = g(0)$, we may assume that $g(0) = \mathrm{id}_G$. We write in Mal'cev coordinates that
    $$\psi_{G_2}(g_2(n)) = \sum_{i \ge 2} {n \choose i} t_i$$
    where $t_i$ are vectors representing the coordinates of the degree $i$ component of $g$ in Mal'cev coordinates. By Cramer's rule, we may pick a rational vector $v$ in $G_2$ with height at most $(ML)^{O(r)^2}$ such that $\eta_i \cdot v = \eta_i \cdot \psi(g(n))$ and such that the linear component of $v$, i.e., the first $d - d_1$ components when written as Mal'cev bases, is zero. We define $\gamma$ via
    $$\psi(\gamma(n)) = \sum_{i \ge 2} {n \choose i} v_i.$$
    By construction, $\gamma$ is $(ML)^{O_k(d^{O_k(1)})}$-rational. We now define $g_1(n) = g(n)\gamma(n)^{-1}$. We observe that
    $$g_{1, 2}(n) = g_1(n)g_1(1)^{-n} = g_2(n)\gamma(n)^{-1} \pmod{[G, G]}$$
    so $g_1(n) \in \mathrm{poly}(\mathbb{Z}, \tilde{G})$ as desired.
\end{proof}

\begin{lemma}[Removing the Rational Term]\thlabel{removerational}
    Suppose $N$ is prime and $Q$ is an integer. Suppose $G/\Gamma$ is a nilmanifold of dimension $d$, complexity $M$, degree $K$, and $g(n)$ is a polynomial sequence on $G$ with $g(n)\Gamma$ periodic. Suppose we may write $g(n) = \epsilon g_1(n)\gamma(n)$ with $g_1(0) = \gamma(0) = \mathrm{id}_G$, $g_1$ has image in a $Q$-rational subgroup $G'$ of $G$ and $\gamma$ is $Q$-rational. Then either $N \ll (MQ)^{O_k(d)^{O_k(1)}}$ or there exists some $g'$ and $\gamma'$ such that $g(n) = \epsilon g'(n) \gamma'(n)$ where $g'$ is $N$-periodic and has image in $G'$, and the image of $\gamma'$ lies in $\Gamma$. 
\end{lemma}
\begin{proof}
    Since $\gamma$ is $Q$-rational, it follows from \cite[Lemma B.14]{Len2} that $\gamma(H \cdot)$ has image in $\Gamma$ for some positive integer $H \ll (MQ)^{O_k(d)^{O_k(1)}}$. Suppose it is not true that $N \ll (MQ)^{O_k(d)^{O_k(1)}}$. Then letting $H'$ be the multiplicative inverse of $H$ modulo $N$, it follows that
    $$g(H'Hn) = g(n)\gamma_1(n)$$
    for some polynomial sequence $\gamma_1$ where $\gamma_1(n) \in \Gamma$ for all $n$. In addition, we see that $g_1(H \cdot)\Gamma$ is $N$-periodic since $\gamma(H \cdot)$ has image in $\Gamma$ and $g$ is $N$-periodic. Hence, $g_1(H'H \cdot)\Gamma$ is $N$-periodic. Thus,
    $$g(H'Hn) = g(n)\gamma_1(n) = \epsilon g_1(H'Hn)\gamma(H'Hn).$$
    Letting $g'(n) = g_1(H'Hn)$ and $\gamma'(n) = \gamma(H'Hn)\gamma_1(n)^{-1}$, we obtain the desired result.
\end{proof}

We will also need the following two lemmas about periodic nilsequences:
\begin{lemma}\thlabel{periodictorus}
Let $g\colon \mathbb{Z} \to \mathbb{R}/\mathbb{Z}$ be a polynomial of degree $k$ with $g$ periodic modulo $N > k!$ a prime. Then $\|g\|_{C^\infty[N]} = 0$. 
\end{lemma}
\begin{proof}
This follows from \cite[Lemma 1.4.1]{Tao12}.
\end{proof}

Before we state the below lemma, we recall \cite[Definition 4.1]{Len2} of $g_2(n)$ given a polynomial sequence $g(n)$ with $g(0) = \mathrm{id}_G$: $g_2(n) = g(n)g(1)^{-n}$.
\begin{lemma}\thlabel{periodic}
    Let $N > 100$ be prime and $G/\Gamma$ be a nilmanifold of degree $k$ and $g \in \mathrm{poly}(\mathbb{Z}, G)$ with $g(n)\Gamma$ be a periodic modulo $N$ with $N$ larger than $k!$. Then for any horizontal characters $\eta\colon G/[G, G] \to \mathbb{R}$ and $\eta^2\colon G_2/[G, G_2] \to \mathbb{R}$, we have (abusively descending $g$)
    $$\|N\eta \circ g\|_{C^\infty[N]} = 0 \text{ and } \|k!^2N\eta^2\circ g\|_{C^\infty[N]} = 0$$
\end{lemma}

\begin{proof}
    The first part of the statement follows from \thref{periodictorus}. For the second part, we recall the definition of $G^\square := G \times_{G_2} G$ and $\Gamma^\square := \Gamma \times_{G_2 \cap \Gamma} \Gamma$, from \cite[Definition 4.2]{Len2} (see also the proof of \thref{twostepcase}). For $h \in \mathbb{N}$, we define $g_h(n) = (\{g(1)^h\}^{-1}g_2(n + h)g(1)^{-n}\{g(1)^h\}, g(n))$. It follows that $g_h(n)\Gamma^\square$ is $N$-periodic. \\\\
    Let $\eta$ be a horizontal character on $G^\square$. By \cite[Lemma A.3]{Len2}, we may decompose $\eta(g', g) = \eta^1(g) + \eta^2(g'g^{-1})$ where $\eta^1\colon G/[G, G] \to \mathbb{R}$ and $\eta^2\colon G_2/[G_2, G] \to \mathbb{R}$. Hence,
    $$\eta(g_h(n)) = \eta^1(g(n)) + \eta^2([g(1)^n, \{g(1)^h\}] g_2(n + h)g_2(n)^{-1}).$$
    We see that $\eta^1(g(n)) \equiv 0 \pmod{1/N}$, and $\eta^2([g(1)^n, \{g(1)^h\}]) = \eta^2([g(1)^n, [g(1)^h]]) \equiv 0 \pmod{1/N}$. Hence, it follows that letting $Q(n) = \eta^2(g_2(n))$, we have that
    $$Q(n + h) - Q(n) - Q(h) \equiv 0 \pmod{1/N}.$$
    Let
    $$Q(n) = \sum_{i \ge 2} \alpha_i {n \choose i}.$$
    We see from Vandermonde's identity that
    $${n + h \choose i} = \sum_{0 \le t \le i} {n \choose t}{h  \choose i - t}$$
    so
    $$Q(n + h) = \sum_{t = 0}^k {n \choose t} \sum_{k \ge i \ge \max(t, 2)} \alpha_i {h \choose i - t}$$
    $$Q(n + h) - Q(n) - Q(h) = {n \choose 1} \sum_{i \ge 2}^k \alpha_i {h \choose i - 1} + \sum_{t = 2}^k {n \choose t}\sum_{k \ge i > t} \alpha_i {h \choose i - t}.$$
    It follows that for each $t \ge 2$,
    $$\sum_{k \ge i > t} \alpha_i {h \choose i - t} \equiv 0 \pmod{1/(k!N)}$$
    and from P\'olya's classification of integer polynomials \cite{P20}, it follows that $\alpha_i \equiv 0 \pmod{k! N}$. In addition, we see that
    $$\sum_{i \ge 2}^k \alpha_i {h \choose i - 1} \equiv 0 \pmod{1/(k!N)}$$
    which also implies that $\alpha_2 \equiv 0 \pmod{k!N}$. Changing bases of $Q$ to monomials $n^I$ instead of binomials ${n \choose I}$ costs another factor of $k!$, so
    $$\|k!^2N\eta^2 \circ g\|_{C^\infty[N]} = 0.$$
\end{proof}
\begin{remark}
    Note that it is not necessarily true that $g((10^k k)!n)$ has all of its coordinates with denominator $N$ as the following example shows:
    $$g(n) = \begin{pmatrix} 1 & a\alpha & a\beta & a\alpha\beta \\ 0 & 1 & 0 & \beta \\ 0 & 0 & 1 & \alpha \\ 0 & 0 & 0 & 1 \end{pmatrix}^n$$
    with $a \in \mathbb{Z}$, $\alpha, \beta \in \widehat{\mathbb{Z}/N\mathbb{Z}}$. One can check that if $G$ is the group where we replace $a\alpha, b\beta, \alpha, \beta, a\alpha\beta$ with arbitrary real numbers, and $\Gamma$ is the subgroup where all of the corresponding entries are integers, then $g(n)\Gamma$ is $N$-periodic if $\alpha$ and $\beta$ have denominator $N$. Note that this example corresponds to the $N$-periodic bracket polynomial $n \mapsto e(a\{\alpha n\}\{\beta n\})$. The example also shows that if we were to add a quadratic part to the upper right corner of $g$, in order for the nilsequence to stay $N$-periodic, the coefficient of the quadratic part must have denominator $N$. This is equivalent to saying that if we were to add higher order parts to $g$, then the horizontal component of $g_2$ must have denominator $N$, as the lemma claims.
\end{remark}

\section{Fourier compleity lemmas}
The purpose of this section is to collect results needed in Section 3 and Section 8 regarding ``Fourier expansion" lemmas for bracket polynomials. Before we state and prove them, we require the following definition.
\begin{definition}
We define the $L^p[N]$ (likewise $L^p([N] \times [H])$) $\delta$-\emph{Fourier complexity} of a function $f: [N] \to \mathbb{C}$ to be the infimum of all $L$ such that
$$f(n) = \sum_i a_i e(\xi_in) + g$$
where $\|g\|_{L^p[N]} \le \delta$ and $\sum_i |a_i| = L$.
\end{definition}
\begin{lemma}[Fourier Complexity Lemma]\thlabel{onevarfouriercomplexity}
    Let 
    $$\alpha_1, \dots, \alpha_d, \beta_1, \dots, \beta_d, \gamma_1, \dots, \gamma_d, \gamma_1', \dots, \gamma_d' \in \mathbb{R}$$ 
    and let $\delta > 0$ a real number and $N > 0$ an integer. Then either $N \ll (\delta/2^dk)^{-O(d)^2}$ or else 
    $$e(k_1\{\alpha_1 n + \gamma_1\}\{\beta_1 n + \gamma_1'\} + k_2\{\alpha_2 n + \gamma_2\} \{\beta_2 n + \gamma_2'\} + \cdots k_d\{\alpha_d n + \gamma_d\}\{\beta_dn + \gamma_d'\})$$
    has $L^1([N])$-$\delta$-Fourier complexity at most $(\delta/2^dk)^{-O(d^{2})}$ for $|k_i| \le k$ integers. Furthermore, if $\alpha_i, \beta_i$ are rational with denominator $N$ and $N$ is prime, then the phases in the Fourier expansion can also be taken to be rational with denominator $N$ (and consequently $g$ is periodic modulo $N$).
\end{lemma}
\begin{proof}
Let 
$$\phi(n, h) = k_1\{\alpha_1 n + \gamma_1 \}\{\beta_1 n + \gamma_1'\} + k_2\{\alpha_2 n + \gamma_2\} \{\beta_2 n + \gamma_2'\} + \cdots k_d\{\alpha_dn + \gamma_d\}\{\beta_dn + \gamma_d'\}.$$ 
The function $e(\phi(n, h))$ resembles a degree one nilsequence on a torus $\mathbb{T}^{2d}$ except with possible discontinuities at the endpoints $x_i = \frac{1}{2} \pmod{1}$. To remedy this possible issue, we must set ourselves in a position so that the endpoints contribute very little to the $L^1([N])$ norm. In the periodic nilsequences case, this will end up being true since $N$ is prime and $\alpha_i$ and $\beta_i$ will always have denominator $N$. In this case, one can multiply $e(\phi(n))$ by smooth cutoff supported supported in $(-1/2, 1/2)$ and equal to one on $[-1/2 + (\delta/2^dk), 1/2 - (\delta/2^dk)]^{2d}$ of derivative at most $(\delta/2^dk)^{-1}$ and Fourier expand everything. The terms for which we have ignored from taking cutoff contribute an $L^1([N])$ norm at most $\delta/2^{d - 1}$ since $\{n \in [N]: \|\alpha_i n - \frac{1}{2}\|_{\mathbb{R}/\mathbb{Z}} < \delta\}$ has at most $2\delta N$ many elements (and similarly for each $\beta_i$) since each $\alpha_i$ has denominator $N$. Since each Fourier phase is an integer combination of the $\alpha_i$'s and $\beta_i$'s, it follows that each element of the Fourier expansion is rational of denominator $N$. Thus, we have obtained the desired Fourier complexity bounds with the desired properties. For the sake of this paper, we will only need the periodic case; for completeness sake, we give a proof of the more general Fourier complexity lemma. \\\\
In the general case, this is proven via an induction on dimensions procedure so our goal is to show that this does not lead to double exponential losses in parameters. This involves carefully quantifying the procedure in \cite[Appendix E]{GTZ11} and making sure that losses are only single exponential in dimension. Morally, the reason why we are able to obtain single exponential bounds in dimension is that we only desire equidistribution on each of the phases $\alpha_i, \beta_i$ individually, and not the joint equidistribution in $\mathbb{R}^{2d}$; this makes it so that the bounds don't feed back into themselves in iteration. For example, in the case where $N$ is prime and each $\alpha_i$ are equal, the suggested procedure works without any further modification since $\alpha n$ for $\alpha$ nonzero is well-equidistributed in $S^1$. Let $1 > \epsilon > 0$ be a quantity to be determined later. If $\{ n \in [N]: \|\alpha_1 n + \gamma_1 - \frac{1}{2}\|_{\mathbb{R}/\mathbb{Z}} < \epsilon\}$ has more than $\sqrt{\epsilon}N$ many elements, then an application of Vinogradov's Lemma shows that there exists some integer $q_1 \le \sqrt{\epsilon}^{-1}$ such that $\|q_1 \alpha_1\|_{\mathbb{R}/\mathbb{Z}} \le 100\frac{\sqrt{\epsilon}}{H}$. We now divide $[N]$ into progressions of common difference $q_1$ such that on each progression, $\{\alpha_1 n\}$ varies by at most $\epsilon^{1/2}$. This can be done as follows:
\begin{itemize}
    \item First, divide $[N]$ into progressions of common difference $q_1$.
    \item On each progression, we have that $\alpha_1 n$ varies by at most $\sqrt{\epsilon}$ modulo $1$; however, it is possible that $\{\alpha_1 n + \gamma_1\}$ ``lapped over" $1/2$ along the progression, so we must divide each progression into two subprogressions, one before $\{\alpha_1 n\}$ lapped over $1/2$ and one after it lapped over $1/2$. 
    \item The number of progressions we obtain is at most $200q_1$.
\end{itemize}
Thus, we may divide $[N] = \bigsqcup P_i$ into such progressions. We shall want to ``prune out" the small progressions and write $[N] = \bigsqcup P_i \sqcup E$ where $P_i'$ are sufficiently large and $E$ is a small error. If $|P_i| \le \epsilon^{1 + 1/2} N$, then we include it in $E$ and otherwise, we do nothing. Thus, we can ensure that $|P_i| \ge \epsilon^{1 + 1/2} N$ and $|E| \le 200q_1\epsilon^{1 + 1/2} N \le 200\epsilon N$. \\\\
We now iterate in the following manner: if $\{ n \in [N]: \|\alpha_2 n + \gamma_2 - \frac{1}{2}\|_{\mathbb{R}/\mathbb{Z}} < \epsilon\}$ has more than $\sqrt{\epsilon}N$ many elements, then (assuming that $\epsilon$ is sufficiently small) there exists a progression $P$ (as in the decomposition above) such that $\{h \in P: \|\alpha_2 n + \gamma_2 - \frac{1}{2}\|_{\mathbb{R}/\mathbb{Z}} < \epsilon\}$ has more than $\frac{\sqrt{\epsilon}}{2} |P|$ many elements, then writing $P = q_1 \cdot [N_1] + r_1$, it follows that there are more than $\frac{\sqrt{\epsilon}}{2}N_1$ many elements in $\{n \in [N_1]: \|q_1\alpha_2n + \gamma\|_{\mathbb{R}/\mathbb{Z}} < \epsilon\}$ for some $\gamma$. Letting $n$ and $m$ be two such elements, it follows that $\|q_1\alpha_2(n - m)\|_{\mathbb{R}/\mathbb{Z}} \le \|-\gamma - q_1\alpha_2 m\|_{\mathbb{R}/\mathbb{Z}} + \|q_1\alpha_2n + \gamma\|_{\mathbb{R}/\mathbb{Z}} \le 2\epsilon$, so there are more than $\frac{\sqrt{\epsilon}}{2}|P|$ many elements such $n \in [N_1]$ such that $\|q_1\alpha_2\|_{\mathbb{R}/\mathbb{Z}} < 2\epsilon$. Thus, (assuming that $N$ is sufficiently large with respect to $\epsilon$) there exists some $q_2 \le 2\sqrt{\epsilon}^{-1}$ such that $\|q_1q_2 \alpha_2\|_{\mathbb{R}/\mathbb{Z}} \le \frac{200\sqrt{\epsilon}}{|P|}$. We now continue down a similar path, dividing each $P_i$ into subprogressions instead of dividing $[N]$ into subprogressions, this time obtaining an error set of size at most $4\epsilon|P_i|$. \\\\
Under an iteration (making sure to choose $\epsilon$ so that $\frac{\sqrt{\epsilon}}{200^{2d}} > \epsilon$), we would obtain at most $\epsilon^{-d}200^{2d}$ many progressions each of size at least $\epsilon^{3d}N$ and common difference at most $\epsilon^{-d}$ such that
\begin{itemize}
    \item either $n \mapsto \{\alpha_i n + \gamma_i\}$ or $n \mapsto \{\beta_i n + \gamma_i'\}$ varies by at most $\epsilon^{1/2}$ on each progression;
    \item or the set $\{ n \in [N]: \|\alpha_i n + \gamma_i - \frac{1}{2}\|_{\mathbb{R}/\mathbb{Z}} < \epsilon\}$ has at most $\sqrt{\epsilon}N$ many elements (or similarly the set $\{n \in [N]: \|\beta_i n + \gamma_i' - \frac{1}{2}\|_{\mathbb{R}/\mathbb{Z}} < \epsilon\}$ has at most $\sqrt{\epsilon}N$ many elements);
\end{itemize}
and an error set $E$ of size at most $200^{2d} \epsilon 2N$. For the $\alpha_i$'s or $\beta_i$'s that lie in the second case, we introduce a smooth cutoff $\varphi_{j}$ of derivative at most $\epsilon^{-1}$ in the dimensions that correspond to the $\alpha_i$'s or $\beta_i$'s. Thus, letting $F(\vec{x}, \vec{y}) = e(k_1 x_1y_1 + \cdots + k_d x_dy_d)$ and letting the above decomposition be $[N] = \bigsqcup_i P_i \sqcup E$, it follows that we may express
$$e(\phi(n, h)) = F(\{\alpha n\}, \{\beta n\}) = \varphi(\{\alpha n\}, \{\beta n\}) F(\{\alpha n\}, \{\beta n\}) + O_{L^1([N]), \le 1}(2d\epsilon^{1/2})$$
where $O_{L^1([N]), \le C}(k)$ denotes a function in $L^1([N])$ whose norm is at most $Ck$ (i.e., $\le 1$ denotes that the implicit constant is $\le 1$). We now take
$$\phi(\{\alpha n\}, \{\beta n\}) F(\{\alpha n\}, \{\beta n\}) = \sum_i 1_{P_i}(n)\varphi(\{\alpha n\}, \{\beta n\}) F(\{\alpha n\}, \{\beta n\}) + O_{L^1([N]), \le 1}(200^{2d}2\epsilon).$$
For each of the remaining phases $\alpha_i$, $\beta_j$ that are not covered by first case in the above bulleted points, we may take some phase $\alpha_i^0$ or $\beta_j^0$ such that $|\{\alpha_in + \gamma_i\} - \{\alpha_i^0\}| \le \epsilon^{1/2}$. Making sure to choose $\epsilon$ so that $k\epsilon^{1/2} < 1/100$, it follows that we may replace $F(\{\alpha n\}, \{\beta n\})$ by some $\tilde{F}(\{\alpha n\}, \{\beta n\})$ which is equal to $\{\alpha_i^0\}$ and $\{\beta_j^0\}$ on please $i$ in the $\vec{x}$ variable and place $j$ in the $\vec{y}$ variable. We thus have
$$\sum_i 1_{P_i}(n)\varphi(\{\alpha n\}, \{\beta n\}) F(\{\alpha n\}, \{\beta n\}) $$
$$= \sum_i 1_{P_i}(n)\varphi(\{\alpha n\}, \{\beta n\}) \tilde{F}(\{\alpha n\}, \{\beta n\}) + O_{L^\infty([N]), \le 1}(2dk \epsilon^{1/2})$$
(here, we are also implicitely using the inequality $|e(\epsilon) - 1| \le |\epsilon|$). Now, each $P_i$ can be written as a product of $P_i^1 \times P_i^2$ (one in $h$ variable, one in $n$ variable). We may also write $P_i^j$ as a product of an interval $1_I$ and some $1_{n \equiv a \pmod{q}}$. The interval we can approximate in $L^1[N]$ (or respectivly, $L^1[N]$) with a $g(n/N)$ where $g$ is a Lipschitz function with norm at most $\epsilon^{-1/2}$ up to an error of at most $2\epsilon^{1/2}|I|$. Each $1_{n \equiv a \pmod{q}}$ can also be thought of as a Lipschitz function on the torus with norm at most $q$ (which in our case will be at most $\epsilon^{-d}$). Thus, we may write:
$$\sum_i \tilde{F}_{P_i}(\{\alpha n\}, \{\beta n\}) + O_{L^1([N]), \le 1}(4\epsilon^{1/2})$$
where each $\tilde{F}_{P_i}$ is a product of $\tilde{F}$ and $g$ and $1_{n \equiv a \pmod{q}}$ functions associated to the progression $P_i$ and has Lipchitz norm at most $(2d + 4)\epsilon^{-1/2}$ and is defined on a $2d + 4$ dimensional torus $\mathbb{R}^{2d + 4}/\mathbb{Z}^{2d + 4}$. Fourier expanding each $\tilde{F}_{P_i}$ via \cite[Lemma A.6]{Len2} (taking $\delta = \frac{\epsilon^{d + 3/2}}{3(2d + 4)200^{2d}}$), we obtain
$$\sum_i \tilde{F}_{P_i}(\{\alpha n\}, \{\beta n\}) = \sum_{\xi} a_\xi e(\xi \cdot n) + O_{L^\infty([N]), \le 1}(\epsilon)$$
where (in fact we can take $|a_\xi| \le 1$) 
$$\sum_\xi |a_\xi| \le O(\epsilon)^{-O(d)^2}.$$
Summing up all the errors together gives an $L^1([N])$ error of at most $O(dk\epsilon^{1/2})$. Now choosing $\epsilon$ so that (each of the conditions we placed on $\epsilon$ are satisfied before and) $dk\epsilon^{1/2} \ll \delta$, we obtain the claim.
\end{proof}

\begin{lemma}[Bilinear Fourier Complexity Lemma]\thlabel{bilinearfouriercomplexity}
    Let 
    $$\alpha_1, \dots, \alpha_d, \beta_1, \dots, \beta_d, \gamma_1, \dots, \gamma_d, \gamma_1', \dots, \gamma_d' \in \mathbb{R}$$ 
    and let $\delta > 0$ a real number and $N, H > 0$ integers. Then either $N \ll (\delta/3^dk)^{-O(d)^2}$, or $H \ll (\delta/3^dk)^{-O(d)^2}$ or else 
    $$e(k_1\{\alpha_1 h + \gamma_1\}\{\beta_1 n + \gamma_1'\} + k_2\{\alpha_2 h + \gamma_2\} \{\beta_2 n + \gamma_2'\} + \cdots k_d\{\alpha_dn + \gamma_d\}\{\beta_dh + \gamma_d'\})$$
    has $L^1([N] \times [H])$-$\delta$-Fourier complexity at most $(\delta/3^dk)^{-O(d^{2})}$ for $|k_i| \le k$ real. Furthermore, if $N, H$ are prime and $\alpha_i$ are rational with denominator $N$ and $\beta_i$ are rational with denominator $H$, then the phases in the Fourier expansion can be taken to be $e(c n + dh)$ where $c$ is rational with denominator $N$ and $d$ is rational with denominator $H$.
\end{lemma}
\begin{proof}
The proof is very similar to that of \thref{onevarfouriercomplexity}. Let 
$$\phi(n, h) = k_1\{\alpha_1 h + \gamma_1 \}\{\beta_1 n + \gamma_1'\} + k_2\{\alpha_2 h + \gamma_2\} \{\beta_2 n + \gamma_2'\} + \cdots k_d\{\alpha_dh + \gamma_d\}\{\beta_dn + \gamma_d'\}.$$ 
The function $e(\phi(n, h))$ resembles a degree one nilsequence on a torus $\mathbb{T}^{2d}$ except with possible discontinuities at the endpoints $x_i = \frac{1}{2} \pmod{1}$. To remedy this possible issue, we must set ourselves in a position so that the endpoints contribute very little to the $L^1([N] \times [H])$ norm. This is not problematic if $H = N$ and $\alpha_i, \beta_i$ are rational with denominator $N$ with $N$ prime. By arguing similarly as \thref{onevarfouriercomplexity}, we obtain the desired conclusion in this case. \\\\
Let $1 > \epsilon > 0$ be a quantity to be determined later. If $\{h \in [H]: \|\alpha_1 h + \gamma_1 - \frac{1}{2}\|_{\mathbb{R}/\mathbb{Z}} < \epsilon\}$ has more than $\sqrt{\epsilon}H$ many elements, then an application of Vinogradov's Lemma shows that there exists some integer $q_1 \le \sqrt{\epsilon}^{-1}$ such that $\|q_1 \alpha_1\|_{\mathbb{R}/\mathbb{Z}} \le 100\frac{\sqrt{\epsilon}}{H}$. We now divide $[H]$ into progressions of common difference $q_1$ such that on each progression, $\{\alpha_1 h\}$ varies by at most $\epsilon^{1/2}$. This can be done as follows:
\begin{itemize}
    \item First, divide $[H]$ into progressions of common difference $q_1$.
    \item On each progression, we have that $\alpha_1 h$ varies by at most $\sqrt{\epsilon}$ modulo $1$; however, it is possible that $\{\alpha_1 h + \gamma_1\}$ ``lapped over" $1/2$ along the progression, so we must divide each progression into two subprogressions, one before $\{\alpha_1 h\}$ lapped over $1/2$ and one after it lapped over $1/2$. 
    \item The number of progressions we obtain is at most $200q_1$.
\end{itemize}
Thus, we may divide $[H] = \bigsqcup P_i$ into such progressions. We shall want to ``prune out" the small progressions and write $[H] = \bigsqcup P_i \sqcup E$ where $P_i'$ are sufficiently large and $E$ is a small error. If $|P_i| \le \epsilon^{1 + 1/2} H$, then we include it in $E$ and otherwise, we do nothing. Thus, we can ensure that $|P_i| \ge \epsilon^{1 + 1/2} N$ and $|E| \le 200q_1\epsilon^{1 + 1/2} H \le 200\epsilon H$. \\\\
We now iterate in the following manner: if $\{h \in [H]: \|\alpha_2 h + \gamma_2 - \frac{1}{2}\|_{\mathbb{R}/\mathbb{Z}} < \epsilon\}$ has more than $\sqrt{\epsilon}H$ many elements, then (assuming that $\epsilon$ is sufficiently small) there exists a progression $P$ (as in the decomposition above) such that $\{h \in P: \|\alpha_2 h + \gamma_2 - \frac{1}{2}\|_{\mathbb{R}/\mathbb{Z}} < \epsilon\}$ has more than $\frac{\sqrt{\epsilon}}{2} |P|$ many elements, then writing $P = q_1 \cdot [H_1] + r_1$, it follows that there are more than $\frac{\sqrt{\epsilon}}{2}N_1$ many elements in $\{h \in [H_1]: \|q_1\alpha_2h + \gamma\|_{\mathbb{R}/\mathbb{Z}} < \epsilon\}$ for some $\gamma$. Letting $n$ and $m$ be two such elements, it follows that $\|q_1\alpha_2(n - m)\|_{\mathbb{R}/\mathbb{Z}} \le \|-\gamma - q_1\alpha_2 m\|_{\mathbb{R}/\mathbb{Z}} + \|q_1\alpha_2n + \gamma\|_{\mathbb{R}/\mathbb{Z}} \le 2\epsilon$, so there are more than $\frac{\sqrt{\epsilon}}{2}|P|$ many elements such $h \in [H_1]$ such that $\|q_1\alpha_2\|_{\mathbb{R}/\mathbb{Z}} < 2\epsilon$. Thus, (assuming that $N$ is sufficiently large with respect to $\epsilon$) there exists some $q_2 \le 2\sqrt{\epsilon}^{-1}$ such that $\|q_1q_2 \alpha_2\|_{\mathbb{R}/\mathbb{Z}} \le \frac{200\sqrt{\epsilon}}{|P|}$. We now continue down a similar path, dividing each $P_i$ into subprogressions instead of dividing $[N]$ into subprogressions, this time obtaining an error set of size at most $4\epsilon|P_i|$. \\\\
Under an iteration (making sure to choose $\epsilon$ so that $\frac{\sqrt{\epsilon}}{200^{2d}} > \epsilon$), we would obtain at most $\epsilon^{-d}200^{2d}$ many progressions each of size at least $\epsilon^{3d}NH$ and common difference at most $\epsilon^{-d}$ such that
\begin{itemize}
    \item either $h \mapsto \{\alpha_i h + \gamma_i\}$ or $n \mapsto \{\beta_i n + \gamma_i'\}$ varies by at most $\epsilon^{1/2}$ on each progression;
    \item or the set $\{h \in [H]: \|\alpha_i h + \gamma_i - \frac{1}{2}\|_{\mathbb{R}/\mathbb{Z}} < \epsilon\}$ has at most $\sqrt{\epsilon}H$ many elements (or similarly the set $\{n \in [N]: \beta_i n + \gamma_i' - \frac{1}{2}\|_{\mathbb{R}/\mathbb{Z}} < \epsilon\}$ has at most $\sqrt{\epsilon}N$ many elements);
\end{itemize}
and an error set $E$ of size at most $200^{2d} \epsilon 2NH$. For the $\alpha_i$'s or $\beta_i$'s that lie in the second case, we introduce a smooth cutoff $\varphi_{j}$ of derivative at most $\epsilon^{-1}$ in the dimensions that correspond to the $\alpha_i$'s or $\beta_i$'s. Thus, letting $F(\vec{x}, \vec{y}) = e(k_1 x_1y_1 + \cdots + k_d x_dy_d)$ and letting the above decomposition be $[N] \times [H] = \bigsqcup_i P_i \sqcup E$, it follows that we may express
$$e(\phi(n, h)) = F(\{\alpha h\}, \{\beta n\}) = \varphi(\{\alpha h\}, \{\beta n\}) F(\{\alpha h\}, \{\beta n\}) + O_{L^1([N] \times [H]), \le 1}(2d\epsilon^{1/2})$$
where $O_{L^1([N] \times [H]), \le C}(k)$ denotes a function in $L^1([N] \times [H])$ whose norm is at most $Ck$ (i.e., $\le 1$ denotes that the implicit constant is $\le 1$). We now take
$$\phi(\{\alpha h\}, \{\beta n\}) F(\{\alpha h\}, \{\beta n\}) = \sum_i 1_{P_i}(n, h)\varphi(\{\alpha h\}, \{\beta n\}) F(\{\alpha h\}, \{\beta n\}) + O_{L^1([N] \times [H]), \le 1}(200^{2d}2\epsilon).$$
For each of the remaining phases $\alpha_i$, $\beta_j$ that are not covered by first case in the above bulleted points, we may take some phase $\alpha_i^0$ or $\beta_j^0$ such that $|\{\alpha_ih + \gamma_i\} - \{\alpha_i^0\}| \le \epsilon^{1/2}$. Making sure to choose $\epsilon$ so that $k\epsilon^{1/2} < 1/100$, it follows that we may replace $F(\{\alpha h\}, \{\beta n\})$ by some $\tilde{F}(\{\alpha h\}, \{\beta n\})$ which is equal to $\{\alpha_i^0\}$ and $\{\beta_j^0\}$ on please $i$ in the $\vec{x}$ variable and place $j$ in the $\vec{y}$ variable. We thus have
$$\sum_i 1_{P_i}(n, h)\varphi(\{\alpha h\}, \{\beta n\}) F(\{\alpha h\}, \{\beta n\}) $$
$$= \sum_i 1_{P_i}(n, h)\varphi(\{\alpha h\}, \{\beta n\}) \tilde{F}(\{\alpha h\}, \{\beta n\}) + O_{L^\infty([N] \times [H]), \le 1}(2dk \epsilon^{1/2})$$
(here, we are also implicitely using the inequality $|e(\epsilon) - 1| \le |\epsilon|$). Now, each $P_i$ can be written as a product of $P_i^1 \times P_i^2$ (one in $h$ variable, one in $n$ variable). We may also write $P_i^j$ as a product of an interval $1_I$ and some $1_{n \equiv a \pmod{q}}$. The interval we can approximate in $L^1[N]$ (or respectivly, $L^1[H]$) with a $g(n/N)$ where $g$ is a Lipschitz function with norm at most $\epsilon^{-1/2}$ up to an error of at most $2\epsilon^{1/2}|I|$. Each $1_{n \equiv a \pmod{q}}$ can also be thought of as a Lipschitz function on the torus with norm at most $q$ (which in our case will be at most $\epsilon^{-d}$). Thus, we may write:
$$\sum_i \tilde{F}_{P_i}(\{\alpha h\}, \{\beta n\}) + O_{L^1([N] \times [H]), \le 1}(4\epsilon^{1/2})$$
where each $\tilde{F}_{P_i}$ is a product of $\tilde{F}$ and $g$ and $1_{n \equiv a \pmod{q}}$ functions associated to the progression $P_i$ and has Lipchitz norm at most $(2d + 4)\epsilon^{-1/2}$ and is defined on a $2d + 4$ dimensional torus $\mathbb{R}^{2d + 4}/\mathbb{Z}^{2d + 4}$. Fourier expanding each $\tilde{F}_{P_i}$ via \cite[Lemma A.5]{Len2} (taking $\delta = \frac{\epsilon^{d + 3/2}}{3(2d + 4)200^{2d}}$), we obtain
$$\sum_i \tilde{F}_{P_i}(\{\alpha h\}, \{\beta n\}) = \sum_{\xi} a_\xi e(\xi \cdot (n, h)) + O_{L^\infty([N] \times [H]), \le 1}(\epsilon)$$
where (in fact we can take $|a_\xi| \le 1$) 
$$\sum_\xi |a_\xi| \le O(\epsilon)^{-O(d)^2}.$$
Summing up all the errors together gives an $L^1([N] \times [H])$ error of at most $O(dk\epsilon^{1/2})$. Now choosing $\epsilon$ so that (each of the conditions we placed on $\epsilon$ are satisfied before and) $dk\epsilon^{1/2} \ll \delta$, we obtain the claim.
\end{proof}

We will also need a related lemma below:
\begin{lemma}[``Trivial" Fourier Complexity Lemma]\thlabel{trivialfouriercomplexity}
Let $a, \alpha \in \mathbb{R}^d$m $1/10 > \delta > 0$ real, and $N > 0$ a prime. Suppose that $a$ and $\alpha$ are rational with denominator $N$. Then either $N \ll \delta^{-O(d)^{O(1)}}$ or
$$e(a_1\{\alpha_1 n\} + \cdots + a_d\{\alpha_dn\})$$
has $L^1[N]$ $\delta$-Fourier complexity at most $\delta^{-O(d)^{O(1)}}$ and one can take the phases in the definition of Fourier complexity to be rational with denominator $N$.
\end{lemma}
\begin{proof}
We separate $a = [a] + \{a\}$. Thus,
$$a_1\{\alpha_1 n\} + \cdots + a_d\{\alpha_dn\} \equiv ([a] \cdot \alpha) n + \{a\} \cdot \{\alpha n\} \pmod{1}.$$
The function $e(\{a\} \cdot \{\alpha n\} + [a] \cdot \alpha n)$ resembles a function on $\mathbb{R}^{d + 1}/\mathbb{Z}^{d + 1}$ which is smooth everywhere except on the boundary. To remedy this issue, we multiply this function by a smooth cutoff $\phi$ which is supported in $(-1/2, 1/2)$ and equal to one on $[-1/2 + (\delta/4d), 1/2 + (\delta/4d)]$ (and has derivative at most $8d^2/\delta$). The point is that outside the support of $\phi$, the measure of $((\{\alpha n\})_{n \in \mathbb{R}^d/\mathbb{Z}^d}, \alpha n)$ is at most $\delta/2$, and when we multiplied by $\phi$, the function becomes smooth with Lipschitz constant at most $O(8d^3/\delta)$. Thus, Fourier expanding this function, we notice that the Fourier phases are an integer linear combination of the $\alpha_i$'s, which ensures both the Fourier complexity and the rational with denominator $N$ properties of the Fourier expansion.
\end{proof}
\begin{remark}
Again, one can prove this lemma in general, not assuming that $\alpha_i$ have denominator $N$; we did not do so since this is rather cumbersome.
\end{remark}
We need one last Fourier complexity lemma.
\begin{lemma}[Fourier Complexity Lemma II]\thlabel{FourierComplexity2}
Either $\delta^{-O(d)^{O(1)}} \ge N$ or one may write
$$e(\sum_i \{a_i(\vec{n})\}\{b_i(\vec{m})\}) = \sum_{\alpha, \beta \in \frac{1}{2}\mathbb{Z}^{2d}, |\alpha|,|\beta| \le \delta^{-O(d)^{O(1)}}} a_{\alpha, \beta} e(\alpha \cdot \{a_i(\vec{n})\} + \beta \cdot \{b_i(\vec{m})\}) + O_{\le 1}(\delta)$$
where $|a_{\alpha, \beta}| \le 1$. 
\end{lemma}

\begin{proof}
We follow the procedure in \cite[Appendix E]{GTZ11}. We take the function $f: [-1, 1]^{2d} \to \mathbb{C}$ by defining it to be $e(\sum_{i = 1}^d x_i y_i)$ on $[-1/2, 1/2]$ and extending it to a function on $[-1, 1]$ with compact support. We then Fourier expand this function and plug in $x_i = \{a_i(\vec{n})\}$ and $y_i = \{b_i(\vec{m})\}$ and finish.
\end{proof}

\section{A proof of the refined bracket polynomial lemma}
In this section, we shall give another proof of the refined bracket polynomial lemma from \cite{Len2}, although our proof is more complicated and our bounds are slightly weaker (and for brevity, we only prove the single parameter case). We believe that our proof better motivates the statement of the refined bracket polynomial lemma and would like to emphasize that the bounds obtained here are single exponential in dimension, and thus are enough for any application here as well as for \cite{Len2}. Below is the version of the refined bracket polynomial lemma we work with.
\begin{lemma}[Single Variable Refined Bracket Polynomial Lemma]\thlabel{refinedappendix1}
Let $N, \delta, M, K > 0$ be fixed with $0 < \delta < 1/10$ and $a, \alpha \in \mathbb{R}^d$ and $\beta \in \mathbb{R}$, with $|a| \le M$. Suppose there are at least $\delta N$ many $n \in [N]$ such that
$$\|\beta + a \cdot \{\alpha n\}\|_{\mathbb{R}/\mathbb{Z}} \le \frac{K}{N}.$$
Then either $N \ll (MK/\delta)^{O(d)^{O(1)}}$ or else there exists $d \ge r \ge 0$, $w_1, \dots, w_r \in \mathbb{Z}^d$ and $\eta_1, \dots, \eta_{d - r} \in \mathbb{Z}^d$ such that $w_i, \eta_j$ are linearly independent, $|w_i|, |\eta_j| \le (\delta/M)^{-O(d)^{O(1)}}$, $\langle w_i, \eta_j \rangle = 0$, and
$$|w_i \cdot a| \le \frac{(\delta/MK)^{-O(d)^{O(1)}}}{N}$$
$$\|\eta_j \cdot \alpha \|_{\mathbb{R}/\mathbb{Z}} \le \frac{(\delta/MK)^{-O(d)^{O(1)}}}{N}.$$
\end{lemma}
As in \cite[Section 3]{Len2} and below in Section 3, we deduce the above from the following:
\begin{lemma}\thlabel{refineddenominator}
Let $N, \delta, M, K > 0$ be fixed with $0 < \delta < 1/10$ and $a, \alpha \in \mathbb{R}^d$ and $\beta \in \mathbb{R}$, with $|a| \le M$ and $\alpha$ having denominator $N$. Suppose there are at least $\delta N$ many $n \in [N]$ such that
$$\|\beta + a \cdot \{\alpha n\}\|_{\mathbb{R}/\mathbb{Z}} \le \frac{K}{N}.$$
Then either $N \ll (MK/\delta)^{O(d)^{O(1)}}$ or else there exists $d \ge r \ge 0$, $w_1, \dots, w_r \in \mathbb{Z}^d$ and $\eta_1, \dots, \eta_{d - r} \in \mathbb{Z}^d$ such that $w_i, \eta_j$ are linearly independent, $|w_i|, |\eta_j| \le (\delta/M)^{-O(d)^{O(1)}}$, $\langle w_i, \eta_j \rangle = 0$, and
$$|w_i \cdot a| \le \frac{(\delta/MK)^{-O(d)^{O(1)}}}{N}$$
$$\|\eta_j \cdot \alpha \|_{\mathbb{R}/\mathbb{Z}} = 0.$$
\end{lemma}
We deduce this from:
\begin{lemma}\thlabel{refinedappendix2}
Let $N, N_1, \delta,M,  K > 0$ fixed with $0 < \delta < 1/10$, $M > 1$, and $a, \alpha \in \mathbb{R}^d$ and $\beta \in \mathbb{R}$, with $|a| \le M$ and $\alpha$ having denominator $N_1$, a prime. Suppose there are at least $\delta N$ many $h \in [N]$ such that
$$|\mathbb{E}_{n \in [N]} e(\beta n + an \cdot \{\alpha h\})| \ge K^{-1}.$$
Then either $N \ll (MK/\delta)^{O(d)^{O(1)}}$ or else there exists $d \ge r \ge 0$, $w_1, \dots, w_r \in \mathbb{Z}^d$ and $\eta_1, \dots, \eta_{d - r} \in \mathbb{Z}^d$ such that $w_i, \eta_j$ are linearly independent, with $\langle w_i, \eta_j \rangle = 0$, $|w_i|, |\eta_j| \le (\delta/M)^{-O(d)^{O(1)}}$, and
$$|w_i \cdot a| \le \frac{(\delta/(MK))^{-O(d)^{O(1)}}}{N}$$
$$\|\eta_j \cdot \alpha \|_{\mathbb{R}/\mathbb{Z}} = 0.$$
\end{lemma}
\begin{proof}[Proof of \thref{refinedappendix1} assuming \thref{refinedappendix2}]
 We note that
$$\|\beta + a \cdot \{\alpha h\}\|_{\mathbb{R}/\mathbb{Z}} \le \frac{K}{N}$$
implies that for $n \in [N_1]$ with $N_1 \le \epsilon N$ implies (since $\sqrt{(\cos(x) - 1)^2 + \sin(x)^2} \le |x|$ which rearranges to $|x|/2 \ge |\sin(x/2)|$) that
$$|e(\beta n + an \cdot \{\alpha h\}) - 1| \le \frac{nK}{N} \le \epsilon K.$$
This implies that
$$\mathbb{E}_{n \in N_1}|e(\beta n + an \cdot \{\alpha h\}) - 1|^2 \le \epsilon^2 K^2$$
so expanding out and bounding $\text{Re}(a) \le |a|$, we have
$$2 - 2|\mathbb{E}_{n \in [N_1]} e(\beta n + an \cdot \{\alpha h\})| \le \epsilon^2 K^2$$
and thus
$$|\mathbb{E}_{n \in [N_1]} e(\beta n + an \cdot \{\alpha h\})| \ge 1 - \frac{\epsilon^2 K^2}{2}.$$
Letting $\epsilon = K^{-1}$ gives us
$$|\mathbb{E}_{n \in [N_1]} e(\beta n + an \cdot \{\alpha h\})| \ge \frac{1}{2}.$$
Thus, applying \thref{refinedappendix2}, we obtain linearly independent $w_1, \dots, w_r$ and $\eta_1, \dots, \eta_{d - r}$ such that
$$|w_i \cdot a| \le \frac{(\delta/ 3^dM)^{-O(d)^{O(1)}}}{N}$$
$$\|\eta_j \cdot \alpha \|_{\mathbb{R}/\mathbb{Z}} = 0$$
as desired.   
\end{proof} 
It remains to deduce \thref{refinedappendix2}. Before we do so, we offer an informal description of the argument.
\subsection{A brief description of the argument}
Since the proof of the refined bracket polynomial lemma is rather complicated, we will give a simplified description of them here. Suppose
$$\|a \cdot \{\alpha h\} + \beta\|_{\mathbb{R}/\mathbb{Z}} \approx 0$$
where where $a, \alpha, \beta$ are as in \thref{iterationlemma}. An attempt to prove these refined bracket polynomial lemmas comes by iterating \cite[Proposition 5.3]{GT12}. The conclusion of \cite[Proposition 5.3]{GT12} shows that either $|a| \approx 0$ or there exists some integer vector $\eta \ll_{\delta, M} 1$ such that $\|\eta \cdot \alpha\|_{\mathbb{R}/\mathbb{Z}} = 0$. Let us assume for simplicity that the first coordinate of $\eta$, denoted as $\eta_1$ is equal to $1$. We thus have
$$\alpha_1  + \eta_2 \alpha_2 + \eta_3 \alpha_3 + \cdots + \eta_d \alpha_d \equiv 0 \pmod{1}.$$
Thus, we have
$$\|\tilde{a} \cdot \{\alpha h\} + \gamma h + \beta + a_1P(h)\|_{\mathbb{R}/\mathbb{Z}} \approx 0$$
where $\tilde{a} = (0, a_2 \eta_1 - a_1 \eta_2, a_3 \eta_1 - a_1\eta_3, \dots, a_d \eta_1 - a_1 \eta_d)$ and
$$P(h) = \{\alpha_1 h\} + \eta_2\{\alpha_2 h\} + \cdots + \eta_d \{\alpha_d h\}.$$
$P(h)$ takes at most $2d|\eta|$ many values so by Pigeonholing in one of those values, we can iterate this procedure to obtain a qualitative analogue of \thref{refinedappendix1}. \\\\
There are two problems that occur if one tries to iterate in that manner. The first problem is that $|\tilde{a}|$ might be far larger than $|a|$ causing us to need to increase the parameter $M$ when we apply \thref{iterationlemma} again. We certainly have the bound of $|\tilde{a}| \le 2|a||\eta|$, but this is not good enough for our iteration. An inspection of the bounds in \thref{iterationlemma} shows that this does give an iteration which would lead to double exponential bounds in dimension, which is problematic for us. The second problem is that when we Pigeonhole $h$ in one of the values of the expression $P(h)$, we will have a density decrease \emph{at best} of $\delta \mapsto \delta^2$. This would also iterate to bounds double exponential in dimension. \\\\
To overcome the first issue, we use the geometry of numbers. Using Minkowski's first theorem, one can show that the nonzero vector $\eta$ must lie in a small tube pointed in the direction of $a$. This can be used to guarantee that upon iteration, the size of $a$ doesn't actually get bigger, but rather gets smaller. This is because if $\eta$ lies in a tube of width $\epsilon$ in the direction of $a$, we write
$$\eta = ta + O_{\le 1}(\epsilon)$$
(where $O_{\le 1}$ denotes that the implicit constant is $\le 1$) then
$$\eta_i = ta_i + O_{\le 1}(\epsilon).$$
Thus
$$\eta_1 a_i - a_1 \eta_i = O_{\le 2}(|a|\epsilon).$$
Incidentally, something similar can also be shown using Green and Tao's method in \cite[Proposition 5.3]{GT12} of Fourier analysis. A use of the Fourier uncertainty principle on something similar to the function $F(x)$ constructed in the proof of \cite[Proposition 5.3]{GT12} also shows that the nonzero vector $\eta$ must lie in a tube in the direction of $a$, though due to irregularities with the Fourier transform, one cannot guarantee that the tube has as small width as what we guarantee using Minkowski's first theorem. This would lead to a slight increase in the size of $a$, rather than a decrease that Minkowski's theorem guarantees. Fortunately, the increase in $a$ is small enough for the argument to still carry through. \\\\
To overcome the second issue, we observe that for each $j$, $\{h: P(h) = j \}$ behaves like a Bohr set in that it has ``bounded Fourier complexity." If we Pigeonhole in $h$, we lose this information. In order to utilize this information, we convert the problem from \thref{refinedappendix1} to \thref{refinedappendix2}. This will allow us to use via \thref{bilinearfouriercomplexity} the fact that $e(a_1nP(h)) = e(\{a_1n\}P(h))$ (the equality follows since $P(h)$ takes integer values) has bounded Fourier complexity. This will allow us to transfer the losses obtained from $\delta$ to losses in $K$. The variable $K$ will turn out to iterate far better than $\delta$ does, because we can ``remember" bracket polynomials of the form $\{a_1n\}P(h)$ from previous steps of the iteration (so on step $j$, we will deal $j$ many bracket polynomials that resemble $\{a_1n\}P_1(h) + \cdots + \{a_j n\}P_j(h)$), so the variable $K$ doesn't actually increase iteratively but increases like $K_j = (K_1/(\delta_j/3^dM2d))^{-O(jd)^3}$ where $K_j$ is the value of $K$ in the $j$th step of the iteration. We are now ready for the deduction of the refined bracket polynomial lemma. 
\subsection{Deduction of \thref{refinedappendix2}:}
We state the lemma which we iterate with.
\begin{lemma}\thlabel{iterationlemma}
Let $a, \alpha \in \mathbb{R}^d$ and $N, N_1, K, \delta > 0$ with $N_1$ a prime integer and $\alpha$ having denominator $N_1$, $\delta < 1/10$. Suppose there are at least $\delta N_1$ many $h \in [N_1]$ such that
$$\|\beta + a \cdot \{\alpha h\}\|_{\mathbb{R}/\mathbb{Z}} \le \frac{K}{N}.$$
Then either $N \ll (\delta/3^ddM)^{-O(1)}$ or $N_1 \ll (\delta/3^ddM)^{-O(1)}$ or $(\delta/3^dMd)^{4d}\|a\|_\infty \le K/N$ or there exists an integer vector $\eta$ of size at most $(\delta/3^dMd)^{-O(d)}$ in a $(\delta/3^dMd)$-tube in the direction of $a$ such that $\|\eta \cdot \alpha\|_{\mathbb{R}/\mathbb{Z}} = 0$.
\end{lemma}
\begin{remark}
The fact that $\eta$ lies in a small tube in the direction of $a$ is crucial for our iteration and is the chief difference between this lemma and the bracket polynomial lemma \cite[Proposition 5.3]{GT12} of Green and Tao.
\end{remark}

\begin{proof}
Suppose
$$\|\beta + a \cdot \{\alpha h\}\|_{\mathbb{R}/\mathbb{Z}} < \frac{K}{N}.$$
If $(\delta/3^dMd)^{4d} \|a\|_{\infty} \ge K/N$, then
$$\|\beta + a \cdot \{\alpha h\}\|_{\mathbb{R}/\mathbb{Z}} < \|a\|_{\infty} (\delta/3^dMd)^{4d}.$$
By the Pigeonhole principle, there exists $\delta/dM$ proportion of $h$ and some integer $\ell$ such that
$$\beta + a \cdot \{\alpha h\} = \ell + O_{< 1} (\|a\|_{\infty} (\delta/3^dMd)^{4d})$$
and by Pigeonholing in a sign pattern, we see that there exists a $\delta/3^ddM$ proportion of $h$ such that
$$\beta + a \cdot \{\alpha h\} = \ell + O_{< 1} (\|a\|_{\infty} (\delta/3^dMd)^{4d})$$
and $\{\alpha h\}$ has the same sign. Taking the difference of any two $h$'s, we see that
\begin{equation}
|a \cdot \{\alpha h\}| < 2\|a\|_\infty (\delta/3^dMd)^{4d}. \tag{1}
\end{equation}
Let $T$ be the tube of width $(\delta/3^dMd)^2$ and length $2^{2d + 1}d^d\sqrt{d\pi} C (\delta/3^dM2d)^{-2d}$ in the direction of $a$. By Minkowski's first theorem, there exists a lattice point $\eta$ inside $T$. Multiplying (1) by $t$ so that $ta$ is $(\delta/3^dMd)^2$-close to $\eta$, we see that $t \le 2^{d + 1}(\delta/3^dMd)^{-2d}d/\|a\|_\infty$ so
$$|\eta \cdot \{\alpha h\}|= O_{\le 2} (2^{2d + 1}d^d\sqrt{d\pi}C(\delta/3^dMd)^{2d} + (\delta/3^dMd)^2)$$
which if $(\delta/3^dMd) < 2(2^{2d + 1}d^d\sqrt{d\pi}C(\delta/3^dMd)^{2d} + (\delta/3^dMd)^2)$, which happens if $\delta \ll_C 1$, then this implies that $\|\eta \cdot \alpha\|_{\mathbb{R}/\mathbb{Z}} = 0$.
\end{proof}

\subsection{Iterating \thref{iterationlemma}}
To iterate this, we define some notation. On step $j$ of the iteration, we will work with the following setup:
\begin{itemize}
\item We have
$$|\mathbb{E}_{n \in [N]} e((\tilde{a}_j) n \cdot \{\alpha h\} + P_j(n, h))| \ge K^{-1}$$
for $\delta_j N'$ many $h \in \mathbb{Z}/N'\mathbb{Z}$ where $P_j(n, h)$ denotes a bracket polynomial in $n$ and $h$ with at most $O(jd)$ many lower order terms each of the form
$$a\{\alpha h + \beta\}\{\alpha'n + \beta'\}.$$
\item There will also exist $0', \dots, (j - 1)'$, distinct integers in $\{1, \dots, d\}$ representing the components of a vector in $\mathbb{R}^d$.
\item We will require $|\tilde{a}_j| \le M_j$ and having components $0', 1', \dots, (j - 1)'$ equal to zero.
\item There will exist linearly independent vectors in $\mathbb{Z}^d$ $\eta^1, \dots, \eta^j$, with $\eta^k$ having size at most $(\delta_k/3^ddM_j)^{-O(d)}$ and $\eta^k$ having zero components $0', \dots, (k - 1)'$ equal to zero with the property that $\|\eta^k \cdot \alpha\|_{\mathbb{R}/\mathbb{Z}} = 0$.
\item The base case of the iteration will have $j = 0$ with $(\tilde{a}_0) = a$, $P_0 = 0$, $\delta_0 = \delta$, and $M_0 = M$. 
\end{itemize}
Thus, in order to carry out the iteration given the data in step $j$, we must construct 
$$P_{j + 1}, M_{j + 1}, \delta_{j + 1}, \eta^{j + 1}, (\tilde{a}_{j + 1}), \text{ and } j'.$$
We will also be considering additional parameters $K_j$ and $q_{j}$, which we will construct in the argument. \\\\
\textbf{Step 1: Constructing $\delta_{j + 1}$ and $K_j$.} Applying \thref{bilinearfouriercomplexity} with $\delta = \min(\delta_j^2/4, K^{-2}/4) = K^{-2}/4$ to $P_{j}(n, h)$ and denoting
$$f(n, h) = e((\tilde{a})_{j} n \cdot \{\alpha h\} + P_{j}(n, h))$$
we may find $g(n, h)$ of Fourier complexity at most $(K/(\delta/3^dM2d))^{O(jd)^3}$ such that
$$\mathbb{E}_{h} \mathbb{E}_n |f(n, h) - e((\tilde{a})_{j} n \cdot \{\alpha h\})g(n, h)| \le K^{-2}/4.$$
Thus, by Markov's inequality, for at least $\delta_jN_j/2$ many $h$, we have
$$|\mathbb{E}_n f(n, h) - e((\tilde{a})_{j} n \cdot \{\alpha h\})g(n, h)| \le K^{-1}/2.$$
Thus, by the triangle inequality, and using the fact that $g$ has bounded Fourier complexity, we obtain phases $\xi$ and $\zeta$ such that
$$|\mathbb{E}_n e((\tilde{a})_{j} n \cdot \{\alpha h\} + \xi n + \zeta h)| \ge (K/(\delta_j/3^dM_j2d))^{-O(jd)^3}.$$
for $\frac{\delta_j}{2} N'$ many $h$. \textbf{We let $\delta_{j + 1} = \frac{\delta_j}{2}$}. Thus, it follows that
$$\|(\tilde{a})_{j} \cdot \{\alpha h\} + \xi\|_{\mathbb{R}/\mathbb{Z}} \le \frac{(K/(\delta_j/3^dM_j2d))^{-O(jd)^3}}{N} := \frac{K_j}{N}.$$
\textbf{Step 2: Constructing $\eta^{j+1}$.} We now apply \thref{iterationlemma} with $K = K_j$, $N = N$, and $N_1 = N'$. We obtain either
\begin{itemize}
\item two degenerate cases of when $N$, or $N' \ll (\delta_j/3^ddM_j)^{-O(1)}$;
\item or $(\delta_j/3^dM_j2d)^{4d}\|(\tilde{a})_{j}\|_\infty \le K_j/N$, in which case we end the iteration (see subsection below);
\item or some integer vector $\eta^{j + 1}$ inside the tube of width $(\delta_j/3^dM_j2d)$ around $(\tilde{a})_{j }$ of size at most $(\delta_j/3^dM_j2d)^{-O(d)}$ and with zero components in its $0', \dots, (j - 1)'$th coordinates and such that 
$$\|\eta^{j + 1} \cdot \alpha\|_{\mathbb{R}/\mathbb{Z}} = 0.$$
\end{itemize}
\textbf{Step 3: Constructing $j'$, $\tilde{a}_{j + 1}$, $q_{j + 1}$, and $P_{j + 1}$.} Pick some $j'$ so that the $j'$th component of $\eta^{j + 1}$, which we denote by $\eta^{j + 1}_{j'}$ is nonzero. Since modular inverses exist for nonzero elements modulo $\mathbb{Z}/N'\mathbb{Z}$, we may replace $h$ by $\eta^{j + 1}_{j'} k$. Since this is rather cumbersome, we write $\eta^{j + 1}_{j'}$ as $q_{j + 1}$. Thus, writing
$$\{\alpha q_{j + 1}k \} = q_{j + 1}\{\alpha k\} - (q_{j + 1}\{\alpha k\} - \{\alpha q_{j + 1}k\})$$
we see that denoting the $i \neq j'$ component of $(\tilde{a})_{j + 1}$ to be $(\tilde{a}_i)_{j} q_{j + 1} - (\tilde{a}_{j'})_{j}\eta^{j + 1}_{i}$ and the $j'$th component equal to zero,
$$(\tilde{a})_{j} n\cdot \{\alpha q_{j + 1}k\} \equiv (\tilde{a})_{j + 1} \cdot \{\alpha k\} + $$
$$\{(\tilde{a})_{j} n\} \cdot (-q_{j + 1}\{\alpha k\} + \{\alpha q_{j + 1}k\}) +  \{(\tilde{a}_{j'})_{j} n\}\left(q_{j + 1}\{\alpha_{j'} k\} + \sum_{i \neq j'} \eta^{j + 1}_i \{\alpha_i k\}\right) \pmod{1}. $$
Thus,
$$(\tilde{a})_{j} n \cdot \{\alpha h\} + P_{j}(n, h) \equiv (\tilde{a})_{j + 1} n \cdot \{\alpha k\}  + P_{j + 1}(n, k) \pmod{1}$$
where
$$P_{j + 1}(n, k) = P_{j}(n, q_{j + 1}k) + $$
$$\{(\tilde{a})_{j} n\} \cdot (-q_{j + 1}\{\alpha k\} + \{\alpha q_{j + 1}k \}) +  \{(\tilde{a}_{j'})_{j} n\}\left(q_{j + 1}\{\alpha_{j'} k\} + \sum_{i \neq j'} \eta^{j + 1}_i \{\alpha_i k\}\right).$$
\textbf{Step 4: Constructing $M_{j + 1}$.} We lastly construct $M_{j + 1}$. We claim that one can take $M_{j + 1} = M_j$. Since $\eta^{j + 1}$ lies in a $(\delta_j/3^dM_j2d)$-tube around $(\tilde{a})_{j}$, there exists some real number $t$ such that
$$\eta^{j + 1} = t(\tilde{a})_{j} + O_{\le 1}(\delta_j/3^dM_j2d)$$
and so for each $i$,
$$\eta_{j'}^{j + 1} = t(\tilde{a}_{j'})_{j} + O_{\le 1}(\delta_j/3^dM_j2d)$$
$$\eta_i^{j + 1} = t(\tilde{a}_i)_{j} + O_{\le 1}(\delta_j/3^dM_j2d).$$
Multiplying the first equation by $(\tilde{a}_i)_{j}$ and the second equation by $(\tilde{a}_{j'})_{j}$ and subtracting the two, we obtain
$$|\eta_{j'}^{j + 1}(\tilde{a}_i)_{j} - \eta_i^{j + 1} (\tilde{a}_{j'})_{j}| \le 2|a|\delta_j/3^dM_j2d \le 1 \le M_j.$$
\textbf{Step 5: Gathering up all the relevant data.}
We thus have 
$$(\delta_{j + 1}, M_{j + 1}, K_j) = (\delta_j/2, M_j, (K/(\delta_j/3^dM_j2d))^{-O(jd)^3}).$$
Under iteration in $d$, we obtain 
\begin{equation}\label{iterationbounds}
\delta_{j} \ge \delta/2^d, M_j = M,\text{ and } K_j \le (K/(\delta_j/3^dM_j2d))^{-O(d^6)}.    
\end{equation}

\subsection{End of the iteration and completing the proof of \thref{refinedappendix2}}
Suppose we ended the iteration on step $m$. This is the case where $(\delta_m/3^dM_m2d)^{4d}\|(\tilde{a})_{m}\|_\infty \le K_m/N$. We note that $(\tilde{a})_{m}$ has each of its components linear in $a$ with coefficients in $\mathbb{Z}$. Let $w_1^m \cdot a, \dots, w_j, \dots, w_{d}^m \cdot a$ be its components where we omit index $j$ if $j \in \{0', \dots, (m - 1)'\}$. We can show by induction on $k \le m$ that 
\begin{itemize}
\item[1.] We can take for each $i \not\in \{0', \dots, k'\}$ that $w^{k + 1}_i = q_{k + 1}w_i^{k} - \eta^{k + 1}_i w_{k'}^k$ and;
\item[2.] if $\eta^j \cdot x = 0$ for all $j = 1, \dots, k$, then one can write $q_1q_2 \cdots q_k x = \sum_{i \not\in \{0', \dots, (k - 1)'\}} w_i^k x_i.$
\end{itemize}
For the base case, one can simply take $w_i^0 = e_i$, which is the unit vector with $0$'s in all its components and a $1$ in its $i$th component. The second point of the statement obviously holds. For the first point, note that $(\tilde{a}_i)_1 = q_1a_i - \eta^1_i a_{0'}$. Noting that $a_i = a \cdot e_i$ and $a_{0'} = a \cdot e_{0'}$, the first point of the claimed statement follows. Supposing the induction hypothesis holds for $k \ge 0$, we see that
$$q_1 \cdots q_k x = \sum_{i \not\in \{0', \dots, (k - 1)'\}} w_i^k x_i.$$
We see that since $\eta^{k + 1} \cdot x = 0$ (and $\eta^{k + 1}$ has no coefficients in $\{0', \dots, (k - 1)'\}$) that we may write $q_{k + 1} x_{k'} = \sum_{i \not\in \{0', \dots, k'\}} -\eta^{k + 1}_i x_i$, we see that we may write
$$q_1 \cdots q_{k + 1} x = \sum_{i \not\in \{0', \dots, k'\}} (q_{k + 1}w_i^{k} - \eta^{k + 1}_i w_{k'}^k)x_i.$$
We in addition see that the $i$th component of $(\tilde{a})_k$ is $q_{k + 1}(\tilde{a}_i)_{k - 1} - \eta^{k + 1}_i (\tilde{a}_{k'})_{k - 1}$ and since $(\tilde{a}_i)_k = w^k_i \cdot a$, we see that we can take $w^{k + 1}_i = q_{k + 1}w_i^{k} - \eta^{k + 1}_i w_{k'}^k$. The induction argument follows. \\\\
Hence, the subspace of annihilators of $\eta^1, \dots, \eta^{m}$ is lies in the span of $\{w_j^m\}_{j \not\in \{0', \dots, (m - 1)'\}}$. This tells us by dimension counting reasons that the annihilators of $\eta^1, \dots, \eta^m$ coincide with the subspace spanned by $\{w_j^m\}_{j \not\in \{0', \dots, (m - 1)'\}}$ and shows that $\{w_j^m\}_{j \not\in \{0', \dots, (m - 1)'\}}$ are linearly independent and are annihilated by $\eta^1, \dots, \eta^m$. The inductive argument also shows that $|w_i^k| \le  (\delta_{j - 1}/3^dM_k2d)^{-O(dk)}$. Finally, we note that $(\delta_m/3^dM_m2d)^{4d}\|(\tilde{a})_{m - 1}\|_\infty \le K_m/N$ implies that for all $j$,
$$|w_j^m \cdot a| \le \frac{(2K \delta^{-1}M_m3^d)^{O(m^4d^4)}}{N}.$$
Letting the $w_i$'s be these $w_i^m$'s, and noting from (\ref{iterationbounds}) that the bounds on $\delta_j$, $M_j$ and $K_j$ are well-controlled \thref{refinedappendix2} follows.
\begin{remark}
One can make the further reduction of taking $a$ to have prime denominator in \thref{refineddenominator}. This would mean that we are in the substantially simpler case of \thref{bilinearfouriercomplexity} where all the phases have prime denominator. We chose not to make this reduction since we wished to present \thref{bilinearfouriercomplexity} in full generality.
\end{remark}
\bibliographystyle{amsplain0.bst}
\bibliography{main.bib}

\end{document}